\documentclass[12pt]{article}
\usepackage[usenames,dvipsnames]{color}
\usepackage{xr,graphicx,amsthm,amsmath,amssymb,xargs,srcltx,bbm,todonotes,stmaryrd,framed,a4wide,accents,sectsty,marvosym,mathrsfs}
\usepackage[shortlabels]{enumitem}
\usepackage{framed}
\usepackage{pdfsync}
\usepackage{extarrows}

\usepackage[colorlinks]{hyperref}

\newcommandx{\conditiondhsumpaper}[2]{\ensuremath{\mathcal{S}(#1,#2)}} 
\newcommandx{\conditiondhsumstrongerpaper}[2]{\ensuremath{\mathcal{S}_+(#1,#2)}} 
\newcommandx{\conditiondhsumstrongergammapaper}[3]{\ensuremath{\mathcal{S}^{(#3)}(#1,#2)}} \newcommandx{\conditiondhsumsdouble}[4]{\ensuremath{\mathcal{S}_{#3,#4}(#1,#2)}}

\newcommand{\DB}{{\rm{DB}}}
\newcommand{\SB}{{\rm{SB}}}
\newcommand{\IC}{\mathcal{IC}}
\newcommand{\BC}{\mathcal{BC}}
\newcommand{\RC}{\mathcal{R}}

\newcommand{\RIC}{\mathcal{R}^{\rm IC}}
\newcommand{\RBC}{\mathcal{R}^{\rm BC}}
\newcommand{\RNC}{\mathcal{R} ^ {\rm NC} }

 \newcommand{\rom}[1]{
  \textup{\lowercase\expandafter{\romannumeral#1}} }

\newcommand{\Rom}[1]{
  \textup{\uppercase\expandafter{\romannumeral#1}} }

\subsubsectionfont{\itshape}

\usepackage{cleveref}

\newtheorem{theorem}{Theorem}[subsection]
\newtheorem{proposition}[theorem]{Proposition}
\newtheorem{lemma}[theorem]{Lemma}
\newtheorem{corollary}[theorem]{Corollary}
\newtheorem{definition}[theorem]{Definition}
\newtheorem{hypothesis}[theorem]{Assumption}
\newtheorem{example}[theorem]{Example}
\newtheorem{remark}{Remark}[section]

\crefname{hypothesis}{assumption}{assumptions}

\crefname{lemma}{Lemma}{Lemmas}
\crefname{theorem}{Theorem}{Theorems}
\crefname{proposition}{Proposition}{Propositions}
\crefname{exercise}{Exercise}{Exercises}
\crefname{corollary}{Corollary}{Corollaries}

\numberwithin{equation}{section}

\newcommandx\TEP[1][1=]{\mathbb{G}^{#1}} 

\newcommand\tepcluster{{\mathbb{G}}}
\newcommandx\tedcluster[2][1=n,2=\dhinterseq]{{\widetilde{\boldsymbol\nu}}^*_{#1,#2}} 
\newcommandx\tedclustersl[2][1=n,2=\dhinterseq]{{\widetilde{\boldsymbol\mu}}^*_{#1,#2}} 
\newcommandx\tedclusterboot[3][1=n,2=\dhinterseq,3=\xi]{{\widetilde{\boldsymbol\nu}}^*_{#1,#2,#3}}
\newcommandx\tedclusterslboot[3][1=n,2=\dhinterseq,3=\xi]{{\widetilde{\boldsymbol\mu}}^*_{#1,#2,#3}} \newcommandx\tedclusterefron[3][1=n,2=\dhinterseq,3={\rm boot}]{{\widetilde{\boldsymbol\nu}}^*_{#1,#2,#3}}
\newcommandx\tedclusterslefron[3][1=n,2=\dhinterseq,3={\rm boot}]{{\widetilde{\boldsymbol\mu}}^*_{#1,#2,#3}}
\newcommandx\tedclusterindep[2][1=n,2=\dhinterseq]{{\widetilde{\boldsymbol\nu}}^\indep_{#1,#2}} 
\newcommandx\tedclustermdep[2][1=n,2=\dhinterseq]{{\widetilde{\boldsymbol\nu}}^{*(m)}_{#1,#2}} 

\newcommandx\tedclusterrandom[2][1=n,2=\dhinterseq]{{\widehat{\boldsymbol\nu}}^*_{#1,#2}}

\newcommand\constant{\mathrm{cst}}

\newcommandx\envelope[1][1={\mathbf H}]{{\mathbf {#1}}}

\newcommand\canditheta{{\vartheta}}
\newcommandx\stoploss[1][1={\rm stoploss}]{\theta_{#1}}
\newcommandx\stoplossest[1][1={{\rm stoploss},n}]{\widetilde\theta_{#1}}
\newcommandx\stoplossestk[1][1={{\rm stoploss},n,\statinterseq}]{\widehat\theta_{#1}}
\newcommandx\largedeviation[1][1={\rm largedev}]{\theta_{#1}}
\newcommandx\largedeviationest[1][1={{\rm largedev},n}]{\widetilde\theta_{#1}}
\newcommandx\largedeviationestk[1][1={{\rm largedev},n,\statinterseq}]{\widehat\theta_{#1}}
\newcommandx\ruin[1][1={\rm ruin}]{\theta_{#1}}
\newcommandx\ruinest[1][1={{\rm ruin},n}]{\widetilde\theta_{#1}}
\newcommandx\ruinestk[1][1={{\rm ruin},n,\statinterseq}]{\widehat\theta_{#1}}

\newcommandx\clfunc[1][1=h]{#1} 

\newcommandx\anticl[2][2=\dhinterseq]{\conditionS(#1,#2)}

\newcommandx\anticlpsi[3][1=\tepseq,2=\dhinterseq,3=\psi]{\conditionS(#1,#2,#3)}

\newcommandx{\norm}[2][1=]{\left|#2\right|_{#1}} 
\newcommandx{\lpnorm}[3][1=,3=]{\|#2\|_{#1}^{#3}}

\newcommandx{\matrixnorm}[2][2=]{\left\|#1\right\|_{#2}} 
\newcommandx{\matrixnormseries}[3][2=,3=]{\left\|#1\right\|_{#2,#3}} 
\newcommandx{\linftynorm}[2][1=]{\|#2\|_{#1}} 
\newcommandx{\anynorm}[2][2=]{\left|#1\right|_{#2}} 

\newcommandx{\supnormclass}[3][3=]{\left\|#1\right\|_{#2}^{#3}} 

\newcommandx{\sphere}[2][1=]{\mathbb{S}_{#1}^{#2}}
\newcommandx{\oball}[3][3=]{B_{#3}(#1,#2)} 
\newcommandx{\cball}[3][3=]{\overline{B}_{#3}(#1,#2)} 
\newcommandx{\coball}[3][3=]{B_{#3}^c(#1,#2)} 
\newcommandx{\ccball}[3][3=]{\overline{B}_{#3}^c(#1,#2)} 
\newcommandx\cone[1][1=C]{\mathcal{#1}}

\newcommandx\conej[1][1=\mathbf{j}]{\mathcal{C}_{#1}}   

\newcommandx{\coneindex}[2][2=\mathcal{C}]{#1_#2}

\newcommand\Nset{\mathbb{N}}
\newcommand\Zset{\mathbb{Z}}

\newcommand\Rset{\mathbb{R}}

\newcommandx\borel[1][1=\csms]{\mathcal{B}(#1)}   

\newcommand{\bszero}{{\boldsymbol0}}

\newcommand\indep{\dag}  

\newcommand\bsx{\boldsymbol{x}}
\newcommand\bsX{\boldsymbol{X}}

\newcommand\bsY{\boldsymbol{Y}}

\newcommand\bsZ{\boldsymbol{Z}}

\newcommand\bsTheta{\boldsymbol{\Theta}}

\newcommand\bsnu{\boldsymbol{\nu}}

\newcommand\tailmeasure{\bsnu} 
\newcommand\tailmeasurestar{{\bsnu}^*} 

\newcommand\exc{\mce}  
\newcommand\clusterlength{\mcl}
\newcommand{\extremalindexfunc}{\mct}
\newcommand\anchor{\mca} 

\newcommandx{\numult}[2][2=]{\bsnu_{\boldsymbol{#1}_{#2}}} 

\newcommand\backest\Upsilon

\newcommandx{\nualphap}[2][1=\alpha,2=p]{\nu_{#1,#2}}  

\newcommandx\numultcondi[2][2=]{\boldsymbol{\nu}_{\boldsymbol{#1}_{#2}}}

\newcommandx\chain[1][1=Y]{\mathbb{#1}}

\newcommandx{\scalingseq}[1][1=n]{c_{#1}}  
\newcommandx{\scalingfunction}[1][1=]{c#1} 

 \newcommandx{\absquantileseq}[2][1=]{a_{#1#2}}

\newcommandx{\dhinterseq}[1][1=n]{r_{#1}}  
\newcommandx{\dhinterseqsmall}[1][1=n]{\ell_{#1}}  
\newcommandx{\dinterseq}[1][1=n]{r_{#1}}  
\newcommandx{\tepseq}[1][1=n]{u_{#1}}  

\newcommandx{\interseq}[1][1=n]{k_{#1}}  

\newcommandx{\scalingseqcone}[1][1=n]{c_{#1}} 

\newcommandx{\scalingseqhidden}[1][1=n]{\tilde{c}_{#1}}
\newcommandx{\scalingfunctionhidden}[1][1=]{\tilde{c}{#1}}
\newcommandx{\scalingseqcev}[1][1=n]{c^*_{#1}}

\newcommand\conditionS{\ensuremath{\mathcal{S}}} 
\newcommandx{\conditiondh}[2][1=\dhinterseq,2=\scalingseq]{\ensuremath{\mathcal{A}\mathcal{C}(#1,#2)}} 
\newcommandx{\conditionANSJB}[1][1=\scalingseq]{\mathrm{ANSJB}(\dhinterseq,#1)}
\newcommandx{\conditiondhsum}[1][1=\scalingseq]{\ensuremath{\mathcal{S}(\dhinterseq,#1)}} 
\newcommandx{\conditiondhsumW}[1][1=\scalingseq]{\ensuremath{\mathcal{SW}(\dhinterseq,#1)}} 

\newcommand\conditionR{\ensuremath{\mathcal{R}}} 

\newcommand\convdistr{\stackrel{\mbox{\tiny\rm d}}{\longrightarrow}} 

\newcommand\convprob{\stackrel{\tiny \mathbb{P}}{\longrightarrow}}

\newcommandx\prohodistsym[1][1=]{\rho_{#1}}
\newcommandx\prohodist[3][3=]{\rho_{#3}(#1,#2)}

\newcommand\rmd{\mathrm{d}} 

\newcommand\esp{\mathbb E}
\newcommand\pr{\mathbb P}
\newcommand\var{\mathrm{Var}}
\newcommand\cov{\mathrm{Cov}}

\newcommandx{\autocov}[1][1=]{\gamma_{#1}}

\newcommandx{\cdfnorm}[1][1=\bsX]{H_{#1}}
\newcommandx{\tailcdfnorm}[1][1=\bsX]{\overline{H}_{#1}}

\newcommand\ind[1]{\mathbbm{1}{\left\{#1\right\}}}

\newcommand\1[1]{\mathbbm{1}_{#1}}

\newcommand\mca{\mathcal A}

\newcommand\mce{\mathcal E}

\newcommand\mch{\mathcal H}

\newcommand\mcl{\mathcal L}

\newcommand\mct{\mathcal T}

\newcommandx\test[2][1=X]{{#1}_{#2}}
\newcommandx\orderstat[3][1=X]{{#1}_{(#2:#3)}}
\newcommand\statinterseq{k}

\newcommandx{\sequence}[3][2=\Zset,3=j]{\{#1_{#3},#3\in#2\}}
\newcommandx{\sequenceshort}[2][2=j]{\{#1_#2\}}
\newcommandx\sequ[3][2=j,3=\mathbb{Z}]{\{#1_#2,#2\in#3\}}
\newcommandx\sequnorm[3][3=j,2=\mathbb{Z}]{\{\norm{#1_#3},#3\in#2\}}
\newcommandx\sequnormq[4][2=,4=j,3=\mathbb{Z}]{\{\norm{#1_#4}^{#2},#4\in#3\}}

\newcommandx\uncompactd[2][1=d]{(\overline{\Rset}^{#1})^{#2}\setminus\{\boldsymbol0\}}
\newcommandx{\barrsetproduct}[2][1=d]{(\overline{\Rset}^{#1})^{#2}}
\newcommandx{\rsetproduct}[2][1=d]{(\Rset^{#1})^{#2}}

\newcommand{\metricspace}{\csms}

\newcommandx\csms[1][1=E]{\mathsf{#1}}   
\newcommandx\borelcsms[1][1=E]{\mathcal{#1}}   
\newcommandx\mplusx[1][1=]{\mathcal{M}#1}  
\newcommandx\mplusxp[1][1=]{\mathcal{N}{#1}} 
\newcommandx\mplusxpb[1][1=\borelcsms]{\mathcal{N}_{pb}({#1})}  
\newcommandx\mplusxpone[1][1=\borelcsms]{\mathcal{N}_{p1}({#1})}  
\newcommandx\mplusxpeps[1][1=\borelcsms]{\mathcal{N}_{p\epsilon}({#1})}  
\newcommandx\mplusxf[1][1=\borelcsms]{\mathcal{M}_f({#1})}
\newcommandx\mplusxpf[1][1=\borelcsms]{\mathcal{N}_{pf}({#1})} 
\newcommandx\mplusxps[1][1=\borelcsms]{\mathcal{N}_{ps}({#1})} 
\newcommandx\mplusxpS[1][1=\borelcsms]{\mathcal{N}_{pS}({#1})} 
\newcommandx\mplusxpsc[1][1=\borelcsms]{\mathcal{N}_{psc}({#1})} 

\newcommand\distance{\mathrm{d}}  
\newcommandx\metric[1][1=\metricspace]{\distance_{#1}} 
\newcommandx\metricmcg{\rho} 

\newcommandx\bracknum[3][2=\mch]{N_{[\,]}(#1,#2,#3)} 
\newcommandx\bracknumarray[2][2=\mch]{N_{[\,]}(#1,#2,L^2_n)} 
\newcommandx\entropynum[3][3=\mch]{N(#1,#3,#2)} 

\newcommandx\process[1][1=X]{\mathbb{#1}}

\newcommandx\hillest[3][1=n,3=]{\widehat{\gamma}_{#1,#2}^{#3}}
\newcommandx\hillmoment[2][1=n]{\widehat{\gamma}_{#1,#2}^{(M)}}

\newcommand\lzero{\ell_0}

\newcommandx\lalpha[1][1=\alpha]{\ell_{#1}}

\newcommand\bernoulli{\mathrm{b}}

\newcommand\iid{i.i.d.}
\newcommand\withoutlog{without loss of generality}

\newcommand\wrt{with respect to}

\newcommand\rhs{right-hand side}

\newcommand\nonnegative{non-negative}

\begin{document}
\title{Asymptotic expansions for blocks estimators: PoT framework}

\author{Zaoli Chen\thanks{University of Ottawa}\ \  and Rafa{\l} Kulik\thanks{Corresponding author: University of Ottawa, rkulik@uottawa.ca}}

\date{\today}
\maketitle

\begin{abstract}We consider disjoint and sliding blocks estimators of cluster indices for multivariate, regularly varying time series in the Peak-over-Threshold framework. 
We aim to provide a complete description of the limiting behaviour of these estimators. This is achieved by a precise expansion for the difference between the sliding and the disjoint blocks statistics.  The rates in the expansion stem from \textit{internal clusters} and \textit{boundary clusters}. To obtain these rates we need to extend the existing results on vague convergence of cluster measures. We reveal dichotomous behaviour between \textit{small blocks} and \textit{large blocks} scenario. 

\end{abstract}
\setcounter{tocdepth}{3}
\tableofcontents

\section{Introduction}\label{sec:intro}
Consider
a stationary, regularly varying $\Rset^d$-valued time series $\bsX=\sequence{\bsX}$. We are interested in estimating cluster indices that describe its extremal behaviour. Informally speaking, a cluster is a triangular array $(\bsX_1/\tepseq,\ldots,\bsX_{\dhinterseq}/\tepseq)$ with $\dhinterseq,\tepseq\to\infty$ that converges in distribution in a certain sense. Cluster indices are obtained by applying the appropriate  functional $H$ to the cluster.
The functionals are defined on $(\Rset^d)^\Zset$ and are such that their values do not depend on coordinates whose entries are small.
More precisely, denote $\bsX_{i,j}=(\bsX_i,\ldots, \bsX_j)\in (\Rset^d)^{(j-i+1)}$.
Then, we identify $H(\bsX_{i,j})$ with
$H((\bszero,\bsX_{i,j},\bszero))$, where $\bszero\in (\Rset^d)^\Zset$ is the zero sequence. Such a functional $H$ will be called a \textit{cluster functional}.

Let $w_n:=\pr(\norm{\bsX_0}>\tepseq)$. 
Given a cluster functional $H$ on $(\Rset^d)^\Zset$, we want to estimate 
\begin{align}
\tailmeasurestar(H):=  \lim_{n\to\infty}  \tailmeasurestar_{n,\dhinterseq}(H):=\lim_{n\to\infty} \frac{\esp[H(\tepseq^{-1}\bsX_{1,\dhinterseq})]}{\dhinterseq w_n}\;. \label{eq:thelimitwhichisnolongercalledbH}
\end{align} 
To guarantee existence of the limit we require additional anticlustering assumptions on $\bsX$. The particular cluster quantity of interest is the (candidate) extremal index obtained with $H(\bsx)=\ind{\bsx^\ast>1}$, $\bsx=\{\bsx_j,j\in\Zset\}\in(\Rset^d)^\Zset$.
See Chapter 6 in \cite{kulik:soulier:2020} and \Cref{sec:cluster-measure-convergence} below. 

Several methods of estimation of the limit $\tailmeasurestar(H)$ in \eqref{eq:thelimitwhichisnolongercalledbH} may be
employed. The natural one is to consider a
statistics based on disjoint blocks of size $\dhinterseq$, 
\begin{align}\label{eq:blocktype}
 \tedcluster(H):= \frac{1}{nw_n}   \sum_{i=1}^{m_n}
  H(\tepseq^{-1}\bsX_{(i-1)\dhinterseq+1,i\dhinterseq}) \;,
\end{align}
where $m_n=[n/\dhinterseq]$. 
Although some special cases were considered (e.g. the extremal index in \cite{hsing:1991} and \cite{smith:weissman:1994}), the general theory was developed in the seminal paper \cite{drees:rootzen:2010}. 
The appropriately scaled and centered statistics is asymptotically normal with the limiting variance given by $\tailmeasurestar(H^2)$. 
A summary of the theory for the disjoint blocks statistics and the corresponding data based estimators (where the threshold $\tepseq$ is replaced with the appropriate intermediate order statistics) can be found in \cite[Chapter 10]{kulik:soulier:2020}. 
Some recent developments in the context of $\ell^p$-blocks can be found in \cite{buritica:mikosch:wintenberger:2023}.

Alternatively, we can consider the sliding blocks statistics
\begin{align}\label{eq:sliding-block-estimator-nonfeasible-1}
\tedclustersl(H):=\frac{1}{n\dhinterseq w_n}\sum_{i=1}^{n-\dhinterseq+1}H\left(\tepseq^{-1}\bsX_{i,i+\dhinterseq-1}\right)\;.
\end{align}
This and the corresponding data based estimators have been studied for some specific functionals $H$, however there was no unified theory available. Recently, \cite{drees:neblung:2021} used the framework of \cite{drees:rootzen:2010} and showed that the limiting variance of the sliding blocks estimator never exceeds that of the disjoint blocks one. In case of the extremal index, both variances are equal. In \cite{cissokho:kulik:2021}, building upon \cite[Chapter 10]{kulik:soulier:2020}, it has been proven that in case of regularly varying time series the asymptotic behaviour of sliding and disjoint blocks estimators is the same. We note in passing that the same holds for so-called runs estimators, which can be viewed as a special case of the sliding blocks estimators. See \cite{cissokho:kulik:2022}.

We note that all the discussion above is valid in the Peak-over-Threshold (PoT) framework. On the basic technical level, the "PoT framework" refers to the assumption $\lim_{n\to\infty}\dhinterseq w_n=0$; see \ref{eq:rnbarFun0} below. To the contrary, in the block maxima framework it has been observed that the asymptotic variance for sliding blocks estimators is strictly smaller as compared to disjoint blocks. We refer to \cite{bucher:segers:2018sliding} and a review in \cite{bucher:zhou:2018}.  

\paragraph{The goal of the paper.}  We aim to provide a mathematical explanation for the aforementioned phenomena observed in the PoT framework. This is achieved by a precise expansion for the difference between the sliding and the disjoint blocks statistics.  The rate in the expansion will stem from \textit{internal clusters} and \textit{boundary clusters}. 

\subsection{Probabilistic tools for asymptotic expansions} 
The techniques we use in the paper stem from Chapters 6 and 10 in \cite{kulik:soulier:2020}. 

\paragraph{Internal clusters.} Intuitively, an internal cluster occurs if there is a large value in a single block, but there are no large values in adjacent blocks. 
The starting point of our analysis is thus vague convergence of clusters.  
\Cref{theo:cluster-RV} (see also Theorem 6.2.5 in \cite{kulik:soulier:2020}) establishes such convergence, where the class of test functions consists of bounded functionals that vanish around zero. The rate of convergence is 
$\dhinterseq w_n$, which is proportional to the probability of an occurrence of a large value in a single block. 

In the present context we face a challenge. Starting with a functional $H$, the internal cluster statistics involves another functional $\widetilde{H}_{\IC}$, which is induced from $H$. For example, if $H(\bsx)=\ind{\bsx^*>1}$, then $\widetilde{H}_{\IC}$ equals a cluster length (subtracted by one), the distance between the location of the last and the first large value. In fact, we shall note that the cluster length functional plays a key role in this paper. Note also that the cluster length is unbounded. Hence, at the first step of our analysis we need to extend vague convergence of clusters to unbounded functionals. Under the appropriate uniform integrability conditions we can recover \Cref{theo:cluster-RV} - cluster functionals still have the rate $\dhinterseq w_n$. The uniform integrability condition holds for tight cluster functionals (such as the cluster length) as long as a new anticlustering condition is valid. The latter in turn is related to \textit{small blocks}. On the other hand, 
in case of \textit{large blocks} or some unbounded non-tight functionals (such as locations of large values), \Cref{theo:cluster-RV} is no longer valid and cluster functionals grow at a different rate.  

\paragraph{Boundary clusters.} When proving central limit theorem (CLT) for (disjoint) blocks estimators, mixing conditions allow to treat the consecutive blocks as if they were independent. In such the case (referred below to as "piecewise stationarity") the chance of large values in two consecutive blocks is proportional to $\dhinterseq^2 w_n^2$. It turns out that this CLT-based heuristic fails in the context of the asymptotic expansion. 

Indeed, in case of \textit{small blocks} the rate for the event "large values in two consecutive blocks" is proportional to $w_n$. Intuitively, large values occur at the end of one block and at the beginning of the next block. Bounded cluster functionals grow then at the same rate $w_n$. 
On the other hand, in case of \textit{large blocks}, the blocks behave as if they were independent and the cluster functionals have a different rate. 

\paragraph{Internal vs. boundary clusters.} Both types of clusters have a different asymptotic behaviour. Also, the notion of \textit{small} and \textit{large} blocks is different for both types of clusters. 

\paragraph{Asymptotic expansions.} 
In case of small blocks the rate in the asymptotic expansion stems from both internal and boundary clusters. As a consequence, after the appropriate scaling, the difference between the disjoint and the sliding blocks statistics is of the rate $O_P(1/\dhinterseq)$. 
On the other hand, in case of large blocks, the difference is of the rate $O_P(\dhinterseq w_n)$. 

\paragraph{Weak dependence assumptions.} Most of the results in the paper are valid just under the appropriate anticlustering condition and only some of them involve mixing. This is the situation of the small blocks scenario. 

The large blocks scenario is presented in case of $\ell$-dependence. We can re-write our results in an expense of cumbersome mixing assumptions. We decided to focus on a very simple dependence structure that is sufficient to understand the difference between the small and the large blocks scenario, as well as between internal and boundary clusters. All phenomena that appear in the paper can be illustrated by a simple, 1-dependent time series!

Also, we obtain an expansion for a piecewise stationary time series. In the small blocks scenario, the result is quantitatively different as compared to the weakly dependent case. 

\paragraph{Conclusion for the asymptotic behaviour of blocks estimators.}
In the PoT framework, disjoint and sliding blocks estimators have the same asymptotic behaviour in either weakly dependent or piecewise stationary situation: the limiting variance in the Central Limit Theorem is the same. This is in contrast to the block maxima framework. 

\subsection{Structure of the paper}

\Cref{sec:prel} contains preliminaries. It introduces the tail process, the relevant class of functions (including the cluster length, which appears to be the most important functional in the context of the paper; see \Cref{sec:classes}), cluster measures and cluster indices. 
\Cref{sec:anticlustering-condition} introduces different types of anticlustering conditions. \Cref{sec:cluster-measure-convergence} recalls 
vague convergence of cluster measures. \Cref{theo:cluster-RV} is the most relevant result in this context. 

In \Cref{sec:two-types-of-clusters} we decompose the difference between disjoint and sliding blocks estimators. Internal (\Cref{sec:contribution-internal-clusters}) and boundary (\Cref{sec:contribution-boundary-clusters}) clusters appear. 

\Cref{sec:main-results} contains the main results on the asymptotics expansion for disjoint and sliding blocks statistics. 
\Cref{theo:main-theorem} deals with the small blocks scenario. We obtain a precise expansion at the rate $O_P(\dhinterseq^{-1})$. Here, both the internal and the boundary clusters contribute at the same rate. 
In \Cref{sec:piecewise} we consider a piecewise stationary case. This is motivated by the following observation. When a central limit theorem for disjoint blocks statistics is considered, a $\beta$-mixing time series has the same asymptotic behaviour as a corresponding piecewise stationary. This idea breaks down when the asymptotic expansion is considered. 
\Cref{thm:main-expansion-piecewise} establishes an expansion in the small blocks scenario. We note that the rate is determined by the internal clusters only. \Cref{sec:large-blocks-MMA(1)} deals with large blocks scenario, which is quantitatively different as compared to the small blocks one.

\Cref{sec:technical-details} contains the first part of technical details, with some results being of independent interest.  
\Cref{sec:internal-clusters} includes results related to internal clusters (that is, functionals of a single block). The main goal is to extend vague convergence of Theorem 6.2.5 in \cite{kulik:soulier:2020} (see \Cref{theo:cluster-RV} below) to unbounded functionals such as the cluster length. 
Under the small blocks condition (which is roughly equivalent to uniform integrability), the cluster length behaves as predicted by \Cref{theo:cluster-RV}. When the uniform integrability fails, the behaviour of the cluster length is different. The phase transition occurs precisely when small blocks are replaced with large blocks. The results are taken from \cite{chen:kulik:2023a}.

\Cref{sec:technical-boundary} deals with the boundary clusters. The key result is the rate of convergence for the event "a big jump in two consecutive blocks." We have a dichotomous behaviour: in case of small blocks, the rate is $w_n=\pr(\norm{\bsX_0}>\tepseq)$, while for large blocks the rate is $\dhinterseq^2 w_n^2$
See \Cref{lem:pa1-cap-pa2-precise} and \Cref{lem:pa1-cap-pa2-precise-large}, respectively. Extensions to unbounded functionals that vanish on the event "at most one big jump" follow.

The second part of technical details is included in \Cref{sec:technical-details-block-statistics}. 
The theory established in \Cref{sec:internal-clusters,sec:technical-boundary} is applied to particular, rather complex, functionals that appear in the context of the asymptotic expansion of blocks statistics. See \Cref{sec:moments-of-clusters} for small blocks and \Cref{sec:moments-of-clusters-large} for large blocks. 
\Cref{sec:dependence-clusters} establishes the growth of both internal and boundary clusters statistics. These statistics are sums of internal and boundary clusters. As such, we use the results from \Cref{sec:internal-clusters,sec:technical-boundary} in conjunction with some weak dependence assumptions. 

The last part of technical details can be found in \Cref{sec:detailed-decomposition}. There, we provide precise computations for somehow cumbersome expansions. 

\subsection{What is missing?}
We provide rather complete theory is the small blocks scenario. The behavior in the large
blocks scenario is illustrated under much simpler
dependence assumptions. However, from the
statistical perspective,
large blocks are much less relevant in practice.
 
Also, we do not consider here functionals that do not vanish around zero. For example functionals that lead to the large deviation index (see e.g. \cite{mikosch:wintenberger:2013}). These functionals may be large due to a cumulation of small values. This is usually prevented by imposing a "negligibility of small values" condition, but it is still not suitable for techniques and the framework of the paper. 

Furthermore, 
we do not formulate here nor prove the central limit theorems for
two types cluster. We refer to \cite{chen:kulik:2023a} for the appropriate results related to the internal clusters.  
\section{Preliminaries}\label{sec:prel}
In this section we fix the notation and introduce the relevant classes of functions. In \Cref{sec:tail-process} we recall the notion of the tail and the spectral tail process (cf. \cite{basrak:segers:2009}). In \Cref{sec:cluster-index} we define cluster indices; see \cite[Chapter 5]{kulik:soulier:2020} for a detailed introduction.

In \Cref{sec:anticlustering-condition} we introduce  anticlustering conditions. The first one, the classical \ref{eq:conditiondh}, is needed, to establish convergence of the cluster measure; see
\Cref{sec:cluster-measure-convergence}.  The results of the latter section are extracted from \cite[Chapter 6]{kulik:soulier:2020}.
See also \cite{planinic:soulier:2018} and \cite{basrak:planinic:soulier:2018}. Another condition, called here \ref{eq:conditionSstronger:gamma}, was introduced in \cite{chen:kulik:2023a}. 

In \Cref{sec:dependence-assumptions} we introduce our dependence assumptions.

In
\Cref{sec:clt-disjoint,sec:clt-sliding} we state simple versions of the existing results on disjoint and sliding blocks estimators for stationary time series. In the PoT framework, disjoint blocks estimators are considered in \cite{drees:rootzen:2010} and \cite[Chapter 10]{kulik:soulier:2020}. Results for sliding blocks are taken from \cite{cissokho:kulik:2021} (see also \cite{drees:neblung:2021}). The (somehow surprising) conclusion is that the disjoint and sliding blocks estimator yield the same asymptotic behaviour.

\subsection{Notation}
Let $\norm{\cdot}$ be an arbitrary norm on $\Rset^d$ and $\{\tepseq\}$, $\{\dhinterseq\}$ be such that
\begin{align}\label{eq:rnbarFun0}
\lim_{n\to\infty}\tepseq=\lim_{n\to\infty}\dhinterseq =\lim_{n\to\infty}nw_n = \infty\;, \ \ \lim_{n\to\infty}\dhinterseq/n=\lim_{n\to\infty}\dhinterseq w_n = 0\;,
\tag{$\conditionR(\dhinterseq,\tepseq)$}
\end{align}
where $w_n=\pr(\norm{\bsX_0}>\tepseq)$. 
For a sequence $\bsx=\{\bsx_j,j\in\Zset\}\in (\Rset^d)^\Zset$ and $i\leqslant j\in\Zset\cup\{-\infty,\infty\}$ denote $\bsx_{i,j}=(\bsx_i,\ldots,\bsx_j)\in (\Rset^d)^{j-i+1}$, $\bsx_{i,j}^\ast=\max_{i\leqslant l\leqslant j}|\bsx_l|$ and $\bsx^\ast=\sup_{j\in\Zset}|\bsx_j|$. By $\bszero$ we denote the zero sequence; its dimension can be different at each of its occurrences.

By $\lzero(\Rset^d)$ we denote the set of $\Rset^d$-valued sequences which tend to zero at infinity.

\subsection{Tail process}\label{sec:tail-process}
Let $\bsX=\sequence{\bsX}$ be a stationary, regularly varying time series with values in $\Rset^d$ and tail index $\alpha$. In particular,
\begin{align*}
\lim_{x\to\infty}\frac{\pr(|\bsX_0|> tx)}{\pr(|\bsX_0|> x)}=t^{-\alpha}
\end{align*}
for all $t>0$.
 Then, there exists a sequence $\bsY=\sequence{\bsY}$ such that
\begin{align*}
  \pr(x^{-1}(\bsX_i,\dots,\bsX_j) \in \cdot \mid |\bsX_0|>x)   \mbox{ converges weakly to }  \pr((\bsY_i,\dots,\bsY_j) \in \cdot)
\end{align*}
as $x\to\infty$ for all $i\leqslant j\in\Zset$.
We call $\bsY$ the tail process.
See \cite{basrak:segers:2009}. Equivalently, viewing $\bsX$ and $\bsY$ as random elements with values in $(\Rset^d)^\Zset$, we have for every bounded or
\nonnegative\ functional~$H$ on $(\Rset^d)^{\Zset}$, continuous \wrt\ the product topology,
\begin{align}\label{eq:tailprocesstozero}
   \lim_{x\to\infty} \frac{\esp[H(x^{-1}\bsX)\ind{\norm{\bsX_0}>x}]} {\pr(\norm{\bsX_0}>x)}
  & =  \esp[H(\bsY)] \; .
\end{align}
The random variable $|\bsY_0|$ has the Pareto distribution with index $\alpha$ and hence $\norm{\bsY_0}>1$.
The spectral tail process $\bsTheta=\{\bsTheta_j,j\in\Zset\}$ is defined as $\bsTheta_j=\bsY_j/|\bsY_0|$.

\subsection{Classes of functions}\label{sec:classes}
Let $\epsilon>0$. 
We will need a sequence of exceedance times associated with $\bsx$. Define 
\begin{subequations}
\begin{align}
T^{(1)}(\bsx,\epsilon)=T_{\rm min}(\bsx,\epsilon)  &= \inf\{ j \in \Zset : | \bsx_j | > \epsilon \}\;, \label{eq:exc-times-x-1}\\
     T_{\rm max}(\bsx,\epsilon)&=\sup\{ j \in \Zset : | \bsx_j | > \epsilon \}\;,  \label{eq:exc-times-x-2}\\
     T^{(i+1)}(\bsx,\epsilon)&=\inf\{j>T^{(i)}(\bsx,\epsilon): |\bsx_j|>\epsilon\}\;, \ \ i\geqslant 1\;, \label{eq:exc-times-x-3}\\
     \Delta T^{(i)}(\bsx,\epsilon)&= T^{(i+1)}(\bsx,\epsilon)-T^{(i)}(\bsx,\epsilon)\;. \label{eq:exc-times-x-4}
\end{align}
\end{subequations}
We will use the convention $\inf\{\emptyset\}=+\infty$, $\sup\{\emptyset\}=-\infty$. If 
$\bsx\in \ell_0(\mathbb{R}^d)$, then $\sup\{ j\geqslant 0: | \bsx_j | > \epsilon   \}<\infty$ and
$\inf\{ j\geqslant 0: | \bsx_j | > \epsilon   \}>-\infty$. Hence, 
when restricted to $\ell_0$, the functionals $T^{(i)}$, $i\geqslant 1$, and $T_{\rm max}$ attain finite values. 

Formally, the above functionals depend on $\epsilon$. However, without loss of generality we will assume throughout the paper that $\epsilon=1$ and we will use the notation $T^{(i)}(\bsx)$ instead. 

Let again $\bsx\in \ell_0(\mathbb{R}^d)$. We can define the
\textbf{cluster length functional} as
\begin{align}\label{eq:cluster-length-def}
\clusterlength(\bsx) =T_{\rm max}(\bsx)-T_{\rm min}(\bsx)+1\;
\end{align}
 with the convention $\clusterlength(\bsx)=0$ whenever $\bsx^\ast=\sup_{j\in\Zset}\norm{\bsx_j}\leqslant 1$. We note also that if $\norm{\bsx_0}>1$, while $\sup_{j\in\Zset, j\not=0}\norm{\bsx_j}\leqslant 1$, then $\clusterlength(\bsx)=1$. Also, neither $\clusterlength$ nor $T^{(i)}$ can be chosen in \eqref{eq:tailprocesstozero}, since they are not defined on $(\mathbb{R}^d)^\Zset$.

We define further  
the functional $\exc$ by $\exc(\bsx)=\sum_{j\in\Zset}\ind{\norm{\bsx_j}>1}$. It returns the number of exceedances over 1. 

We shall consider functionals under the following assumptions.
\begin{hypothesis}
\label{Assumption:class-mathcalH}
We denote by $\mch$ a collection of functionals $\ell_0(\Rset^d) \to \Rset_+$ such that  each $H \in \mathcal{H}$ satisfies:
\begin{itemize}
    \item[$(\rom1)$] $H$ is continuous with respect to the law of the process $\bsY$;
    \item[$(\rom2)$] If $\exc(\bsx)=0$, then $H(\bsx)=0$;
    \item[$(\rom3)$] If $\exc(\bsx)>0$, then $H(\bsx) =H(\bsx_{T_{\rm min}(\bsx),T_{\rm max}(\bsx)})$, where
    $T_{\rm min}(\bsx)$ and $T_{\rm max}(\bsx)$, are the first and the last exceedance times defined in \eqref{eq:exc-times-x-1}-\eqref{eq:exc-times-x-2}.
\end{itemize}
\end{hypothesis}
We note that the above assumption allows for unbounded functionals. We will need to control their growth.
\begin{hypothesis}
\label{Assumption:class-mathcalHcH}
We denote by $\mch(\gamma)\subseteq \mch$ a collection of functionals $\ell_0(\Rset^d) \to \Rset$ such that  each $H \in\mch(\gamma)$ satisfies:
\begin{itemize}
\item[$(\rom4)$] There exists a constant $C_H>0$ such that
    $ H(\bsx)  \leqslant C_H  \big[ \clusterlength(\bsx) \big]^{\gamma}$ for all $\bsx \in \ell_0( \Rset^d )$.
\end{itemize}
\end{hypothesis}
Some functionals will play a special role:
\begin{itemize}
\item Obviously, we can take $H$ as the cluster length functional itself: $H=\clusterlength$. Then $\gamma=1$. 
\item Extremal index functional defined as $\extremalindexfunc(\bsx)=\ind{\bsx^\ast>1}$. It fulfills \Cref{Assumption:class-mathcalHcH} with $C_{\extremalindexfunc}=1$ and $\gamma=0$. This functional is bounded. In particular, $\extremalindexfunc(\bsY)=1$. 
\end{itemize}
We make additional comments:
\begin{itemize}
\item The class $\mch(\gamma)$ is parametrized by $\gamma$. 
\item If $\gamma=0$, we will denote $\|H\|=\sup_{\bsx\in(\Rset^d)^\Zset}|H(\bsx)|$. 

\item If $0<\gamma_1<\gamma_2$, then $\mch(\gamma_1)\subseteq \mch(\gamma_2)$. 

\item If $H\in\mch(\gamma)$, then for $p>0$, $|H|^p\in \mch(p\gamma)$.
\item Note that the functional $H$ can depend on small values, but only those that occur between large values ("within a cluster"). 
\item The property $(\rom2)$ in \Cref{Assumption:class-mathcalH} will be referred to as "$H$ vanishes around $\bszero$" or "$H$ has support separated from $\bszero$". In particular, $\clusterlength$ and $\extremalindexfunc$ vanish around $\bszero$. 
\end{itemize}
Let $p>0$. Any $H\in \mch(\gamma)$ induces the following three new functionals. 
\begin{subequations}
\begin{align}
\widetilde{H}_{\IC}(\bsx)&:=\sum_{i=1}^{\exc(\bsx)-1}\Delta T^{(i)}(\bsx)\{H(\bsx_{-\infty,T^{(i)}(\bsx)})+H(\bsx_{T^{(i+1)}(\bsx),\infty})-H(\bsx)\}\;,\label{eq:new-functional-IC}\\
\widetilde{H}_{\BC}(\bsx)&:=\sum_{i=1}^{\clusterlength(\bsx)-1}
\left\{H(\bsx)-H(\bsx_{-\infty,i-1})-H(\bsx_{i,\infty})\right\}\label{eq:new-functional-BC}\;, \\
\widetilde{H}_{\BC,p}(\bsx)&:=\sum_{i=1}^{\clusterlength(\bsx)-1}
|H(\bsx)-H(\bsx_{-\infty,i-1})-H(\bsx_{i,\infty})|^p\label{eq:new-functional-BC-p}
\;.
\end{align}
\end{subequations}
\begin{example}{\rm 
\begin{enumerate}
\item Take $H=\ind{\bsx^*>1}$. Then $\widetilde{H}_{\IC}(\bsx)=T_{\rm max}(\bsx)-T_{\rm min}(\bsx)=\clusterlength(\bsx)-1$ and $\widetilde{H}_{\BC}(\bsx)=-(\clusterlength(\bsx)-1)$. 
\item 
If $H$ is a linear functional, that is the functional of the form $H(\bsx_{i,j})+H(\bsx_{j+1,k})=H(\bsx_{i,k})$, $i<j<k$, then also $\widetilde{H}_{\IC}=\widetilde{H}_{\BC}=\widetilde{H}_{\BC,p}\equiv 0$. 
\item 
Take 
$H=\clusterlength$. Then $\widetilde{H}_{\IC}=-\sum_{i=1}^{\exc(\bsx)-1}(\Delta T^{(i)}(\bsx))^2$. We can view $\widetilde{H}_{\IC}$ as a "measure of sparsity" in the tail process. 
\end{enumerate}
}
\end{example}
\begin{remark}{\rm
\begin{itemize}
\item If $\bsX$ is extremally independent, then $\bsY_j=0$ for $j\not=0$. Then $\exc(\bsY)=1$ and $\clusterlength(\bsY)=1$. Hence, $\widetilde{H}_{\IC}=\widetilde{H}_{\BC}=\widetilde{H}_{\BC,p}\equiv 0$.
\item If $H\in\mch(\gamma)$, then $\widetilde{H}_{\IC}\in \mch(\gamma+1)$. The additional exponent comes from 
$\sum_{i=1}^{\exc(\bsx)-1}\Delta T^{(i)}(\bsx)= \clusterlength(\bsx)$. Then 
$|\widetilde{H}_{\IC}|\in \mch(p(\gamma+1))$. 
Likewise, $H\in\mch(\gamma)$, gives $\widetilde{H}_{\BC,p}\in \mch(p\gamma+1)$.
\end{itemize}
}
\end{remark}

\subsection{Cluster measure and cluster indices}\label{sec:cluster-index}
Consider
the infargmax functional $\anchor_0$ defined on $(\Rset^d)^\Zset$ by
    $\anchor_0(\bsx)=\inf\{j:\bsx_{-\infty,{j}}^\ast=\bsx^\ast\}$, with the convention that
    $\inf\{\emptyset\}=+\infty$.
If  $\pr(\anchor_0(\bsY)\notin\Zset)=0$ then we can define
\begin{align*}
  \canditheta = \pr(\anchor_0(\bsY)=0) \; .
\end{align*}
In fact, $\anchor_0$ can be replaced with any anchoring map (see \cite{planinic:soulier:2018} and \cite[Theorem 5.4.2]{kulik:soulier:2020}). In particular,
\begin{align}\label{eq:canditheta-anchor-conclusion}
\canditheta=\pr(\anchor_0(\bsY)=0)=\pr\left(\bsY_{-\infty,-1}^\ast\leqslant 1\right)=\pr\left(\bsY_{1,\infty}^\ast\leqslant 1\right)\;.
\end{align}
Therefore, $\canditheta$ can be recognized as the (candidate) extremal index. It becomes the usual extremal index under additional mixing and anticlustering conditions.

\begin{definition}[Cluster measure]
  \label{def:clustermeasure}
  Let $\bsY$ and $\bsTheta$ be the tail process and the spectral tail process, respectively, such that $\pr(\lim_{|j|\to\infty} |\bsY_j|=0)=1$. The {cluster measure} is the measure $\tailmeasurestar$ on
  $\lzero(\Rset^d)$ defined by
  \begin{align*}
\tailmeasurestar    = \canditheta \int_0^\infty \esp[\delta_{r\bsTheta}\ind{\anchor_0(\bsTheta)=0}] \alpha r^{-\alpha-1} \rmd r   \; .
\end{align*}
\end{definition}
The measure $\tailmeasurestar$ is
boundedly finite on $(\Rset^d)^\Zset\setminus\{\bszero\}$, puts no mass at $\bszero$ and is
$\alpha$-homogeneous. For every bounded or non-negative functional $H$ such that $H(\bsx)=0$ if
$\bsx^\ast\leqslant 1$ we have
\begin{align}
    \label{eq:cluster-measure}
  \tailmeasurestar(H) &=  \esp[H(\bsY) \ind{\anchor_0(\bsY)=0}]
  =\esp[H(\bsY) \ind{\bsY_{-\infty,-1}^\ast\leqslant 1}]
   \; .
\end{align}
\begin{definition}[Cluster index]
\label{def:cluster-index}
We will call $\tailmeasurestar(H)$ the cluster index associated to the
functional $H$.
\end{definition}

\subsection{Change of measure}\label{sec:change-of-measure}
It is important to notice the presence of $\ind{\bsY_{-\infty,-1}^\ast\leqslant 1}$ in the definition of $\tailmeasurestar$.  
It will be convenient to express some formulas in the language of an auxiliary process $\bsZ$. Define $\bsZ$
as $\bsY$ conditioned on the first exceedance over 1 to happen at time zero, that is
$\bsY_{-\infty,-1}^*\leqslant 1$. Then, combining \eqref{eq:canditheta-anchor-conclusion} with \eqref{eq:cluster-measure}, we obtain 
\begin{align*}
\tailmeasurestar(H)=\esp[H(\bsY) \ind{\bsY_{-\infty,-1}^\ast\leqslant 1}]=\canditheta\esp[H(\bsZ)]\;.
\end{align*}
\subsection{Anticlustering conditions}\label{sec:anticlustering-condition}
For each fixed $r\in\Nset$, the distribution of
$\tepseq^{-1}\bsX_{-r,r}$ conditionally on $\norm{\bsX_0}>\tepseq$ converges weakly to the
distribution of $\bsY_{-r,r}$. In order to let $r$ tend to infinity, we must embed all these finite
vectors into one space of sequences. By adding zeroes on each side of the vectors
$\tepseq^{-1}\bsX_{-r,r}$ and $\bsY_{-r,r}$ we identify them with elements of the space
$\lzero(\Rset^d)$. Then $\bsY_{-r,r}$
converges (as $r\to\infty$) to $\bsY$ in $\lzero(\Rset^d)$ if (and only if) $\bsY\in\lzero(\Rset^d)$ almost surely.

However, this is not enough for statistical purposes and we consider the following definition that controls persistence of large values on one block. 
\begin{definition}
  [\cite{davis:hsing:1995}, Condition~2.8]\label{def:DH}
  Condition~\ref{eq:conditiondh}
  holds if for all $s,t>0$,
  \begin{align}
    \label{eq:conditiondh}
    \lim_{\ell\to\infty} \limsup_{n\to\infty}\pr\left(\max_{\ell\leqslant |j|\leqslant  \dhinterseq}|\bsX_j|
    > \tepseq s\mid |\bsX_0|> \tepseq t \right)=0 \; .
    \tag{$\conditiondh[\dhinterseq][\tepseq]$}
  \end{align}
\end{definition}

Condition \ref{eq:conditiondh} is referred to as the (basic) anticlustering condition. It is fulfilled by many models, including geometrically ergodic Markov chains,
short-memory linear or max-stable processes.
\ref{eq:conditiondh} implies that $\bsY\in \lzero(\Rset^d)$
and
    $\vartheta = \pr\left ( \bsY^\ast_{-\infty,-1} \leqslant 1  \right)> 0$.
Also, \ref{eq:conditiondh} holds for sequence of \iid\ random
variables whenever $\lim_{n\to\infty} \dhinterseq w_n=0$.

\begin{definition}\label{def:conditionSstronger:gamma}
\hypertarget{SummabilityAC}{Condition
  \ref{eq:conditionSstronger:gamma}} holds if
for all $s,t>0$
\begin{align}
    \label{eq:conditionSstronger:gamma}
    \lim_{\ell\to\infty} \limsup_{n\to\infty} \frac{1}{ w_n}
    \sum_{i=\ell}^{\dhinterseq}i^\gamma \pr(\norm{\bsX_0}>\tepseq s,\norm{\bsX_i}>\tepseq t) = 0 \;
    . \tag{$\conditiondhsumstrongergammapaper{\dhinterseq}{\tepseq}{\gamma}$}
  \end{align}
\end{definition}
This condition has been introduced in \cite{chen:kulik:2023a}. 
\begin{remark}\label{rem:rn2gamma}
{\rm
\begin{itemize}
\item
It is obvious that in case of \iid\ or $\ell$-dependent sequences \ref{eq:conditionSstronger:gamma} holds if and only if $\lim_{n\to\infty}\dhinterseq^{\gamma+1}w_n=0$. However, we could not establish that \ref{eq:conditionSstronger:gamma} and $\lim_{n\to\infty}\dhinterseq^{\gamma+1}w_n=0$ are equivalent in general. 
\item
The condition~\hyperlink{SummabilityAC}{$\mathcal{S}^{(\gamma_1)}(\dhinterseq, \tepseq)$} implies \hyperlink{SummabilityAC}{$\mathcal{S}^{(\gamma_2)}(\dhinterseq, \tepseq)$} whenever $\gamma_1>\gamma_2\geqslant 0$.
    The condition~\hyperlink{SummabilityAC}{$\mathcal{S}^{(0)}(\dhinterseq, \tepseq)$} implies
    the condition~\ref{eq:conditiondh}.
\item In several places, we will need the new anticlustering condition to hold with $c\dhinterseq$ ($c>1$) instead of $\dhinterseq$. With some abuse of notation, we will not make a distinction.
\end{itemize}
}
\end{remark}
Next, we note the following bound. Let $c_1>c_1>1$. If \ref{eq:conditionSstronger:gamma} holds, then 
\begin{align}\label{eq:consequence-condition-s}
\sum_{i=c_1\dhinterseq}^{c_2\dhinterseq}\pr(|\bsX_0|>\tepseq,|\bsX_i|>\tepseq)=o(\dhinterseq^{-\gamma}w_n)\;. 
\end{align}

\subsection{Conditional convergence of clusters}
From the definition of the tail process we obtain immediately
\begin{align*}
    \lim_{n\to\infty} \esp[H(\tepseq^{-1} \bsX_{i,j})\mid \norm{\bsX_0}>\tepseq] = \esp[H(\bsY_{i,j})] \; ,
  \end{align*}
  for $i\leqslant j$ and suitable functionals $H$. Thanks to the anticlustering condition we can replace $i,j$ with a sequence diverging to $\infty$.

  \begin{proposition}[\cite{basrak:segers:2009}, Proposition~4.2; \cite{kulik:soulier:2020}, Theorem 6.1.4]
  \label{lem:tailprocesstozero}
  Let $\bsX$ be a stationary time series. Assume that~\ref{eq:conditiondh} holds.
  Let $H$ be a bounded or non-negative functional defined on $\ell_0(\Rset^d)$ that is almost
surely continuous with respect to the distribution of the tail process $\bsY$.
  Then for any constant $C>0$,
  \begin{align*}
    \lim_{n\to\infty} \esp[H(\tepseq^{-1} \bsX_{-C\dhinterseq,C\dhinterseq})\mid \norm{\bsX_0}>\tepseq] = \esp[H(\bsY)] \; .
  \end{align*}
\end{proposition}

\subsection{Vague convergence of cluster measure}\label{sec:cluster-measure-convergence}
We now investigate the unconditional convergence of
$\tepseq^{-1}\bsX_{1,\dhinterseq}$. Contrary to
\Cref{lem:tailprocesstozero}, where an extreme value was imposed at time 0, a large value in the
cluster can happen at any time.

Define the measures $\tailmeasurestar_{n,\dhinterseq}$, $n\geqslant1$, on $\lzero(\Rset^d)$ as follows:
\begin{align*}
  \tailmeasurestar_{n,\dhinterseq}
  & =  \frac{1} {\dhinterseq  w_n} \esp \left[\delta_{\tepseq^{-1}\bsX_{1,\dhinterseq}} \right] \; .
\end{align*}
We are interested in convergence of $\tailmeasurestar_{n,\dhinterseq}$ to
$\tailmeasurestar$.
We quote Theorem 6.2.5 in \cite{kulik:soulier:2020}.
\begin{proposition}
\label{theo:cluster-RV}
  Let  condition~\ref{eq:conditiondh} hold. Then, for all bounded continuous shift-invariant functionals
$H$ with support separated from $\bszero$ we have
  \begin{align*}
    \lim_{n\to\infty} \tailmeasurestar_{n,\dhinterseq}(H)
    = \lim_{n\to\infty}\frac{\esp[ H(\tepseq^{-1}\bsX_{1,\dhinterseq})]} {\dhinterseq w_n}
    = \tailmeasure^\ast(H) \; .
  \end{align*}
\end{proposition}
The immediate consequence is the following limit:
  \begin{align}
  \label{eq:extremal-index}
    \lim_{n\to\infty} \frac{\pr(\bsX_{1,\dhinterseq}^\ast>\tepseq)}{\dhinterseq w_n}
     =\canditheta  \; .
  \end{align}

\subsection{Dependence assumptions}\label{sec:dependence-assumptions}
Let $\sequence{\bsX}$ be a time series. We shall divide the sample $\{\bsX_1,\ldots, \bsX_n\}$ into $m_n\in\mathbb{N}$ disjoint blocks of same size $\dhinterseq\in\mathbb{N}$.
Without loss of generality we will assume that $n=m_n\dhinterseq$. Dependence within each block ("local dependence") is controlled by the appropriate anticlustering condition. On the other hand, dependence between the blocks ("global dependence") is controlled by the appropriate temporal dependence assumption.
We will consider two temporal dependence schemes.
\begin{itemize}
\item \textbf{Stationary, weakly dependent:} $\sequence{\bsX}$ is a stationary, regularly varying $\Rset^d$-valued time series. Weak dependence will be controlled by $\alpha$-mixing. We will assume that there exists $C,\epsilon>0$ such that 
    \begin{align}\label{eq:mixing-rates}
    \alpha_j=O(\exp(-Cj))\;, \ \ \lim_{n\to\infty}\exp\left(-\frac{C\epsilon}{2+\epsilon}\dhinterseq\right)/(\dhinterseq w_n)^{\epsilon/(2+\epsilon)}=0\;. 
    \end{align}
    This assumption will imply all the mixing bounds used in the paper. 
    
    In most of the results in the paper we can replace the above rate with a polynomial decay. We keep the exponential bound for clarity. 
\item \textbf{Piecewise stationary:} For a sample of size $n$, we have observations 
    \begin{align*}
    (\bsX_1,\ldots,\bsX_n)=((\bsX_{1}^{(1)},\ldots,\bsX_{\dhinterseq}^{(1)}),   (\bsX_{1}^{(2)},\ldots,\bsX_{\dhinterseq}^{(2)}),\ldots,
    (\bsX_{1}^{(m_n)},\ldots,\bsX_{\dhinterseq}^{(m_n)}))
    \end{align*}
    and $\{\bsX_{j}^{(i)},j\in\Nset\}$, $i=1,\ldots,m_n$, are independent copies of the regularly varying, weakly dependent, time series $\{\bsX_{j},j\in\Nset\}$. Formally, $(\bsX_1,\ldots,\bsX_n)$ is an array of random elements. 
\end{itemize}
The latter assumption will serve for some illustration purposes only. The rationale for this assumption is that, from the point of view of the central limit theorems, both disjoint and sliding blocks estimators behave as if the blocks were independent. We will show that this is not the case when the asymptotic expansion is considered.

For a future use, we recall the following mixing inequality. Let $\mathcal{F}_{i,j}$ be the sigma field generated by $\bsX_{i,j}$. Let $p,q,r>0$ be such that $1/r+1/p+1/q$. Then for $U\in \mathcal{F}_{-\infty,\ell}$ and $V\in \mathcal{F}_{\ell+i,\infty}$ we have
\begin{align}
\label{eq:mixing-inequality}
|\cov(U,V)|\leqslant 8\alpha_i^{1/r}\|U\|_p\|V\|_q\;. 
\end{align} 
\subsection{Central Limit Theorem - disjoint blocks estimators}\label{sec:clt-disjoint}
We quote results on blocks estimators in the PoT setting.
Let $\tepcluster$ be the Gaussian process on $L^2(\tailmeasurestar)$ with covariance
\begin{align*}
   \cov(\tepcluster(H),\tepcluster(\widetilde{H})) =  \tailmeasurestar(H\widetilde{H}) \; .
 \end{align*}
We state Theorem 10.2.1 in \cite{kulik:soulier:2020}. We refer to that theorem for the set of assumptions on functional $H$ and the mixing rates. Theorem is valid in particular for bounded, shift-invariant functionals that vanish around zero.
\begin{proposition}\label{thm:disjoint-block-clt-stat}
Let $\bsX=\sequence{\bsX}$ be a stationary, regularly varying $\Rset^d$-valued time series. Assume
  that~\ref{eq:rnbarFun0}, \ref{eq:conditiondh} hold and $\bsX$ is beta-mixing with the appropriate rates. 
Then
\begin{align}\label{eq:clt-estimator-disj}
\sqrt{n\pr(\norm{\bsX_0}>\tepseq)}\left\{\tedcluster(H)-\esp[\tedcluster(H)]\right\}\convdistr
\tepcluster(H)\;.
\end{align}
\end{proposition}
For piecewise stationary time series, the central limit theorem follows from Theorem 10.2.1 in \cite{kulik:soulier:2020}, without a need of the mixing assumption. We only need the basic anticlustering condition to control each disjoint block of size $\dhinterseq$.
\begin{proposition}\label{thm:disjoint-block-clt-piecewise}
Let $\bsX=\sequence{\bsX}$ be a piecewise stationary, regularly varying $\Rset^d$-valued time series. Assume
  that~\ref{eq:rnbarFun0}, \ref{eq:conditiondh} hold.
Then the central limit theorem \eqref{eq:clt-estimator-disj} holds.
\end{proposition}

\subsection{Central Limit Theorem - sliding blocks estimators}\label{sec:clt-sliding}
The next result is from \cite{cissokho:kulik:2021}. We again refer to that paper for the precise assumptions.

\begin{proposition}[Theorem 5.12 in \cite{cissokho:kulik:2021}]\label{thm:sliding-block-clt-stat}
Let $\bsX=\sequence{\bsX}$ be a stationary, regularly varying $\Rset^d$-valued time series. Assume
  that~\ref{eq:rnbarFun0}, \ref{eq:conditiondh} hold and $\bsX$ is beta-mixing with the appropriate rates. 
Then
\begin{align}\label{eq:clt-estimator-1}
\sqrt{n\pr(\norm{\bsX_0}>\tepseq)}\left\{\tedclustersl(H)-\esp[\tedclustersl(H)]\right\}\convdistr
\tepcluster(H)\;.
\end{align}
\end{proposition}
From the asymptotic expansions developed in the paper, we can also infer that (under additional restrictions on the blocks size) the above asymptotics holds for the piecewise stationary case.

\textbf{Thus, in the PoT setting, the disjoint and sliding blocks estimator for both stationary and piecewise stationary case, lead to the same asymptotic variance.}

\section{Internal and boundary clusters}\label{sec:two-types-of-clusters}


\subsection{Notation}\label{sec:notation}
We consider disjoint blocks of size $\dhinterseq$:
\begin{align*}
I_{j}=\{(j-1)\dhinterseq+1,\ldots,j\dhinterseq\}\;, \ \ j=1,\ldots,m_n\;.
\end{align*}
The subscript $j$ will always indicate the numbering of blocks. 
We assume without loss of generality that $m_n\dhinterseq=n$. 
Set
\begin{align}\label{eq:def-DB}
\DB_j:=\DB_j(H):=\dhinterseq
  H(\tepseq^{-1}\bsX_{(j-1)\dhinterseq+1,j\dhinterseq})\;, \ \ 
\DB:=\DB(H):=\sum_{j=2}^{m_n-1} \DB_j(H)\;.
\end{align}
Note that to avoid dealing with boundary terms, we consider blocks $j=2,\ldots,m_n-1$ only. Keeping this in mind, the disjoint blocks statistics in
\eqref{eq:blocktype} can be written as
\begin{align}\label{eq:disjoint-rep}
\tedcluster(H)&= \frac{1}{n\dhinterseq w_n} \DB(H)\;.
\end{align}
Likewise, let
\begin{align}\label{eq:def-SBj}
\SB_j:=\SB_j(H):=\sum_{i\in I_j}
    H\left(\tepseq^{-1}\bsX_{i,i+\dhinterseq-1}\right)\;.
\end{align}
That is, $\SB_j(H)$ is the contribution to the sliding blocks statistics coming from the block $I_j$. Note that $\SB_j(H)$ is a function of the random variables $\bsX_i$, $i\in I_j\cup I_{j+1}$ only.
Set 
\begin{align}\label{eq:def-SB}
\SB:=\SB(H):=\sum_{j=2}^{m_n-1} \SB_j(H)\;.
\end{align} 
Thus, ignoring the boundary effects, the sliding blocks statistics
defined in \eqref{eq:sliding-block-estimator-nonfeasible-1} becomes
\begin{align}\label{eq:sliding-rep}
\tedclustersl(H)&=\frac{1}{n \dhinterseq w_n} \SB(H)\;.
\end{align}
The difference between the sliding and disjoint blocks statistics will heavily depend on the number of exceedances in the consecutive blocks. 
Denote by
\begin{align*}
N_{j}:=\exc(\tepseq^{-1}\bsX_{(j-1)\dhinterseq+1,j\dhinterseq})=\sum_{i\in I_{j}}\ind{\norm{\bsX_i}>\tepseq}\;, \ \ j=1,\ldots,m_n\;,
\end{align*}
the number of exceedances in the block $I_{j}$.
Set
\begin{align}\label{eq:set-Aj}
A_{j} = \{ \exists i : (j-1)\dhinterseq + 1 \leqslant i \leqslant  j\dhinterseq , \vert \bsX_i \vert > \tepseq  \}=\{\bsX_{(j-1)\dhinterseq + 1,j\dhinterseq}^\ast>\tepseq\}\;.
\end{align}
Recall the notation \eqref{eq:exc-times-x-1}-\eqref{eq:exc-times-x-3}. 
If $N_{j}\not=0$, then denote by 
\begin{align*}
t_j{(i)}=T^{(i)}(\tepseq^{-1}\bsX_{(j-1)\dhinterseq+1,j\dhinterseq})
\end{align*} 
the $i$th exceedance time in the block $j$. Hence, 
\begin{align}\label{eq:exceedance}
t_j{(1)}<\cdots< t_{j}{(N_{j})}\;, \ \ j=1,\ldots,m_n\;,
\end{align}
are 
the exceedance times in the $j$th block. We will use the convention $t_{j}(0)\equiv (j-1)\dhinterseq$ and $t_{j}(N_{j}+1)\equiv j\dhinterseq$. Note that it is possible that
$t_{j}{(N_{j})}=t_{j}(N_{j}+1)=j\dhinterseq$. This happens when the last jump in the block $j$ occurs at the right-end point, $j\dhinterseq$. On the other hand, since the $j$th block starts at $(j-1)\dhinterseq+1$, $t_{j}(1)$ is strictly larger than $t_{j}(0)$.
We will set $t_1(i)=t(i)$ for $i=1,\ldots,N_1$. 


We will also define 
\begin{align*}
\Delta t_{j}(i) =t_{j}(i+1) -t_{j}(i)
\end{align*}
for each $i=0,\ldots, N_j$.
With this notation we have
$
\sum_{i=0}^{N_j}\Delta t_{j}(i) =\dhinterseq$. 
Note that this is valid regardless whether $N_j=0$ or $N_j\not=0$, thanks to the convention introduced above.
Furthermore, using the notation introduced in \eqref{eq:cluster-length-def},
\begin{align*}
\sum_{i=1}^{N_j-1}\Delta t_{j}(i) +1=
t_{j}{(N_{j})}-t_{j}{(1)}+1=\clusterlength(\tepseq^{-1}\bsX_{(j-1)\dinterseq+1,j\dhinterseq})=
:\clusterlength_{j}
\end{align*}
is the cluster length in the $j$th block.
We use the convention $\sum_{i=\ell}^j=0$ whenever $j<\ell$.
Hence, $\clusterlength_{j}=1$ whenever $N_{j}=1$. If there are no exceedances over the threshold $\tepseq$, then we set $\clusterlength_{j}\equiv 0$ (there will be no issue, since $\clusterlength_{j}$ will appear with the appropriate indicator.) 

We need to keep track of the exceedances in each block.
For this, for each $j=1,\ldots,m_n$ and each $1\leqslant k_1 \leqslant k_2 \leqslant N_{j}$, we define
\begin{align*}
\mathbb{X}_{j}{(k_1:k_2)} = \tepseq^{-1} \big(\bsX_{ t_{j}{(k_1)}}, \ldots, \bsX_{t_{j}{(k_2)} }  \big)\;.
\end{align*}
Thus, $\mathbb{X}_{j}{(1:N_{j})}$ is a vector of size $\clusterlength_{j}$ that consists of all (scaled) exceedances in the block $j$ and all small values between the first and the last exceedance. We use the convention $\mathbb{X}_{j}{(k_1:k_2)}=0$ whenever $k_1>k_2$. Hence, in particular,
$\mathbb{X}_{j}{(1:N_{j})}=0$ if $N_{j}=0$.

We will also use the notation 
\begin{align*}
\mathbb{X}_{j}=\tepseq^{-1}\bsX_{(j-1)\dhinterseq+1,j\dhinterseq}\;, \ \ j=1,\ldots,m_n
\end{align*}
for the scaled block $j$.
By the definition, for $H\in \mch(\gamma)$ we have $H(\mathbb{X}_{j})=H(\mathbb{X}_{j}{(1:N_{j})})$.

\subsubsection{Adjacent blocks}\label{sec:adjacent-blocks}
We will also use the notation
\begin{align}\label{eq:DBj,j+1}
\DB_{j,j+1}(H)=\dhinterseq H(\tepseq^{-1}\bsX_{(j-1)\dhinterseq+1,(j+1)\dhinterseq})\;
\end{align}
and define
\begin{align}
    \label{eq:boundary-full-cluster-length}
    \mathcal{L}_{j,j+1} = t_{j+1}{(N_{j+1})} - t_{j}(1) + 1\;, 
\end{align}
the joint cluster length for blocks $j$ and $j+1$. When adjacent blocks with large values are considered, we shall rename $t_{j}(1)< \cdots <t_{j}{(N_j)}< t_{j+1}(1)< \cdots<t_{j+1}{(N_{j+1})}$ as
\begin{align}\label{eq:boundary-full-jump-times}
t_{j,j+1}(1)< \cdots <t_{j,j+1}{(N_j)}< t_{j,j+1}{(N_j+1)}< \cdots<t_{j,j+1}{(N_j+N_{j+1})}\;. 
\end{align}
By the convention
$t_{j,j+1}(0)=(j-1)\dhinterseq$ and $t_{j,j+1}{(N_j+N_{j+1}+1)}=(j+1)\dhinterseq$. 

We will also use the notation $\Delta t_{j,j+1}(i)=t_{j,j+1}(i+1)-t_{j,j+1}(i)$.
\subsection{Two types of blockwise clusters}\label{subsec:two-types-clusters}
We will distinguish between two kinds of clusters, with distinct properties.
We shall call the first kind an \textit{internal cluster}: it typically lives inside a block $I_j$, and whose emergence excludes exceedances in neighboring blocks $I_{j-1}$ and $I_{j+1}$.
The second kind is referred to as a \textit{boundary cluster}. Its name suggests that it describes simultaneous exceedances in two adjacent disjoint blocks $I_j$ and $I_{j+1}$. The exceedances are typically ``near'' the boundary points $j\dhinterseq$ and $j\dhinterseq+1$. In this case, the exceedances in $I_{j+1}$ persist from $I_{j}$.
\subsubsection{Internal clusters}\label{sec:contribution-internal-clusters}
Recall the notation \eqref{eq:set-Aj}: $A_j=\{\bsX_{(j-1)\dhinterseq+1,j\dhinterseq}^*>\tepseq\}$. 
Informally speaking, we have an internal cluster in the  block $I_j$ if $\ind{A_{j-1}^c\cap A_j\cap A_{j+1}^c}=1$. That is, there is a large value in the block $j$, but there are no big values in the adjacent blocks.

Formally, the name "internal cluster" is due to the following result. We note that
the scaled support of exceedance times $ \dhinterseq^{-1}\{ t^{(1)}_1,\ldots,  t^{(1)}_{N_1}  \}$ converges weakly to the random singleton $\{U_1 \}$, which  lies in the interior of $(0,1)$. The following result follows from \cite{chen:kulik:2023a}.
\begin{lemma}
\label{prop:main-results-interior-cluster}
Assume that~\ref{eq:conditiondh} and \ref{eq:rnbarFun0} hold.
Let a random variable $U_1$ be  ${\rm Uniform}(0,1)$ and independent of the tail process $\bsY$.
Then
    \begin{align*}
    \left(
     \frac{t_{1}(1)}{\dhinterseq},\frac{t_{1}{(N_1)}}{\dhinterseq},
    \ind{A_0^c},\ind{A_2^c}
    \right)
    \xLongrightarrow{\mathbb{P}(\cdot \mid  A_1 )}
    \left(
     U_1,U_1, 1, 1
    \right)\;.
\end{align*}
\end{lemma}
\paragraph{Internal clusters statistics.}
Motivated by \Cref{prop:main-results-interior-cluster}, on a general event $A_{j-1}^c \cap A_j \cap A_{j+1}^c$, there will be three blocks estimators
capturing the cluster in $I_j$. Namely, the sliding blocks estimators $\SB_{j-1}$ and $\SB_j$, and the disjoint blocks estimator $\DB_j$.
To be more specific 
(see \Cref{sec:detailed-decomposition} for the detailed explanation):
\begin{enumerate}
\item On $A_{j-1}^c\cap A_j$ ($N_{j-1}=0$ and $N_j\not=0$), we have
\begin{align}\label{eq:SBj-1:nojump-jump}
\SB_{j-1}(H)&=\sum_{i=(j-2)\dhinterseq+1}^{(j-1)\dhinterseq}H(\tepseq^{-1}\bsX_{i,i+\dhinterseq-1})
=\sum_{i=1}^{N_j}\Delta t_j(i)H(\mathbb{X}_{j}{(1:i)})\;.
\end{align}
\item On $A_j\cap A_{j+1}^c$ ($N_j\not=0$ and $N_{j+1}=0$), we have
\begin{align}\label{eq:SBj:jump-nojump}
\SB_j(H)&=\sum_{i=0}^{N_j-1} \Delta t_j(i)H(\mathbb{X}_{j}{(i+1:N_j)})\;.
\end{align}
\item On $A_j$, $\DB_j(H)=\dhinterseq H(\mathbb{X}_{j})$.
\end{enumerate} 
This leads to the following definition, that represents a contribution to the difference between sliding and disjoint blocks, stemming from the internal clusters (see \Cref{sec:detailed-decomposition} for the detailed calculation):
\begin{align}\label{eq:interior-clusters-def}
&\IC:=\IC(H):=\sum_{j=2}^{m_n-1}(\SB_{j-1}+\SB_j-\DB_j)   \ind{A_{j-1}^c\cap A_j\cap A_{j+1}^c}=:\sum_{j=2}^{m_n-1}\IC_{j}(H)\;,
\end{align}
with 
\begin{align}\label{eq:internal-as-functional}
\IC_j(H)=\widetilde{H}_{\IC}(\tepseq^{-1}\bsX_{(j-1)\dhinterseq+1,j\dhinterseq})\ind{A_{j-1}^c\cap A_j\cap A_{j+1}^c}\;.
\end{align}
We will refer to $\IC_j(H)$ as the \textit{internal cluster} and to $\IC$ as the \textit{internal clusters statistics}. 
\begin{example}\label{xmpl:extremal-index-1}{\rm
If $H(\bsx)=\ind{\bsx^*>1}$. Here, $H(\mathbb{X}_{j}{(i_1:i_2)})=1$ for any $1\leqslant i_1\leqslant i_2\leqslant N_j$. Hence, $\IC_j(H)$ in \eqref{eq:internal-as-functional} becomes 
\begin{align*}
\IC_j(H)=\left(t_{j}{(N_j)}-t_{j}(1)\right)\ind{A_{j-1}^c\cap A_j\cap A_{j+1}^c}=\left(\clusterlength_j-1\right) \ind{A_{j-1}^c\cap A_j\cap A_{j+1}^c}\;.
\end{align*}
Obviously, $\IC_j(H)\equiv 0$ if there is only one large jump in the $j$th block. 
}
\end{example}
\begin{remark}{\rm
We note that it is important to combine $\SB_{j-1}$ and $\SB_j$ together. This allows us to reduce the sums
    \begin{align*}
    \widetilde{\IC}_{j}(H):
    &=\ind{A_{j-1}^c\cap A_j\cap A_{j+1}^c}\sum_{i=0}^{N_j}\Delta t_{j}{(i)}\left( H(\mathbb{X}_{j}{(1:i)})+  H(\mathbb{X}_{j}{(i+1:N_j)}) -H(\mathbb{X}_{j})\right)
    \end{align*}
    to $\IC_{j}(H)$. Replacing the sum $\sum_{i=0}^{N_j}$ with $\sum_{i=1}^{N_j-1}$ is crucial.  
Indeed, if $H$ is bounded, then
$$|\widetilde\IC_{j}(H)|\leqslant 3\|H\| \dhinterseq \ind{A_j}$$
while
$$|\IC_{j}(H)|\leqslant 3\|H\|\clusterlength_j \ind{A_j}\;.$$
The cluster length $\clusterlength_j$ is tight under the conditional distribution (given $A_j$). See \cite{chen:kulik:2023a}. Hence, combining $\SB_{j-1}$ with $\SB_j$ allows us to get the optimal convergence rates.
}
\end{remark}

\subsubsection{Boundary clusters}\label{sec:contribution-boundary-clusters}
Informally, we have a boundary cluster between blocks $I_j$ and $I_{j+1}$ if $\ind{A_{j}\cap A_{j+1}}=1$. We have large values in the adjacent blocks. The chance of this event is much smaller than $\pr(A_1)$. If the blocks were independent, then $\pr(A_j\cap A_{j+1})$ would be of the order $\dhinterseq^2 w_n^2$. In case of temporal dependence, the chance of this joint event is different as indicated in the following lemma. The lemma provides the precise asymptotic behaviour in case of small blocks - recall that \hyperlink{SummabilityAC}{$\mathcal{S}^{(1)}(\dhinterseq, \tepseq)$} requires $\dhinterseq^2w_n\to 0$ in e.g. \iid\ case. 
\begin{lemma}\label{lem:pa1-cap-pa2-precise-at-the-beginning}
Assume that \hyperlink{SummabilityAC}{$\mathcal{S}^{(1)}(\dhinterseq, \tepseq)$} holds.
Then
\begin{align*}
\lim_{n\to\infty}\frac{\pr(A_1\cap A_2)}{w_n}=\canditheta \esp[(\clusterlength(\bsZ) - 1)] 
\;.
\end{align*}
\end{lemma}
Interestingly, in case of large blocks the asymptotic behaviour changes. See  
\Cref{lem:pa1-cap-pa2-precise}.
 
Informally speaking, the above lemma explains the name 
"boundary cluster": a (finite number) of large values occurs at the end of block $j$ and at the beginning of block $j+1$. 

\paragraph{Boundary clusters statistics.}
We analyse the contribution from the boundary clusters to blocks statistics. See \Cref{sec:detailed-decomposition}. 
\begin{itemize}
\item On $A_{j-1}^c\cap A_j$, the contribution $\SB_{j-1}(H)$ is the same as in \eqref{eq:SBj-1:nojump-jump}.
\item On $A_{j+1}\cap A_{j+2}^c$, the contribution is $\SB_{j+1}(H)$ is the same as in \eqref{eq:SBj:jump-nojump} (with $j$ replaced by $j+1$).
\item On $A_j$, $\DB_j(H)=\dhinterseq H(\mathbb{X}_{j})$.
\item On $A_j\cap A_{j+1}$, $\SB_j(H)$ has a cumbersome form (see \Cref{sec:detailed-decomposition}).
\end{itemize}
This leads to the following definition, that represents a contribution to the difference between sliding and disjoint blocks, stemming from the boundary clusters:
\begin{align}\label{eq:boundary-clusters-def}
\BC:=\BC(H)&=\sum_{j=2}^{m_n-1}(\SB_{j-1}+\SB_j-\DB_j+\SB_{j+1}-\DB_{j+1})   \ind{A_{j-1}^c\cap A_j\cap A_{j+1}\cap A_{j+2}^c}\notag\\
&=:\sum_{j=2}^{m_n-1}\BC_j(H)\;.
\end{align}
We will refer to $\BC_j$ as the \textit{boundary clusters} and to $\BC$ as the \textit{boundary clusters statistics}. 

We need to expand it further. 
Recall the notation $\DB_{j,j+1}(H)$ from \eqref{eq:DBj,j+1}.
Write \eqref{eq:boundary-clusters-def} as
\begin{align}\label{eq:boundary-clusters-decomposition}
\BC(H):=\BC(H;1)+\BC(H;2)=
\sum_{j=2}^{m_n-1}\BC_j(H;1)+\sum_{j=2}^{m_n-1}{\BC}_j(H;2)
\end{align}
with
\begin{align}
   \BC_j(H;1) = &
     \left(\DB_{j,j+1} - \DB_j - \DB_{j+1}
    \right)\ind{A_{j-1}^c\cap A_j\cap A_{j+1}\cap A_{j+2}^c}\;
    \label{eq;difference-SB&DB;boundary-cluster-part;exceedance-times-expansion;main}
    \end{align}
and
\begin{align}
   {\BC}_j(H;2) = &
      \left(  \SB_{j-1} + \SB_j + \SB_{j+1} - \DB_{j,j+1}  \right)\ind{A_{j-1}^c\cap A_j\cap A_{j+1}\cap A_{j+2}^c}\;.
\label{eq;difference-SB&DB;boundary-cluster-part;exceedance-times-expansion;minor}
\end{align}
Recall \eqref{eq:DBj,j+1}-\eqref{eq:boundary-full-cluster-length}.
Thus, $\BC_j(H;1)$ has a very simple form:
\begin{align}\label{eq:boundary-1-as-functional}
\BC_j(H;1)=\dhinterseq\left(H(\mathbb{X}_{j})+H(\mathbb{X}_{j+1})- H(\mathbb{X}_{j,j+1})\right)
\ind{A_{j-1}^c\cap A_j\cap A_{j+1}\cap A_{j+2}^c}\;.
\end{align}
Next, on the event 
$\{\clusterlength_{j,j+1}<\dhinterseq\}$, with the definition of $\widetilde{H}_{\IC}(\bsx)$ in \eqref{eq:new-functional-IC}, we can write  
\begin{align*}
    &{\BC}_j(H;2)=\widetilde{H}_{\IC}(\mathbb{X}_{j,j+1})\ind{A_{j-1}^c\cap A_j\cap A_{j+1}\cap A_{j+2}^c}\;.
\end{align*}
with 
\begin{align}\label{eq:boundary-2-as-functional}
&\widetilde{H}_{\IC}(\mathbb{X}_{j,j+1}):=
\sum_{i=1}^{M_j} \Delta t_{j,j+1}(i) \left(
    H( \mathbb{X}_{j,j+1}{(1:i)} ) + H( \mathbb{X}_{j,j+1}{(i+1:N_{j}+N_{j+1})} ) 
    - H( \mathbb{X}_{j,j+1} )\right)\;,
\end{align}
where $M_j:=N_{j}+N_{j+1}-1$. See again \Cref{sec:detailed-decomposition} for the details. 

On the other hand, the event $\{\clusterlength_{j,j+1}\geqslant \dhinterseq\}$ has a small probability; see \Cref{cor:clusterlength-tail}.
\begin{example}\label{xmpl:extremal-index-2}{\rm
Let $H(\bsx)=\ind{\bsx^*>1}$.
Then $\BC_j(H;1)=-\dhinterseq\ind{A_{j-1}^c\cap A_j\cap A_{j+1}\cap A_{j+2}^c}$,
\begin{align*}
\BC_j(H;2)=\ind{A_{j-1}^c\cap A_j\cap A_{j+1}\cap A_{j+2}^c}
\left(\left(t_{j+1}{(N_{j+1})}-t_{j}(1)\right)-\left(t_{j+1}(1)-t_{j}{(N_j)}-\dhinterseq\right)_+\right)  
 \;
 \end{align*}
 and the last part simplifies to 
$\left(t_{j+1}{(N_{j+1})}-t_{j}(1)\right)$
on the event $\{\clusterlength_{j,j+1}<\dinterseq\}$
 }
\end{example}\section{Asymptotic expansions}\label{sec:main-results}
Throughout this section it is everywhere assumed that \ref{eq:conditiondh} and \ref{eq:rnbarFun0} hold.
Recall \eqref{eq:disjoint-rep}-\eqref{eq:sliding-rep}. The difference $\tedcluster(H)-\tedclustersl(H)$ between the sliding blocks and the disjoint blocks statistics is written as 
\begin{align*}
\tedcluster(H)-\tedclustersl(H)=
\frac{1}{n \dhinterseq w_n}\IC(H)+\frac{1}{n \dhinterseq w_n}\BC(H) +\frac{1}{n \dhinterseq w_n}\RC(H)\;,
\end{align*}
where the internal and the boundary clusters statistics are defined in \eqref{eq:interior-clusters-def} and \eqref{eq:boundary-clusters-def}, respectively, while the remainder term will be defined later, cf. \Cref{sec:remainder}. 

We summarize the results of this section as follows:
\begin{itemize}
\item In the small blocks scenario, the sliding and the disjoint blocks statistics differ by $O_P(\dhinterseq^{-1})$. In all the examples we consider the rate is in fact $o_P(\dhinterseq^{-1})$, but we are unable to draw a general conclusion. Both the internal and the boundary clusters contribute at the same rate.  See \Cref{theo:main-theorem}.
\item In the small blocks scenario, if $\bsX$ is piecewise stationary, then the difference is also of the order $O_P(\dhinterseq^{-1})$. This rate is sharp and cannot be improved to $o_P$ in general. Only the internal clusters contribute here. See \Cref{thm:main-expansion-piecewise}. 
\item In the large blocks scenario, the expected distance between the disjoint and the sliding blocks is of the order $O(\dhinterseq w_n)$ and this rate cannot be improved in general. Only the internal clusters contribute here. From this perspective, the large blocks scenario resembles the piecewise stationary case. 
    See \Cref{thm:main-expansion-large}.  
\end{itemize}
For the large blocks scenario we consider a simple case of an MMA(1) process and $H(\bsx)=\ind{\bsx^*>1}$. To get general results, one needs to establish proper asymptotic results for boundary clusters in the large blocks scenario. We do not pursue this direction. Indeed, 
one can argue that the large blocks scenario is statistically not relevant: large blocks yield large variance of blocks statistics. 

\subsection{Small blocks scenario}
Recall the functionals $\widetilde{H}_{\IC}$ and $\widetilde{H}_{\BC}$; cf. 
\eqref{eq:new-functional-IC}, \eqref{eq:new-functional-BC}.
The result below deals with asymptotic expansion for the disjoint and the sliding blocks statistics. We refer to this result as the "small blocks scenario", since the anticlustering condition 
\hyperlink{SummabilityAC}{$\mathcal{S}^{(\delta)}(\dhinterseq, \tepseq)$} forces the blocks size to fulfill $\dhinterseq^{\delta+1}w_n\to 0$. Note that both internal and boundary clusters contribute at the same rate. 
\begin{theorem}\label{theo:main-theorem}
Assume that $\bsX$ is stationary and mixing with the rates \eqref{eq:mixing-rates}. 
Assume that~\hyperlink{SummabilityAC}{$\mathcal{S}^{(2\gamma+3)}(\dhinterseq, \tepseq)$} 
holds.  Let $H\in\mch(\gamma)$. 
Then 
\begin{align*}
\tedcluster(H)-\tedclustersl(H)=\frac{1}{n \dhinterseq w_n}\IC(H)+\frac{1}{n \dhinterseq w_n}\BC(H)+\textcolor{black}{O_P\left(\frac{\dhinterseq}{n}+\dhinterseq^{-(\gamma+4)}\right)}\;,
\end{align*}
where 
\begin{align}\label{eq:internal-clusters-convprob}
\frac{\IC(H)}{nw_n}\convprob\tailmeasurestar(\widetilde{H}_{\IC})\;,
\end{align}
\begin{align}\label{eq:boundary-clusters-convprob}
\frac{\BC(H)}{n w_n} \convprob
         \tailmeasurestar(\widetilde{H}_{\BC})\;.
\end{align}
The latter convergence holds 
as long as $\dhinterseq/(nw_n)\to 0$.
\end{theorem}
\begin{proof}
\Cref{corollary:interior-cluster-means} and \Cref{prop:total-interior-cluster-variance} give 
\eqref{eq:internal-clusters-convprob}.
\Cref{corollary:boundary-cluster-means} and \Cref{prop:boundary-clusters-full-variance} yield 
\eqref{eq:boundary-clusters-convprob}.
The rate for the remainder follows from \Cref{lem:remainder}.
\end{proof}
We can provide some heuristic for the result above. 
\begin{itemize}
\item
For the internal clusters, recall \eqref{eq:internal-as-functional}. Then, $\widetilde{H}_{\IC}$ is a tight cluster functional, the event $A_j$ occurs with the probability proportional to $\dhinterseq w_n$ and we have $m_n$ blocks. Hence, the rate for the internal clusters statistics is $m_n\times \dhinterseq w_n=nw_n$.   
\item For the boundary clusters statistics, recall first \eqref{eq:boundary-2-as-functional}. Again, $\widetilde{H}_{\IC}$ is tight, and the event $A_j\cap A_{j+1}$ occurs with the probability $o(\dhinterseq w_n)$. Hence, the term $\BC(H;2)$ will not contribute.
\item On the other hand, recall \eqref{eq:boundary-1-as-functional}. Then the functional considered there is or order $\dhinterseq$, the event $A_j\cap A_{j+1}$ occurs with the probability $w_n$. In total, the boundary clusters statistics will contribute at the rate $m_n\times \dhinterseq \times w_n=nw_n$, the same as the internal clusters statistics.  
\end{itemize}
\begin{remark}{\rm 
\begin{itemize}
\item Note that the disjoint and the sliding blocks statistics have the same mean. Thus, if $\sqrt{nw_n}\left\{\frac{1}{\dhinterseq}+\frac{\dhinterseq}{n}\right\}\to 0$, then \Cref{theo:main-theorem} gives that the disjoint and the sliding blocks statistics yield the same central limit theorem. The restriction on $w_n$ and $\dhinterseq$ is certainly stronger as compared to the results obtained by direct computations.
\item The anticlustering condition \hyperlink{SummabilityAC}{$\mathcal{S}^{(2\gamma+3)}(\dhinterseq, \tepseq)$} may seem too strong, but seems to be optimal to get the general results. 
\item Recall $\bsZ$ from \Cref{sec:change-of-measure}. We note that for suitably chosen $g:\Rset\to\Rset$ we have 
\begin{align*}
&\tailmeasurestar(g(\widetilde{H}_{\IC}))=\canditheta\esp\left[g(\widetilde{H}_{\IC}(\bsZ))\right]\\
&=\canditheta\esp\left[g\left(\sum_{i=1}^{\exc(\bsZ)-1}\Delta T^{(i)}(\bsZ)\{ H(\bsZ_{0,T^{(i)}(\bsZ)}) +  H(\bsZ_{T^{(i+1)}(\bsZ),T_{\rm max}(\bsZ)}) - H(\bsZ) \}\right)\right]\;.\notag
\end{align*}
and
\begin{align}\label{eq:short-formula-for-Hbc}
&\tailmeasurestar(\widetilde{H}_{\BC,p})=\canditheta\esp\left[\widetilde{H}_{\BC,p}(\bsZ)\right]
\\
&=\canditheta\esp\left[\sum_{i=1}^{\clusterlength(\bsZ)-1}| H(\bsZ) -  
H(\bsZ_{0,i-1})- H(\bsZ_{i,T_{\rm max}(\bsZ)})|^p\right]\;.\notag
\end{align}
Note also that $T^{(1)}(\bsZ)=0$, $T_{\rm max}(\bsZ)=\clusterlength(\bsZ)-1$. 
\item \Cref{theo:main-theorem} tells us that the difference between the disjoint and the sliding blocks statistics is of the order $O_P(\dhinterseq^{-1})$:
    \begin{align}\label{eq:conv-prob-expansion}
    \dhinterseq \left[\tedcluster(H)-\tedclustersl(H)\right]\convprob \tailmeasurestar(\widetilde{H}_{\IC}+\widetilde{H}_{\BC})\;.
    \end{align}
\item The expression on the \rhs\ of \eqref{eq:conv-prob-expansion} vanishes in several scenarios. In such the case 
    \begin{align}\label{eq:conv-prob-expansion-0}
    \dhinterseq \left[\tedcluster(H)-\tedclustersl(H)\right]\convprob 0\;.
    \end{align}
\begin{itemize}
\item Take $H(\bsx)=\sum_j \phi(x_j)$, where $\phi(x_j)=0$ whenever $|x_j|<1$. Here $\widetilde{H}_{\IC}=\widetilde{H}_{\BC,p}\equiv 0$ and hence  $\tailmeasurestar(\widetilde{H}_{\IC})=\tailmeasurestar(\widetilde{H}_{\BC})=0$. This is not surprising since the choice of the functional $H$ implies that the sliding and disjoint blocks estimator coincide (up to negligible boundary terms). 
\item 
Take $H=\clusterlength$. Here $\widetilde{H}_{\IC}=\widetilde{H}_{\BC,p}\equiv 0$ and hence  $\tailmeasurestar(\widetilde{H}_{\IC})=\tailmeasurestar(\widetilde{H}_{\BC})=0$.
\item Take $H=\ind{\bsx^*>1}$. We have $\tailmeasurestar(\widetilde{H}_{\IC})=\canditheta \esp[\clusterlength(Z)-1]=-\tailmeasurestar(\widetilde{H}_{\BC})$. 
\item If $\bsX$ is extremally independent, then $\tailmeasurestar(\widetilde{H}_{\IC})=\tailmeasurestar(\widetilde{H}_{\BC})=0$ for any $H$.    
\item Assume that $\bsY_j\not=0$ for any $j\geqslant 1$. This is for example the case of AR($p$) or ARCH($p$) process. 
    Then $\tailmeasurestar(\widetilde{H}_{\IC})\not =0$, $\tailmeasurestar(\widetilde{H}_{\BC})\not=0$, but \eqref{eq:conv-prob-expansion-0} holds. Indeed, here $T^{(i)}(\bsZ)=i-1$ and $\Delta T^{(i)}(\bsZ)=1$. 
    \end{itemize}
    \item However, we are unable to answer the question: Do we always have \eqref{eq:conv-prob-expansion-0}?
\end{itemize}
}
\end{remark}
\subsection{Small blocks scenario: piecewise stationary case}\label{sec:piecewise}
Consider the piecewise stationary time series as defined in \Cref{sec:dependence-assumptions}. 

A proof of central limit theorem for disjoint blocks statistics works as follows: we replace the original time series with
its piecewise stationary version and prove the limit theorem for the latter. 
We will show below that in the context of the expansion considered in the paper, this idea breaks down. That is, the blocks in the stationary time series cannot be treated as independent. 

Recall \eqref{eq:internal-as-functional}. Define
\begin{align*}
\IC^{\textrm{PS}}(H)=\sum_{j=2}^{m_n-1}\IC_j^{\textrm{PS}}(H)=\sum_{j=2}^{m_n-1}\widetilde{H}_{\IC}(\tepseq^{-1}\bsX_{(j-1)\dhinterseq+1,j\dhinterseq})\ind{A_j }\;.
\end{align*}
We formulate the following result.  We note that only the internal clusters contribute.

\begin{theorem}\label{thm:main-expansion-piecewise}
Assume that $\bsX$ is piecewise stationary.
Assume that~\hyperlink{SummabilityAC}{$\mathcal{S}^{(2\gamma+3)}(\dhinterseq, \tepseq)$} holds. 
Then 
\begin{align*}
\tedcluster(H)-\tedclustersl(H)=\frac{1}{n \dhinterseq w_n}\IC^{\textrm{PS}}(H)+\textcolor{black}{O_P\left(\dhinterseq w_n+\frac{\dhinterseq}{n}\right)}\;,
\end{align*}
where 
\begin{align*}
\frac{\IC^\textrm{PS}(H)}{nw_n}\convprob\tailmeasurestar(\widetilde{H}_{\IC})\;.
\end{align*} 
\end{theorem}
\begin{proof}
Since the  blocks are independent, \Cref{proposition:interior-cluster-means} and \eqref{eq:extremal-index} give 
\begin{align*}\frac{1}{n\dhinterseq w_n}\esp\left[\left|\IC(H)-\IC^{\textrm{PS}}(H)\right|\right]
&\leqslant 2\frac{1}{n\dhinterseq w_n}m_n \pr(A_1)\esp[|\widetilde{H}_{\IC}|(\tepseq^{-1}\bsX_{1,\dhinterseq})\ind{A_1}] \\
&=\frac{O(1)}{n\dhinterseq w_n}m_n \dhinterseq w_n \dhinterseq w_n =O(w_n)\;. 
\end{align*} 
Thus, in case of independent blocks, the indicators of small jumps, $\ind{A_j^c}$, can be dropped. 
Furthermore, \Cref{prop:boundary-cluster-means-main-piecewise,prop:boundary-cluster-means-main-1-piecewise} give 
\begin{align*}
\frac{1}{n\dhinterseq w_n}\esp[|\BC(H)|]=O(m_n)\frac{1}{n\dhinterseq w_n} \dhinterseq^3 w_n^2=O(\dhinterseq w_n)\;. 
\end{align*}
\end{proof}
\subsection{Large blocks scenario}\label{sec:large-blocks-MMA(1)}
For a simple treatment of the large blocks scenario, we consider the functional $H(\bsx)=\ind{\bsx^*>1}$ and a very special case of a time series. The result can be extended to a mixing case with a cumbersome mixing rates. Furthermore, unbounded functionals can be considered. 

However, we  decided to keep a very simple formulation of the large blocks scenario to illustrate a quantitative difference as compared to the small blocks scenario. From a statistical point of view, one can argue that large blocks are not relevant. 

To proceed, consider 
an MMA($1$) sequence $\bsX=\{X_j,j\in\Zset\}$ defined by
\begin{align*}
X_j=c_0 \xi_{j}\vee c_1 \xi_{j+1}\;,
\end{align*}
where $c_0,c_1$ are non-negative numbers and $\xi=\{\xi_j,j\in\Zset\}$ is a sequence of \iid\ non-negative, regularly varying random variables with the tail index $\alpha$. Non-negativity here is assumed for simplicity, to illustrate easily the main messages from the paper.
Then,
\begin{align}
&\lim_{x\to\infty}\frac{\pr(X_0>x)}{\pr(\xi_0>x)}=c_0^\alpha+c_1^\alpha\;, \ \ \lim_{x\to\infty}\pr(X_1>x\mid X_0>x)=\frac{(c_0\wedge c_{1})^\alpha}{c_0^\alpha+c_1^\alpha}\;.
\label{eq:MMA1-tail}
\end{align}
The tail process is $Y_1=\bernoulli (c_1/c_0)Y_0$, $Y_{-1}=(1-\bernoulli)(c_0/c_1)Y_0$, where $\bernoulli$ is a Bernoulli random variables with the mean $1/(1+(c_1/c_0)^\alpha)$, and $Y_0$ is a standard Pareto random variable with the tail index $\alpha$. Furthermore, $Y_j=0$ if $|j|\geqslant 2$. Hence
\begin{align*}
\pr(Y_1>1)=\pr(\bernoulli=1)\pr((c_1/c_0)Y_0>1)=\frac{c_0^\alpha}{c_0^\alpha+c_1^\alpha}\left(\left(\frac{c_1}{c_0}\right)^\alpha\wedge 1\right)=
\frac{(c_0\wedge c_1)^\alpha}{c_0^\alpha+c_1^\alpha}\;,
\end{align*}
(This can be also derived from \eqref{eq:MMA1-tail}, since $\pr(Y_1>1)=\lim_{x\to\infty}\pr(X_1>x\mid X_0>x)$.)
The candidate extremal index is then
\begin{align*}
\canditheta=\pr(\bsY_{1,\infty}^*\leqslant 1)=\pr(Y_1\leqslant 1)=1-\pr(Y_1>1)=1-\frac{(c_0\wedge c_1)^\alpha}{c_0^\alpha+c_1^\alpha}=\frac{(c_0\vee c_1)^\alpha}{c_0^\alpha+c_1^\alpha}\;.
\end{align*}
The next result shows that the rates in the asymptotic expansion for blocks statistics are different as compared to the small blocks scenario. Also, similarly to the piecewise stationary case, only the internal clusters contribute. In other words, in the large blocks scenario, the blocks behave as if they were independent.  
\begin{theorem}\label{thm:main-expansion-large}
Assume that $\bsX$ is the MMA(1) process and $H(\bsx)=\ind{\bsx^*>1}$. 
Assume $\dhinterseq^{3}w_n\to \infty$. 
Then 
\begin{align*}
\tedcluster(H)-\tedclustersl(H)=\frac{1}{n \dhinterseq w_n}\IC(H)+o_P(\dhinterseq w_n)+O_P\left(\frac{\dhinterseq^2w_n}{n}\right)\;,
\end{align*}
with 
\begin{align*}
\lim_{n\to\infty}\frac{\esp[\IC(H)]}{n\dhinterseq^2w_n^2}=\frac{1}{6}\canditheta^2\;. 
\end{align*}
\end{theorem}
\begin{proof}
\Cref{cor:moments-clusters-large-ic} gives the rate for the internal clusters statistics, \Cref{cor:moments-clusters-large-bc} yields that the internal clusters statistics is negligible, while the second part of \Cref{lem:remainder} (applied with $\gamma=0$) gives the rate for the remainder.
\end{proof}
In particular, 
\begin{align*}
    \left[\tedcluster(H)-\tedclustersl(H)\right]=O_P(\dhinterseq w_n)\;
    \end{align*}
and this rate cannot be improved in general.

\section{Proofs I - Extensions of vague convergence}\label{sec:technical-details}
\subsection{Consequences of the anticlustering conditions}
\subsubsection{First and last jump decompositions}\label{sec:first-and-last-jump-decomposition}
Recall the notation \eqref{eq:exceedance} for the exceedance times.
Set $t(i)=t_1{(i)}$, $i=1,\ldots,N_1$.
Note that for $i=1,\ldots,\dhinterseq$,
\begin{align*}
\{t(1)=i\}=\{\bsX_{1,i-1}^\ast\leqslant \tepseq, \norm{\bsX_{i}}>\tepseq\}\;, \ \ 
\{t{(N_1)}=i\}=\{\bsX_{i+1,\dhinterseq}^\ast\leqslant \tepseq, \norm{\bsX_{i}}>\tepseq\}\;, 
\end{align*}
with the convention $\bsX_{1,0}^\ast\equiv 0$
and $\bsX_{\dhinterseq+1,\dhinterseq}^\ast\equiv 0$. These types of decompositions will play a crucial role. 
\subsubsection{Conditional convergence of clusters}\label{sec:conditional-convergence-of-clusters}
As a consequence of \Cref{lem:tailprocesstozero}, for 
any bounded functional $H$ on $\ell_0(\Rset^d)$ that is almost
surely continuous with respect to the distribution of the tail process $\bsY$
and 
any $i_1,i_2,j_1,j_2,s\geqslant 0$,
\begin{align}
\lim_{n\to\infty}\esp\left[H(\tepseq^{-1}(\bsX_{-[\dhinterseq s]-j_1,-i_1},\bsX_{i_2,j_2+[\dhinterseq s]})) \mid \norm{\bsX_0}>\tepseq\right]&=
\esp\left[H(\bsY_{-\infty,-i_1},\bsY_{i_2,\infty})\right]\;.\label{eq:tool-conditional-conv-rn-twosided}
\end{align}
This type of convergence will be referred to as the \textit{conditional convergence of clusters}.
\subsubsection{Vague convergence of clusters}\label{sec:integral-convergence-of-clusters}
In view of the anticlustering condition \ref{eq:conditiondh}, the following argument will be used repeatedly (cf. \Cref{theo:cluster-RV} and the proof of Theorem 6.2.5 in \cite{kulik:soulier:2020}). Let $G$ be a bounded cluster functional such that $G(\bsx)=0$ whenever $\bsx^\ast\leqslant 1$. Then
\begin{align}\label{eq:convergence-series-anticlustering-1}
\frac{1}{\dhinterseq}\sum_{i=1}^{\dhinterseq}\esp\left[G(\bsX_{1-i,\dhinterseq-i}/\tepseq) \ind{\bsX_{1-i,-1}^\ast\leqslant \tepseq} \mid\norm{\bsX_0}>\tepseq\right]
=\int_{0}^1 g_n(s)\rmd s
\end{align}
with
$$
g_n(s)=\esp\left[G(\bsX_{1-[\dhinterseq s],\dhinterseq-[\dhinterseq s]}/\tepseq) \ind{\bsX_{1-[\dhinterseq s],-1}^\ast\leqslant \tepseq} \mid\norm{\bsX_0}>\tepseq\right]
$$
converging to $g(s):= \esp[G(\bsY)\ind{\bsY_{-\infty,-1}^\ast\leqslant 1}]$ for each $s\in (0,1)$. By the dominated convergence the limit of the expression in \eqref{eq:convergence-series-anticlustering-1} becomes
\begin{align*}
\esp[G(\bsY)\ind{\bsY_{-\infty,-1}^\ast\leqslant 1}]=\tailmeasurestar(G)\;.
\end{align*}

In the similar spirit, since
$g_n(s):=\pr(\bsX_{1-[\dhinterseq s],-1}^\ast\leqslant \tepseq)$ converges to $1$ for any $s>0$, we have for $\gamma\geqslant 0$,
\begin{align}\label{eq:P-sum-smallvalues}
\lim_{n\to\infty}\frac{1}{\dhinterseq}\sum_{i=1}^{\dhinterseq}\left(\frac{i}{\dhinterseq}\right)^\gamma\pr(\bsX_{1-i,-1}^\ast\leqslant \tepseq)=\lim_{n\to\infty}\int_0^1s^\gamma g_n(s)\rmd s=\frac{1}{\gamma+1}\;.
\end{align}
We will refer to this type of convergence as the \textit{vague convergence of clusters}.

\subsection{Extensions of vague convergence: internal clusters}\label{sec:internal-clusters}
In this section we consider a single block of size $\dhinterseq$.
\Cref{theo:cluster-RV} is valid for bounded functionals. The goal is to extend it to unbounded ones. 

In the small blocks scenario the uniform integrability holds and \Cref{theo:cluster-RV}  is still valid. See \Cref{lem:clusterlength-moments}. It fails in the large blocks scenario. See \Cref{lem:clusterlength-moments-large}. In particular, for the $\gamma$ moment of the cluster length, the dichotomous between the small and the large blocks scenario is related to whether $\dhinterseq^{\gamma+1}w_n\to 0$ or $\dhinterseq^{\gamma+1}w_n\to \infty$. 

All results are stated without a proof. See \cite{chen:kulik:2023a}.
\subsubsection{Moments of the cluster length - small blocks scenario}
\begin{lemma}\label{lem:clusterlength-moments}
Assume that 
\hyperlink{SummabilityAC}{$\mathcal{S}^{(\gamma+\delta)}(\dhinterseq, \tepseq)$} holds. 
Then for any $G\in\mch(\delta)$
\begin{align*}
   \lim_{n\to \infty} \frac{1}{\dhinterseq w_n} \esp \left[\clusterlength^\gamma(\bsX_{1,\dhinterseq}/\tepseq) G(\bsX_{1,\dhinterseq}/\tepseq)\ind{ A_1 }\right]
     = \tailmeasurestar(G \clusterlength^\gamma)=\canditheta\esp \left[G(\bsZ)\clusterlength^\gamma(\bsZ) \right]\;.
\end{align*}
\end{lemma}
\begin{remark}\label{rem:clusterlength-add-small-jumps-for-free}
{\rm 
We note that $\ind{A_1}$ in the statement of \Cref{lem:clusterlength-moments} can be replaced with $\ind{A_0^c\cap A_1\cap A_2^c}$. In other words, indicators of "no large jumps" can be easily added. 
}
\end{remark}
Furthermore, with help of \Cref{lem:clusterlength-moments}, we can obtain the tail bound on the cluster length.
\begin{corollary}\label{cor:clusterlength-tail-1}
Fix $a>0$. 
Assume that 
\hyperlink{SummabilityAC}{$\mathcal{S}^{(\gamma+\delta)}(\dhinterseq, \tepseq)$} holds. 
Then for any $G\in\mch(\delta)$,
\begin{align*}\
\esp[\ind{\clusterlength(\bsX_{1,\dhinterseq}/\tepseq)>a\dhinterseq}G(\bsX_{1,\dhinterseq}/\tepseq)]=O(w_n\dhinterseq^{1-\gamma})\;. 
\end{align*}
\end{corollary}
\begin{proof}
For any $\gamma>0$ we can write
\begin{align*}
\esp[\ind{\clusterlength(\bsX_{1,\dhinterseq}/\tepseq)>a\dhinterseq}G(\bsX_{1,\dhinterseq}/\tepseq)]
\leqslant 
\frac{1}{a^\gamma\dhinterseq^\gamma}\esp[\clusterlength^{\gamma}(\bsX_{1,\dhinterseq})G(\bsX_{1,\dhinterseq}/\tepseq)\ind{A_1}]\;. 
\end{align*}
Apply \Cref{lem:clusterlength-moments}. 
\end{proof}
\subsubsection{Moments of the cluster length - large blocks scenario}
We recall that \hyperlink{SummabilityAC}{$\mathcal{S}^{(\gamma)}(\dhinterseq, \tepseq)$} is almost equivalent to $\dhinterseq^{\gamma+1}w_n\to 0$. Thus, we ask what happens if the latter condition is violated.  
\begin{lemma}\label{lem:clusterlength-moments-large}
Assume that $\bsX$ is
stationary and $\ell$-dependent. 
Assume that \ref{eq:conditiondh} holds and $\dhinterseq^{\gamma+1}w_n\to \infty$. 
Then 
\begin{align*}
&\lim_{n\to\infty}\frac{\esp\left[\clusterlength_1^\gamma \ind{A_1}\right]}{\dhinterseq^{\gamma+2} w_n^2}=\frac{1}{(\gamma+1)(\gamma+2)}\canditheta^2\;.
\end{align*}
\end{lemma}
We obtain immediately the following counterpart to \Cref{cor:clusterlength-tail-1}.
\begin{corollary}\label{cor:clusterlength-tail-2}
Fix $a>0$. Assume that $\bsX$ is stationary and  
$\ell$-dependent.
Assume that \ref{eq:conditiondh} holds and $\dhinterseq^{\gamma+1}w_n\to \infty$. 
Then 
\begin{align*}
\pr(\clusterlength(\bsX_{1,\dhinterseq})>a\dhinterseq)=O(\dhinterseq^2w_n^2)\;. 
\end{align*}
\end{corollary}

\subsection{Extensions of vague convergence: boundary clusters}\label{sec:technical-boundary}

\Cref{sec:cluster-measure-convergence} deals with convergence of cluster functionals, conditionally on the event $A_1=\{\bsX_{1,\dhinterseq}^\ast>\tepseq\}$. The present section deals with convergence of cluster functionals, conditionally on $A_1\cap A_2=\{\bsX_{1,\dhinterseq}^\ast>\tepseq,\bsX_{\dhinterseq+1,2\dhinterseq}^\ast>\tepseq\}$. We note the dichotomous behaviour, depending on the blocks size. 
\subsubsection{Summary of the results}\label{sec:summary-boundary}
\begin{itemize}
\item \Cref{prop:pa1-cap-pa2-bounded} and \Cref{lem:pa1-cap-pa2-precise} give the rates of convergence for $\pr(A_1\cap A_2)$ in the small blocks scenario. The rate is $w_n$, as opposite to the rate $\dhinterseq w_n$ for $\pr(A_1\cap A_2)$.   That is, a finite number of $\pr(\norm{\bsX_{i_1}}>\tepseq,\norm{\bsX_{i_2}}>\tepseq)$, $i_1\in I_1,i_2\in I_2$, plays a role (even in the simple case of 1-dependence!). The required assumption on the blocks size is $\dhinterseq^2w_n\to 0$.
\item \Cref{prop:pa1-cap-pa2-unbounded-uniformintegrability,prop:cluster-length-UniformIntegrability-boundary-clusters} extend the previous results to unbounded functionals with a particular focus on cluster length. We note that the small blocks condition for the finite $\gamma$-moment of the cluster length is $\dhinterseq^{\gamma+2}w_n\to 0$. 
    This should be compared to the situation of 
\Cref{lem:clusterlength-moments}. There, the small blocks condition for the $\gamma$ moment of the cluster length is $\dhinterseq^{\gamma+1}w_n\to 0$. In other words, small block condition has a different meaning in the case of internal and boundary clusters. 
\item 
\Cref{lem:pa1-cap-pa2-precise-large} deals with the large blocks scenario and bounded functionals - the blocks behave as if they were independent. The large blocks conditions is again $\dhinterseq^{\gamma+2}w_n\to \infty$ (with $\gamma=0$ for a bounded $H$). 
\item 
\Cref{lem:boundary-joint-cluster-length-MMA1} extends \Cref{lem:pa1-cap-pa2-precise-large} to unbounded cluster functionals. 
To control $\gamma$ moments of the cluster length, the large blocks condition reads $\dhinterseq^{\gamma+2}w_n\to \infty$.   The result should be compared to \Cref{lem:clusterlength-moments-large}. In both cases of boundary and internal clusters, the jump locations behave as if they were independent. However, note again different large blocks condition: $\dhinterseq^{\gamma+1}w_n\to \infty$ for internal clusters and $\dhinterseq^{\gamma+2}w_n\to\infty$ for boundary clusters. 
\end{itemize}
The results of this section are new. All proofs are given in \Cref{sec:proof-for-boundary}. 
\subsubsection{Rates for boundary clusters - small blocks scenario}
\begin{proposition}\label{prop:pa1-cap-pa2-bounded}
Assume that \hyperlink{SummabilityAC}{$\mathcal{S}^{(1)}(\dhinterseq, \tepseq)$} holds. 
For any bounded $H\in \mch$ we have 
\begin{align}\label{eq:pa1-cap-pa2-bounded}
\lim_{n\to\infty}\frac{\esp[H(\tepseq^{-1}\bsX_{1,2\dhinterseq})\ind{A_1\cap A_2}]}{w_n}=
\canditheta
\esp\left[ (\clusterlength(\bsZ) - 1) H(\bsZ) \right]\;.
\end{align}
\end{proposition}
\begin{corollary}\label{lem:pa1-cap-pa2-precise}
Assume that \hyperlink{SummabilityAC}{$\mathcal{S}^{(1)}(\dhinterseq, \tepseq)$} holds.
Then
\begin{align*}
\lim_{n\to\infty}\frac{\pr(A_1\cap A_2)}{w_n}=\canditheta
\esp\left[ (\clusterlength(\bsZ) - 1)\right]\;.
\end{align*}
\end{corollary}

The next result extends \Cref{prop:pa1-cap-pa2-bounded} to unbounded functionals if the appropriate uniform integrability condition holds. 
\begin{proposition}
\label{prop:pa1-cap-pa2-unbounded}
 Assume that \hyperlink{SummabilityAC}{$\mathcal{S}^{(1)}(\dhinterseq, \tepseq)$} holds. Let $H\in\mch$.  Assume that 
  \begin{align}\label{eq:pa1-cap-pa2-uniform-integrability}
  \lim_{\ell\to\infty}\limsup_{n\to\infty}\frac{\esp[ |H|(\tepseq^{-1}\bsX_{1,2\dhinterseq})\ind{A_1\cap A_2}\ind{|H|(\tepseq^{-1}\bsX_{1,2\dhinterseq})>\ell}]} {w_n}=0\;. 
  \end{align}
  Then \eqref{eq:pa1-cap-pa2-bounded} holds. 
\end{proposition}
\begin{proposition}\label{prop:pa1-cap-pa2-unbounded-uniformintegrability}
Assume that \hyperlink{SummabilityAC}{$\mathcal{S}^{(\gamma+1)}(\dhinterseq, \tepseq)$} holds. Let $H\in\mch(\gamma)$. Then the uniform integrability condition \eqref{eq:pa1-cap-pa2-uniform-integrability} holds.   
\end{proposition}
\paragraph{Joint cluster length.}
Recall that notation \eqref{eq:boundary-full-cluster-length} for the joint cluster length $\mathcal{L}_{j,j+1} = t_{j+1}{(N_{j+1})} - t_{j}(1) + 1$ as well as \eqref{eq:boundary-full-jump-times} for the jump times in the adjacent blocks. 
The following result is a counterpart to
\Cref{lem:clusterlength-moments} and follows directly from \Cref{prop:pa1-cap-pa2-unbounded}. This time we study the behaviour of the cluster length $\clusterlength_{1,2}=\clusterlength(\tepseq^{-1}\bsX_{1,2\dhinterseq})$ defined as a function of two adjacent blocks. Note the rate change as compared to \Cref{lem:clusterlength-moments}, primarily due to the rates in 
\Cref{lem:pa1-cap-pa2-precise} in the small blocks scenario. 
\begin{corollary}
    \label{prop:cluster-length-UniformIntegrability-boundary-clusters}
Assume that \hyperlink{SummabilityAC}{$\mathcal{S}^{(\gamma+\delta+1)}(\dhinterseq, \tepseq)$}  holds. For any $G\in \mch(\delta)$ we have, 
\begin{align*}
   \lim_{n\to \infty} \frac{\esp \left[\clusterlength^\gamma(\tepseq^{-1}\bsX_{1,2\dhinterseq})G(\tepseq^{-1}\bsX_{1,2\dhinterseq})\ind{ A_1\cap A_2 } \right]}{w_n}
     = \canditheta\esp\left[(\clusterlength(\bsZ) - 1 )\clusterlength^\gamma(\bsZ)G(\bsZ) \right]\;.
\end{align*}
\end{corollary}

We also state the next corollary to \Cref{prop:pa1-cap-pa2-unbounded}, for a future use.  Below, 
\eqref{eq;convergence-boundary-cluster-H-full} is obvious. However, the remaining statements require a short argument.   
\begin{corollary}
    \label{prop:general-function-UniformIntegrability-boundary-clusters}
    Assume that $H\in\mch(\gamma)$. 
Under the condition~
\hyperlink{SummabilityAC}{$\mathcal{S}^{(\gamma+1)}(\dhinterseq, \tepseq)$}, we have 
\begin{subequations}
\begin{align}
     &\lim_{n\to\infty}\frac{\esp\left[ H \left(
        \tepseq^{-1}\bsX_{1,2\dhinterseq}\right)
      \ind{ A_1\cap A_2 } \right] } {w_n}=
      \canditheta\esp\left[ (\clusterlength(\bsZ)-1)H(\bsZ)\right]\;,
      \label{eq;convergence-boundary-cluster-H-full} \\
      & \lim_{n\to\infty}\frac{ \esp \left[ H \left(
         \tepseq^{-1}\bsX_{1,\dhinterseq} \right)
     \ind{ A_1\cap A_2 } \right] } {w_n}
      =\canditheta\esp  \left[\sum_{j=1}^{\clusterlength(\bsZ)-1} H(\bsZ_{0,j-1})   \right]\;,
      \label{eq;joint-convergence-boundary-cluster;H-front} \\
      & \lim_{n\to\infty}\frac{ \esp\left[ H \left(
        \tepseq^{-1} \bsX_{\dhinterseq+1,2\dhinterseq} \right)
     \ind{ A_1\cap  A_2 } \right] } {w_n}
      = \canditheta\esp\left[\sum_{j=1}^{\clusterlength(\bsZ)-1} H(\bsZ_{j,\infty})   \right]\;.
      \label{eq;joint-convergence-boundary-cluster;H-back}
\end{align}
\end{subequations}
\end{corollary}
In the spirit of \Cref{cor:clusterlength-tail-1}, we obtain the tail asymptotic for the joint cluster length. Note that $\clusterlength_{1,2}\geqslant \dhinterseq$ implies $\ind{A_1\cap A_2}=1$.  \begin{corollary}\label{cor:clusterlength-tail}
Assume that \hyperlink{SummabilityAC}{$\mathcal{S}^{(\gamma+\delta+1)}(\dhinterseq, \tepseq)$}  holds. For any $G\in \mch(\delta)$ we have,
\begin{align*}
\esp\left[\ind{\clusterlength(\tepseq^{-1}\bsX_{1,2\dhinterseq})\geqslant \dhinterseq}G(\tepseq^{-1}\bsX_{1,2\dhinterseq}) \right]=O(w_n\dhinterseq^{-\gamma})\;. 
\end{align*}
\end{corollary}

\subsubsection{Rates for boundary clusters - large blocks scenario}
\Cref{lem:pa1-cap-pa2-precise} gives the rate for $\pr(A_1\cap A_2)$ in the small blocks scenario. The next lemma deals with large blocks - the blocks behave as if they were independent (recall that $\pr(A_1)\sim \canditheta \dhinterseq w_n$). Recall also that bounded $H$ corresponds to $\mch(0)$, so the large blocks condition reads $\dhinterseq^{\gamma+2}w_n\to 0$. 
\begin{lemma}\label{lem:pa1-cap-pa2-precise-large}
Assume that $\bsX$ is stationary and $\ell$-dependent.
Assume that \ref{eq:conditiondh} holds and $\dhinterseq^{2}w_n\to \infty$.  
Then for any bounded $H\in \mch$,
\begin{align}\label{eq:mdep-P(a1-cap-a2)-largeblocks}
\limsup_{n\to\infty}\frac{\esp[H(\tepseq^{-1}\bsX_{1,2\dhinterseq})\ind{A_1\cap A_2}]}{\dhinterseq^2 w_n^2}<\infty\;.
\end{align}
\end{lemma}

\begin{remark}{\rm
We will show in \Cref{sec:MMA(1)} that the rate 
in \eqref{eq:mdep-P(a1-cap-a2)-largeblocks} is sharp. 
}
\end{remark}
\subsubsection{Boundary clusters for the toy example - MMA(1)}\label{sec:MMA(1)}
Boundary clusters seem to be harder to handle with in a large blocks scenario. Hence, we illustrate the precise rates for the toy example MMA(1) introduced in \Cref{sec:large-blocks-MMA(1)}.  Proofs are given in \Cref{sec:proof-for-boundary}.

In \Cref{lem:pa1-cap-pa2-precise} we obtained a precise asymtptotics for $\pr(A_1\cap A_2)$ in the small block case, while \Cref{lem:pa1-cap-pa2-precise-large} gives a bound in the large block case. Here, we provide additionally the precise asymptotics in the large blocks case, showing that the rate in
\eqref{eq:mdep-P(a1-cap-a2)-largeblocks} is sharp. The proof also shows precisely where does the small blocks rate $w_n$ come from.

\begin{lemma}\label{lem:MMA1-P(a1-cap-a2)}
Consider the MMA(1) process. Then
\begin{itemize}
\item[{\rm (i)}] If $\dhinterseq^2 w_n\to 0$, then
\begin{align*}
\lim_{n\to\infty}\frac{1}{w_n}\pr(A_1\cap A_2)=\frac{(c_0\wedge c_1)^\alpha}{c_0^\alpha+c_1^\alpha}=\pr(Y_1>1)\;.
\end{align*}
\item[{\rm (ii)}] If $\dhinterseq^2 w_n\to \infty$, then
\begin{align*}
\lim_{n\to\infty}\frac{1}{\dhinterseq^2w_n^2}\pr(A_1\cap A_2)=\left(\frac{(c_0\vee c_1)^\alpha}{c_0^\alpha+c_1^\alpha}\right)^2=\canditheta^2\;.
\end{align*}
\end{itemize}
\end{lemma}
The next result gives the precise asymptotics for the cluster length based on two blocks. In the large blocks scenario, the result should be compared to \Cref{lem:clusterlength-moments-large}. In the current situation as well as in the situation discussed in the aforementioned lemma, the jump locations behave as if they were independent. Note also different small and large blocks condition, when comparing \Cref{lem:clusterlength-moments-large} to the lemma below. 
\begin{lemma}\label{lem:boundary-joint-cluster-length-MMA1}
Consider the MMA(1) process. 
\begin{itemize}
\item[{\rm (i)}]
If $\dhinterseq^{\gamma+2}w_n\to 0$, then 
\begin{align*}
\lim_{n\to\infty}\frac{1}{w_n}\esp\left[\left(t_{2}{(N_2)}-t_1{(1)}\right)^\gamma\ind{A_1 \cap A_2^c}\right]=\frac{(c_0\wedge c_1)^\alpha}{c_0^\alpha+c_1^\alpha}=\pr(Y_1>1)\;. 
\end{align*}
\item[{\rm (ii)}] If $\dhinterseq^{\gamma+2}w_n\to \infty$, then 
\begin{align*}
\lim_{n\to\infty}\frac{1}{\dhinterseq^{\gamma+2}w_n^2}\esp\left[\left(t_{2}{(N_2)}-t_1{(1)}\right)^\gamma\ind{A_1 \cap A_2^c}\right]=\frac{2^{\gamma+2}-1}{(\gamma+1)(\gamma+2)}\canditheta^2\;. 
\end{align*}
\end{itemize}
\end{lemma}
Finally, we will need a specific lemma for the boundary clusters. 
\begin{lemma}\label{lem:mma1-last-and-first-jump-consecutive-blocks-expectation}
Consider the MMA(1) process. Then
\begin{align*}
\lim_{n\to\infty}\frac{1}{\dhinterseq^3 w_n^2}\esp\left[\ind{A_1\cap A_2}\left((t_2{(1)}-t_{1}{(N_1)})-\dhinterseq\right)_+\right]=\frac{1}{6}\canditheta^2\;.
\end{align*}
\end{lemma}

\subsubsection{Proofs for boundary clusters}\label{sec:proof-for-boundary}
\begin{proof}[Proof of \Cref{prop:pa1-cap-pa2-bounded}]
Fix an integer $\ell$. For any $n$ such that $\dhinterseq > \ell$, it follows that
\begin{align}\label{eq:boundary-clusters-asymptotics}
    &\esp\left[ H(\tepseq^{-1}\bsX_{1,2\dhinterseq}) \ind{ A_1\cap A_2 }  \right] 
    =\sum_{i_1=1}^{\dhinterseq}\sum_{i_2=\dhinterseq+1}^{2\dhinterseq}
    \esp\left[H(\tepseq^{-1}\bsX_{i_1,i_2})\ind{t_{1}(1) = i_1, t_{2}{(N_2)} = i_2}\right]\nonumber\\
    &
    =\sum_{ \substack{  \dhinterseq - \ell + 1 \leqslant i_1 \leqslant \dhinterseq \\  \dhinterseq+1 \leqslant i_2 \leqslant \dhinterseq+\ell} }
    \esp\left[H(\tepseq^{-1}\bsX_{i_1,i_2})\ind{t_{1}(1) = i_1, t_{2}{(N_2)} = i_2}\right] \nonumber\\
     &\phantom{=} +
    \sum_{ \text{other pairs } (i_1,i_2) }
     \esp\left[H(\tepseq^{-1}\bsX_{i_1,i_2})\ind{t_{1}(1) = i_1, t_{2}{(N_2)} = i_2}\right] = : J(H,\dhinterseq;\ell)+ \widetilde{J}(H,\dhinterseq;\ell)\;.
\end{align}
For $J(H,\dhinterseq;\ell)$, we have first by shifting by $i_1$, then by changing the variables $i_1$ into $i_1-\dhinterseq$ and $i_2$ into $i_2-\dhinterseq$:
\begin{align*}
    & J(H,\dhinterseq;\ell) \\
    &=\sum_{ \substack{  \dhinterseq - \ell + 1 \leqslant i_1 \leqslant \dhinterseq \\  \dhinterseq+1 \leqslant i_2 \leqslant \dhinterseq+\ell} }\esp\left[H(\tepseq^{-1}\bsX_{i_1,i_2})
    \ind{\bsX^\ast_{ 1,i_1 - 1 } \leqslant \tepseq , |\bsX_{i_1}| > \tepseq,|\bsX_{i_2}| > \tepseq , \bsX^\ast_{i_2+1, 2\dhinterseq} \leqslant \tepseq}\right]  \\
    &= \sum_{ \substack{  - \ell + 1 \leqslant i_1 \leqslant 0 \\  1 \leqslant i_2 \leqslant \ell} }
    \esp\left[H(\tepseq^{-1}\bsX_{0,i_2-i_1}) 
    \ind {\bsX^\ast_{ 1-\dhinterseq - i_1, - 1 } \leqslant \tepseq , |\bsX_0|>\tepseq,|\bsX_{i_2 - i_1}| > \tepseq, \bsX^\ast_{i_2 - i_1 + 1, \dhinterseq - i_1} \leqslant \tepseq} \right]  \\
    &=w_n\sum_{\substack{- \ell + 1 \leqslant j \leqslant 0 \\1-j\leqslant i\leqslant \ell-j} } 
     \esp\left[H(\tepseq^{-1}\bsX_{0,i})\ind{\bsX^\ast_{1-\dhinterseq-j,-1}\leqslant \tepseq,\norm{\bsX_{i}}>\tepseq,\bsX^\ast_{i+1,\dhinterseq-j}\leqslant \tepseq}\mid \norm{\bsX_0}>\tepseq\right]\;.
\end{align*}
For each $j$, using the conditional convergence of clusters (cf.~\eqref{eq:tool-conditional-conv-rn-twosided}) and the process $\bsZ$, 
\begin{align*}
&\lim_{n\to\infty}\esp\left[H(\tepseq^{-1}\bsX_{0,i})\ind{\bsX^\ast_{1-\dhinterseq-j,-1}\leqslant \tepseq,\norm{\bsX_{i}}>\tepseq,\bsX^\ast_{i+1,\dhinterseq-j}\leqslant \tepseq}\mid \norm{\bsX_0}>\tepseq\right]\\
&= \esp\left[H(\bsY_{0,i})\ind{\bsY^\ast_{-\infty,-1}\leqslant 1,\norm{\bsY_i}>1,\bsY^\ast_{i+1,\infty}\leqslant 1}\right]=\canditheta
\esp\left[H(\bsZ_{0,i})
 \ind{\clusterlength(\bsZ)= i +1}\right]
\end{align*}
and hence 
\begin{align*}
    &\lim_{n\to\infty}\frac{J(H,\dhinterseq;\ell)}{w_n}\\
    &= \canditheta\sum_{i = 1}^{\ell} i  
\esp\left[H(\bsZ_{0,i})
 \ind{\clusterlength(\bsZ)= i +1}\right] +  \canditheta\sum_{i = \ell+1}^{2\ell-1}(2\ell-i)\esp\left[H(\bsZ_{0,i})
 \ind{\clusterlength(\bsZ)= i +1}\right]\\
 & =:J_1(\ell)+J_2(\ell) \; .
\end{align*}
We note that 
\begin{align*}
\lim_{\ell\to\infty} J_1(\ell)=\canditheta
\esp\left[ (\clusterlength(\bsZ) - 1) H(\bsZ_{0,\clusterlength(\bsZ)-1}) \right]=
\canditheta
\esp\left[ (\clusterlength(\bsZ) - 1) H(\bsZ) \right]
\end{align*}
and since $H$ is bounded we have 
\begin{align*}
\lim_{\ell\to\infty}J_2(\ell)&\leqslant \canditheta \|H\|\lim_{\ell\to\infty}
\sum_{i=\ell}^\infty i\pr \left(\clusterlength(\bsZ)= i   \right)
=\|H\|\lim_{\ell\to\infty}
\sum_{i=\ell}^\infty i\pr \left(\clusterlength(\bsY)= i,\bsY_{-\infty,-1}^\ast\leqslant 1   \right)\\
&\leqslant  \|H\|\lim_{\ell\to\infty}
\sum_{i=\ell}^\infty  i\pr \left(\norm{\bsY_i}>1   \right)
=0
\end{align*}
on account of \hyperlink{SummabilityAC}{$\mathcal{S}^{(1)}(\dhinterseq, \tepseq)$}; cf. the first part of
\Cref{rem:rn2gamma}. 

For $\widetilde{J}(H,\dhinterseq;\ell)$, we note that
\begin{align*}
    \widetilde{J}(H,\dhinterseq;\ell)
    \leqslant &  \|H\|\sum_{ i=\ell  } ^{2\dhinterseq-1} i
    \pr\left( |\bsX_0 | > \tepseq, |\bsX_i| > \tepseq \right)
\end{align*}
and hence due to the condition~\hyperlink{SummabilityAC}{$\mathcal{S}^{(1)}(\dhinterseq, \tepseq)$},
\begin{align*}
\lim_{\ell\to \infty} \lim_{n\to \infty} \frac{\widetilde{J}(H,\dhinterseq;\ell)}{w_n} = 0\;.
\end{align*}
\end{proof}
\begin{proof}[Proof of \Cref{lem:pa1-cap-pa2-precise-large}]
We proceed in a slightly different way as compared to \Cref{prop:pa1-cap-pa2-bounded}. 
We use the last jump decomposition in the first block, then the first jump decomposition in the second block, and combine this with stationarity to obtain
\begin{align}
    \pr\left( A_1 \cap A_2   \right)
   &= \sum_{j_1=1}^{\dhinterseq} \sum_{j_2=\dhinterseq+1} ^ {2\dhinterseq} \pr\left( \norm{\bsX_0} > \tepseq, \bsX_{1,j_2 - j_1 -1}^* \leqslant \tepseq,  \norm{\bsX_{j_2 - j_1}} > \tepseq \right)\;. \label{eq:pr-a1-a2} \\
&\leqslant w_n\sum_{i=1}^{2\dhinterseq-1} \left(i \wedge (2\dhinterseq-i)\right) \pr \left( \tau_n = i  \mid  |\bsX_0| > \tepseq \right)\;,  \nonumber
\end{align}
where $\tau_n$ is the
random variable
\begin{align*}
 \tau_n := \inf \{  t \geqslant 1: |\bsX_t| > \tepseq  \}\;.
\end{align*}
Fix an integer $\ell\geqslant 1$ (we can assume that $\ell\leqslant \dhinterseq$) and split
\begin{align*}
\pr(A_1\cap A_2)&\leqslant w_n\sum_{i=1}^{\ell} i\cdot  \pr \left( \tau_n = i  \mid  |\bsX_0| > \tepseq \right)+
w_n\sum_{i=\ell+1}^{\dhinterseq} i\cdot  \pr \left( \tau_n = i  \mid  |\bsX_0| > \tepseq \right)\\
&\phantom{=}+w_n\sum_{i=\dhinterseq+1}^{2\dhinterseq} (2\dhinterseq-i) \pr \left( \tau_n = i \mid |\bsX_0| > \tepseq \right)=:J_1(\ell)+J_2(\ell,\dhinterseq)+J_3(\dhinterseq)\;.
\end{align*}
For each finite $i\ge 1$ we have
\begin{align*}
\lim_{n\to\infty} \pr \left( \tau_n = i  \mid |\bsX_0| > \tepseq \right)=
\pr(\bsY_{1,i-1}^*\leqslant 1,|\bsY_i|>1)\;
\end{align*}
and hence $J_1(\ell)=O(w_n)$. 

In case of $\ell$-dependence we have immediately
\begin{align*}
J_2(\ell,\dhinterseq)+J_3(\dhinterseq)\leqslant w_n\left\{  \sum_{i=\ell+1}^{\dhinterseq} i+\sum_{i=\dhinterseq+1}^{2\dhinterseq}(2\dhinterseq-i) \right\}\pr(\norm{\bsX_{i}}>\tepseq\mid \norm{\bsX_0}>\tepseq)=O(\dhinterseq^2 w_n^2)\;
\end{align*}
and the latter rate dominates $w_n$ in the large blocks scenario.
\end{proof}
\begin{proof}[Proof of \Cref{prop:pa1-cap-pa2-unbounded}]
We mimic the proof of \Cref{prop:pa1-cap-pa2-bounded}. Recall the term $J(H,\dhinterseq;\ell)$ in \eqref{eq:boundary-clusters-asymptotics}. On the event $\{t_{2}{(N_2)}-t_1{(1)}>\ell/2\}\subseteq\{\clusterlength_{1,2}>\ell/2\}$, we have $J(\dhinterseq;\ell/2)\equiv0$. 

Since $H\in \mch(\gamma)$, we bound $|H|(\tepseq^{-1}\bsX_{1,\dhinterseq})\leqslant C_H(t_{N_1}-t_1)^\gamma$.  Assume \withoutlog\ that $C_H=1$. 

Recall again $\widetilde{J}(H,\dhinterseq;\ell)$ defined in \eqref{eq:boundary-clusters-asymptotics}. 
It follows that
\begin{align*}
    &\esp\left[ |H|(\tepseq^{-1}\bsX_{1,2\dhinterseq})\ind{|H|(\tepseq^{-1}\bsX_{1,2\dhinterseq})>\ell/2} \ind{ A_1\cap A_2 }  \right]\\
    &\leqslant
    \esp\left[ \clusterlength_{1,2}^\gamma\ind{\clusterlength_{1,2}>(\ell/2)^{\gamma/2}} \ind{ A_1\cap A_2 }  \right] \leqslant\widetilde{J}(\clusterlength^\gamma,\dhinterseq;(\ell/2)^{\gamma/2})\;.
\end{align*}
We have 
\begin{align*}
    \widetilde{J}(H,\dhinterseq;(\ell/2)^{\gamma/2})
    \leqslant &  \sum_{ i=\ell  } ^{2\dhinterseq-1} i^{\gamma+1}
    \pr\left( |\bsX_0 | > \tepseq, |\bsX_i| > \tepseq \right)
\end{align*}
and hence due to the condition~\hyperlink{SummabilityAC}{$\mathcal{S}^{(\gamma+1)}(\dhinterseq, \tepseq)$},
\begin{align*}
\lim_{\ell\to \infty} \lim_{n\to \infty} \frac{\widetilde{J}(H,\dhinterseq;(\ell/2)^{\gamma/2})}{w_n} = 0\;.
\end{align*}
The proof of the uniform integrability is finished.
\end{proof}
\begin{proof}[Proof of  \Cref{prop:general-function-UniformIntegrability-boundary-clusters}] 
We will only show \eqref{eq;joint-convergence-boundary-cluster;H-front}. 
We assume that $H$ is bounded. The unbounded case follows from \Cref{prop:pa1-cap-pa2-unbounded-uniformintegrability}. 

For a bounded case, we proceed in a slightly  
     different manner as compared to \Cref{prop:pa1-cap-pa2-bounded}. 
    Fix an integer $\ell>0$. For any $n$ such that $\dhinterseq > \ell$, it follows that
    \begin{align*}
    &\esp \left[ H \left(\tepseq^{-1}\bsX_{1,\dhinterseq} \right)
     \ind{ A_1\cap A_2 } \right] =  \sum_{ i_1 = \dhinterseq +1 - \ell} ^ {\dhinterseq}
     \esp \bigg[
     \ind{t_{1}(1) = i_1, \clusterlength_{1,2} \geqslant \dhinterseq + 2 - i_1}
      H( \tepseq^{-1} \bsX_{i_1,\dhinterseq}  )
     \bigg] \\
     & + \sum_{ i_1 = 1 } ^ {\dhinterseq-\ell}
     \esp \bigg[
     \ind{t_{1}(1) = i_1, \clusterlength_{1,2} \geqslant \dhinterseq + 2 - i_1}
      H( \tepseq^{-1} \bsX_{i_1,\dhinterseq}  )
     \bigg] = :  J(\dhinterseq;\ell)+\widetilde{J}(\dhinterseq;\ell)\;.
    \end{align*}
We have 
\begin{align*}
J(\dhinterseq;\ell)&=w_n\sum_{i_1=\dhinterseq+1-\ell}^{\dhinterseq}\esp[\ind{\bsX_{1-i_1,-1}^\ast\leqslant \tepseq,\clusterlength_{1,2} \geqslant \dhinterseq + 2 - i_1}H(\tepseq^{-1}\bsX_{0,\dhinterseq-i_1})\mid \norm{\bsX_0}>\tepseq]\\
&=w_n\sum_{i=1}^\ell \esp[\ind{\bsX^\ast_{1-(\dhinterseq+i-\ell),-1}\leqslant\tepseq,\clusterlength_{1,2}\geqslant 2-i+\ell}
H(\tepseq^{-1}\bsX_{0,\ell-i})\mid \norm{\bsX_0}>\tepseq]\;. 
\end{align*}
Since $H$ is bounded, the conditional convergence of clusters argument yields
\begin{align*}
&\lim_{\ell\to\infty}\lim_{n\to\infty}\frac{J(\dhinterseq;\ell)}{w_n}=\lim_{\ell\to\infty}\sum_{i=1}^\ell
\esp[\ind{\bsY_{-\infty,-1}^\ast\leqslant 1,\clusterlength(\bsY)\geqslant 2-i+\ell}H(\bsY_{0,\ell-i})]\\
&=\lim_{\ell\to\infty}\sum_{j=2}^{\ell+1}\esp[\ind{\bsY_{-\infty,-1}^\ast\leqslant 1,\clusterlength(\bsY)\geqslant j}H(\bsY_{0,j-2})]
\\&=\sum_{j=2}^\infty \esp[\ind{\bsY_{-\infty,-1}^\ast\leqslant 1,\clusterlength(\bsY)\geqslant j}H(\bsY_{0,j-2})]\\
&=
\esp\left[\ind{\bsY_{-\infty,-1}^\ast\leqslant 1}\sum_{j=2}^{\clusterlength(\bsY)}H(\bsY_{0,j-2})\right]=\canditheta
\esp\left[\sum_{j=2}^{\clusterlength(\bsZ)}H(\bsZ_{0,j-2})\right]
\;.
\end{align*}
Next, 
    \begin{align*}
     \widetilde{J}(\dinterseq;\ell)\leqslant \|H\|
        \sum_{i = \ell}^{2\dhinterseq} i \cdot
        \pr( | \bsX_0 | > \tepseq , | \bsX_i | > \tepseq )
    \end{align*}
    and the term vanishes on account of the anticlustering condition. 
\end{proof}


\begin{proof}[Proof of \Cref{lem:MMA1-P(a1-cap-a2)}]
As in the proof of  \Cref{prop:pa1-cap-pa2-bounded}, we proceed with the last jump decomposition in the first block, followed by the first jump decomposition in the second block. We have 
\begin{align*}
\pr(A_1\cap A_2)&= \sum_{j_1=1}^{\dhinterseq}\sum_{j_2=\dhinterseq+1}^{2\dhinterseq}\pr(t_1(1)=j_1,t_2(N_2)=j_2)\\
&= \sum_{j_1=1}^{\dhinterseq}\sum_{j_2=\dhinterseq+1}^{2\dhinterseq}\pr(\bsX_{1,j_1-1}^*\leq \tepseq,X_{j_1}>\tepseq,X_{j_2}>\tepseq,\bsX_{j_2+1,2\dhinterseq}^*\leq \tepseq)\;. 
\end{align*}
By 1-dependence we have
\begin{align}\label{eq:proof-for-mma1-1}
\pr(A_1\cap A_2)
&= \sum_{j_1=1}^{\dhinterseq}\sum_{j_2=\dhinterseq+1}^{2\dhinterseq}\pr(\bsX_{1,j_1-1}^*\leq \tepseq,X_{j_1}>\tepseq)\pr(X_{j_2}>\tepseq,\bsX_{j_2+1,2\dhinterseq}^*\leq \tepseq)\nonumber\\
&\phantom{=}-w_n^2\pr(\bsX_{1-\dhinterseq,-1}^*\leq \tepseq \mid X_{0}>\tepseq)\pr(\bsX_{1,\dhinterseq-1}^*\leq \tepseq\mid X_{0}>\tepseq)\nonumber\\
&\phantom{=}+\pr(\bsX_{1,\dhinterseq-1}^*\leq \tepseq,X_{\dhinterseq}>\tepseq,X_{\dhinterseq+1}>\tepseq,\bsX_{\dhinterseq+2,2\dhinterseq}^*\leq \tepseq)\nonumber\\
&=:J_1(\dhinterseq)-J_2(\dhinterseq)+J_3(\dhinterseq)\;.
\end{align}
The vague convergence of clusters gives $J_1(\dhinterseq)\sim \canditheta^2 \dhinterseq^2 w_n^2$. 
Next, the conditional convergence of clusters gives $J_2(\dhinterseq)\sim \canditheta^2 w_n^2$ as $n\to\infty$. Finally,
\begin{align}
&J_3(\dhinterseq)\sim \pr(c_1\xi_{\dhinterseq}\leqslant \tepseq,c_0\xi_{\dhinterseq}\vee c_1\xi_{\dhinterseq+1}>\tepseq,c_0\xi_{\dhinterseq+1}\vee c_1\xi_{\dhinterseq+2}>\tepseq,c_0\xi_{\dhinterseq+2}\leqslant \tepseq)\nonumber\\
&\sim \pr((c_0\wedge c_1)\xi_{\dhinterseq+1}>\tepseq) \sim w_n \pr(Y_1>1)\;.\label{eq:j3-mma1}
\end{align}
In summary, the term $J_2(\dhinterseq)$ does not contribute, while the terms $J_1(\dhinterseq)$ and $J_3(\dhinterseq)$ dominate in the large and the small blocks scenario, respectively. 
\end{proof}
\begin{proof}[Proof of \Cref{lem:boundary-joint-cluster-length-MMA1}]
We use the notation from \eqref{eq:proof-for-mma1-1} to get
\begin{align*}
   & \esp\left[\left(t_2{(N_2)}-t_1{(1)}\right)^\gamma\ind{A_1^c \cap A_2^c}\right]\\
   &
   = \sum_{j_1=1}^{\dhinterseq} \sum_{j_2=\dhinterseq+1} ^ {2\dhinterseq} 
   (j_2-j_1)^\gamma\pr(\bsX_{1,j_1-1}^*\leqslant  \tepseq, X_{j_1}> \tepseq,  X_{j_2} > \tepseq,\bsX_{j_2+1,2\dhinterseq}^*\leqslant \tepseq)\\
   &= w_n^2\sum_{j_1=1}^{\dhinterseq} \sum_{j_2=\dhinterseq+2} ^ {2\dhinterseq} 
   (j_2-j_1)^\gamma\pr(\bsX_{1-j_1,-1}^*\leqslant  \tepseq\mid X_{0}> \tepseq) 
   \pr(\bsX_{1,2\dhinterseq-j_1}^*\leqslant \tepseq\mid X_0>\tepseq)\\
   &\phantom{=}-J_2(\dhinterseq)+J_3(\dhinterseq)=:w_n^2J_0(\dhinterseq)-J_2(\dhinterseq)+J_3(\dhinterseq)\;. 
\end{align*}
For $J_0(\dhinterseq)$ we have as $n\to\infty$,
\begin{align*}
&w_n^2 J_0(\dhinterseq)\sim \canditheta^2  w_n^2 \sum_{j=1}^{2\dhinterseq}j^\gamma[ (2\dhinterseq-j)\wedge j]\sim  \canditheta^2 w_n^2 \dhinterseq^{\gamma+2} \left[\int_0^1 s^{\gamma+1}\rmd s+\int_1^2 s^\gamma(2-s)\rmd s\right]\\
&
\sim\frac{2^{\gamma+2}-1}{(\gamma+1)(\gamma+2)} \canditheta^2w_n^2 \dhinterseq^{\gamma+2}\;. 
\end{align*}
Combining this with \eqref{eq:j3-mma1} and negligibility of $J_2(\dhinterseq)$, we conclude the proof. 
\end{proof}
 \begin{proof}[Proof of \Cref{lem:mma1-last-and-first-jump-consecutive-blocks-expectation}]
The expectation of interest is
\begin{align*}
&\sum_{j_1=1}^{\dhinterseq}\sum_{j_2=\dhinterseq+j_1}^{2\dhinterseq}(j_2-j_1-\dhinterseq)\pr(X_{j_1}>\tepseq,\bsX_{j_1+1,\ldots,j_2-1}^*\leqslant\tepseq,X_{j_2}>\tepseq)\\
&=\pr(c_0\xi_0\vee c_1\xi_1>\tepseq, c_0\xi_1\leqslant \tepseq)\\
&\phantom{=}\times
\pr(c_1\xi_{0}\leqslant \tepseq,c_0\xi_{0}\vee c_1\xi_{1}>\tepseq)\sum_{i=\dhinterseq+1}^{2\dhinterseq-i}(i-\dhinterseq)(2\dhinterseq-i)\pr(D_{2,i-1})\\
&\sim \left(\frac{(c_0\vee c_1)^\alpha}{c_0^\alpha+c_1^\alpha}\right)^2\int_1^2 (s-1)(2-s)\rmd s=\frac{1}{6}\canditheta^2\;.
\end{align*}
\end{proof}

\section{Proofs II - Limit theorems for internal and boundary clusters statistics}\label{sec:technical-details-block-statistics}
In this section we apply the theory established in \Cref{sec:internal-clusters,sec:technical-boundary} to particular functionals that appear in the context of the asymptotic expansion of blocks statistics. As a result, we obtain convergence in probability  for $\IC(H)$, $\BC(H;1)$ and $\BC(H;2)$. 
\subsection{Moments of clusters - small blocks}\label{sec:moments-of-clusters}

\subsubsection{Internal Clusters}
\label{subsubsec:mean-interior-clusters}
Recall the notation \eqref{eq:internal-as-functional} for $\IC_j(H)$. 
We analyse the moments of $\IC_1$. 
Since one block is involved and $\widetilde{H}_{\IC}$ is tight under the conditional law, the scaling is $\dhinterseq w_n$. Let $p\ge 0$. 
\begin{corollary}
\label{proposition:interior-cluster-means}
Assume that~\hyperlink{SummabilityAC}{$\mathcal{S}^{(p(\gamma+1))}(\dhinterseq, \tepseq)$} holds. For any $H\in\mch(\gamma)$ we have
\begin{align*}
&\lim_{n\to\infty}  \frac{  \esp[|\IC_1(H)|^p] } { \dhinterseq w_n } 
 =\tailmeasurestar(|\widetilde{H}_{\IC}|^p)\;.
\end{align*}
\end{corollary}
\begin{corollary}
    \label{prop:individual-interior-cluster-variance}
    Assume that~\hyperlink{SummabilityAC}{$\mathcal{S}^{(2(\gamma+1))}(\dhinterseq, \tepseq)$} holds. For any $H\in H(\gamma)$ we have
\begin{align*}
   \lim_{n\to\infty}\frac{ \var(\IC_1(H)) } { \dhinterseq w_n }
  =\tailmeasurestar(\widetilde{H}_{\IC}^2)\;.
\end{align*}
\end{corollary}
\begin{proof}[Proof of \Cref{proposition:interior-cluster-means,prop:individual-interior-cluster-variance}]
The functional $\widetilde{H}_{\IC}$ is clearly continuous with respect to the law of $\bsY$. We also note that
   $| \widetilde{H}_{\IC}( \mathbb{X}_{1}  ) | \leqslant 3C_H \clusterlength^{\gamma+1}(\mathbb{X}_{1})$ yielding
   \begin{align*}
   |\IC_{1}(H)|\leqslant \ind{A_1} |\widetilde{H}_{\IC}(\mathbb{X}_{1})|\leqslant 3C_H\ind{A_1}
     \clusterlength^{\gamma+1}(\mathbb{X}_{1})\;.
     \end{align*}
   Thus, the first conclusion follows from \Cref{lem:clusterlength-moments} along with \Cref{rem:clusterlength-add-small-jumps-for-free}.

With $p=1$, \Cref{proposition:interior-cluster-means} gives $(\esp[\IC_1(H)])^2= O(\dhinterseq^2 w_n^2)=o(\dhinterseq w_n)$.
Hence,
\begin{align*}
\lim_{n\to\infty}\frac{1}{\dhinterseq w_n}\var(\IC_1(H))=
\lim_{n\to\infty}\frac{1}{\dhinterseq w_n}\esp[\IC_1^2(H)]
\end{align*}
and the term converges to the desired limit.
\end{proof}

\subsubsection{Boundary clusters}
Recall the functionals $\widetilde{H}_{\BC}$ and $\widetilde{H}_{\BC,p}$  defined in \eqref{eq:new-functional-BC}-\eqref{eq:new-functional-BC-p}. 
Recall \eqref{eq;difference-SB&DB;boundary-cluster-part;exceedance-times-expansion;main}- \eqref{eq:boundary-2-as-functional}:
\begin{align*}
   \BC_j(H;1) = &
     \dhinterseq\left(H(\mathbb{X}_{j})+H(\mathbb{X}_{j+1})- H(\mathbb{X}_{j,j+1})\right)
\ind{A_{j-1}^c\cap A_j\cap A_{j+1}\cap A_{j+2}^c}\;
    \end{align*} 
and
\begin{align}
   &\BC_j(H;2) = \big[  \SB_{j-1} + \SB_j + \SB_{j+1} - \DB_{j,j+1}  \big]\ind{A_{j-1}^c\cap A_j\cap A_{j+1}\cap A_{j+2}^c}\label{eq:BCj-tilde-0}\\
   &=
   \ind{\clusterlength_{j,j+1}<\dhinterseq}\BC_j(H;2) +\ind{\clusterlength_{j,j+1}\geqslant \dhinterseq}\BC_j(H;2)\nonumber \\
   &=
     \ind{\clusterlength_{j,j+1}<\dhinterseq}\ind{A_{j-1}^c\cap A_j\cap A_{j+1}\cap A_{j+2}^c}\widetilde{H}_{\IC}(\mathbb{X}_{j,j+1})+\ind{\clusterlength_{j,j+1}\geqslant \dhinterseq}\BC_j(H;2)\nonumber\\
     &\leqslant
   \ind{A_{j-1}^c\cap A_j\cap A_{j+1}\cap A_{j+2}^c}\widetilde{H}_{\IC}(\mathbb{X}_{j,j+1})+\ind{\clusterlength_{j,j+1}\geqslant \dhinterseq}\BC_j(H;2)\nonumber\\
     &=: \widetilde{\BC}_j(H;2)+ \overline{\BC}_j(H;2)\;.\label{eq:BCj-tilde}
\end{align}
The asymptotic behaviour of these terms is significantly different, as illustrated by the results below. 
For the $p$th moment of $\BC_j(H;1)$ the scaling is $\dhinterseq^p w_n$, due to its specific, simple, structure, with $\dhinterseq$ as a multiplier.  
For $\BC_j(H;2)$, we note first that the representation \eqref{eq:BCj-tilde-0} yields a rough bound 
\begin{align}\label{eq:Bcj-second-rough}
|\BC_j(H;2)|\leqslant \dhinterseq H(\mathbb{X}_{j,j+1})\ind{A_{j-1}^c\cap A_j\cap A_{j+1}\cap A_{j+2}^c}\;.
\end{align}
This bound is not sharp enough - the right hand side of \eqref{eq:Bcj-second-rough} is of the order $\dhinterseq w_n$, with $w_n$ due to the presence of $A_j\cap A_{j+1}$. 
The trick is to split $\BC_j(H;2)$ according to small and large values of the cluster length. A small cluster length allows to replace $\dhinterseq H(\mathbb{X}_{j,j+1})$ with $\widetilde{H}_{\IC}(\mathbb{X}_{j,j+1})$. The latter is tight under the conditional law (given $A_j\cap A_{j+1}$). This yields the scaling $w_n$ for 
$\widetilde{\BC}_j(H;2)$. 
Finally, for $\overline{\BC}_j(H;2)$ we use the bound \eqref{eq:Bcj-second-rough} in conjunction with \Cref{cor:clusterlength-tail}, that controls the tail of the cluster length. 
\begin{corollary}
\label{prop:boundary-cluster-means-main}
    Assume that~\hyperlink{SummabilityAC}{$\mathcal{S}^{(p\gamma+1)}(\dhinterseq, \tepseq)$} holds. For any $H\in H(\gamma)$ we have
    \begin{align}
    \lim_{n\to\infty}
     \frac{\esp[|\BC_{1}(H;1)|^p]}{\dhinterseq^p w_n} = \tailmeasurestar(\widetilde{H}_{\BC,p})
    \;.
     \label{eq;mean-boundary-cluster-BC1,1;1}
    \end{align}
\end{corollary}
\begin{proof}
Recalling that $\DB_j(H)=\dhinterseq H(\bsX_{(j-1)\dhinterseq+1,j\dhinterseq}/\tepseq)=\dhinterseq H(\mathbb{X}_{j})$ and 
$\DB_{1,2}(H)=\dhinterseq H(\mathbb{X}_{1,2})$, the limit \eqref{eq;mean-boundary-cluster-BC1,1;1} follows easily from \Cref{prop:general-function-UniformIntegrability-boundary-clusters}$(\rom1)$, along with 
 \Cref{rem:clusterlength-add-small-jumps-for-free} (we can add the indicators of the small values for free), in conjunction with \eqref{eq:short-formula-for-Hbc}. 
\end{proof}

\begin{corollary}
\label{lem:boundary-cluster;First-Term;Variance}
    Assume that~\hyperlink{SummabilityAC}{$\mathcal{S}^{(2\gamma+1)}(\dhinterseq, \tepseq)$} holds. For any $H\in H(\gamma)$ we have
    \begin{align}
       & \lim_{n\to\infty} \frac{\var(\BC_{1}(H;1))}{\dhinterseq^2 w_n} =
    \tailmeasurestar(\widetilde{H}_{\BC,2})\;.
    \label{eq:boundary-cluster;First-Term;Variance}
    \end{align}
\end{corollary}    
\begin{proof}
It suffices to calculate $\esp[\BC_{1}^2(H;1)]$. Then 
\eqref{eq:boundary-cluster;First-Term;Variance} follows from \Cref{prop:general-function-UniformIntegrability-boundary-clusters} by noting that $H^2 \in \mch({2\gamma})$. 
\end{proof}
\begin{corollary}
\label{prop:boundary-cluster-means-minor}
    Assume that~\hyperlink{SummabilityAC}{$\mathcal{S}^{(p(\gamma+1)+1)}(\dhinterseq, \tepseq)$} holds. For any $H\in H(\gamma)$ we have
    \begin{align*}
    & \lim_{n\to\infty}
     \frac{\esp[|\widetilde{\BC}_{1}(H;2)|^p]}{w_n} =
        \canditheta\esp \left[ (\clusterlength(\bsZ) - 1)|\widetilde{H}_{\IC}(\bsZ)|^p
    \right]
    \;.
    \end{align*}
\end{corollary}
\Cref{prop:boundary-cluster-means-minor} gives immediately:
\begin{corollary}\label{prop:boundary-cluster-means-minor-var}
Assume that \hyperlink{SummabilityAC}{$\mathcal{S}^{(2\gamma+3)}(\dhinterseq, \tepseq)$} holds. For any $H\in\mch(\gamma)$ we have 
    \begin{align*}
       & \lim_{n\to\infty} \frac{\var(\widetilde{\BC}_1(H;2))}{w_n}  =
       \canditheta\esp\left[ (\clusterlength({\bsZ})-1)\widetilde{H}_{\IC}^2(\bsZ)
    \right]\;.
    \end{align*}
\end{corollary}

\begin{proof}[Proof of \Cref{prop:boundary-cluster-means-minor}]
We use the representation \eqref{eq:BCj-tilde}. 
Since $H\in \mch(\gamma)$, we have $\widetilde{H}_{\IC}\in \mch(\gamma+1)$ and $\widetilde{H}_{\IC}^p\in \mch(p(\gamma+1))$. We apply \eqref{eq;convergence-boundary-cluster-H-full} to conclude the proof. For this, we need 
\hyperlink{SummabilityAC}{$\mathcal{S}^{(p(\gamma+1)+1)}(\dhinterseq, \tepseq)$}. 
\end{proof}
\begin{corollary}
\label{prop:boundary-cluster-means-minor-overline}
Let $\eta>0$. 
    Assume that~\hyperlink{SummabilityAC}{$\mathcal{S}^{(p\gamma+\eta+1)}(\dhinterseq, \tepseq)$} holds. For any $H\in H(\gamma)$,
    \begin{align*}
    \esp[|\overline{\BC}_{1}(H;2)|^p]=O(w_n\dhinterseq^{p-\eta})\;.
    \end{align*}
\end{corollary}
\begin{proof}We have 
\begin{align*}
\esp[|\overline{\BC}_{1}(H;2)|^p]\leqslant \dhinterseq^p \esp[|H|^p(\mathbb{X}^{(1,2)})\ind{\clusterlength_{1,2}\geqslant \dhinterseq}]\;.
\end{align*}
Apply \Cref{cor:clusterlength-tail} with $\delta=p\gamma$ and $\gamma=\eta$. 
\end{proof}

\subsubsection{Boundary clusters - piecewise stationary case}
\begin{corollary}
\label{prop:boundary-cluster-means-main-piecewise}
Assume that $\bsX$ is piecewise stationary. Assume that \hyperlink{SummabilityAC}{$\mathcal{S}^{(\gamma)}(\dhinterseq, \tepseq)$} holds. For any $H\in\mch(\gamma)$, 
\begin{align*}
\esp[|\BC_1(H;1)|]=O(\dhinterseq^3 w_n^2)\;.
\end{align*}
\end{corollary}
\begin{proof}
Independence between blocks, \Cref{lem:pa1-cap-pa2-precise-large,lem:clusterlength-moments} give
\begin{align*}
   &\esp[|\BC_1(H;1)|] \\
   &
     \leqslant\dhinterseq \esp[|H|(\tepseq^{-1}\bsX_{1,2\dhinterseq})\ind{A_1\cap A_2}]\pr^2(A_1^c)
       + 2\dhinterseq \esp[|H|(\tepseq^{-1}\bsX_{1,\dhinterseq})\ind{A_1}]\pr^2(A_1^c)\pr(A_1)\\
      &= O(\dhinterseq \dhinterseq^2 w_n^2)+O(\dhinterseq \dhinterseq w_n \dhinterseq w_n)\;.
    \end{align*}
\end{proof}
\begin{corollary}
\label{prop:boundary-cluster-means-main-1-piecewise}
Assume that $\bsX$ is piecewise stationary. Assume that \hyperlink{SummabilityAC}{$\mathcal{S}^{(\gamma)}(\dhinterseq, \tepseq)$} holds. For any $H\in\mch(\gamma)$, 
\begin{align*}
\esp[|{\BC}_1(H;2)|]=O(\dhinterseq^3 w_n^2)\;.
\end{align*}
\end{corollary}
\begin{proof}
Recall the definition of $\SB_j$ in \eqref{eq:def-SBj}. On $A_{j-1}\cap A_j$, $\SB_{j-1}$ is a function of the block $I_j$ only. Hence, the independence between blocks gives 
\begin{align*}
\esp[\SB_{j-1}\ind{A_{j-1}^c\cap A_j\cap A_{j+1}\cap A_{j+2}^c}]&= \esp[\SB_{j-1} \ind{A_j}]\pr(A_{j-1})\pr(A_{j+1}^c)\pr(A_{j+2})\\
&\leqslant \dhinterseq \esp[|H|(\tepseq^{-1}\bsX_{1,\dhinterseq})\ind{A_1}]\pr^2(A_1^c)\pr(A_1)\;.
\end{align*}
Hence, \Cref{lem:clusterlength-moments} gives
\begin{align*}
\esp[\SB_{j-1}\ind{A_{j-1}^c\cap A_j\cap A_{j+1}\cap A_{j+2}^c}]&=O(\dhinterseq \dhinterseq w_n \dhinterseq w_n)\;. 
\end{align*}
Likewise, 
\begin{align*}
\esp[\SB_{j+1}\ind{A_{j-1}^c\cap A_j\cap A_{j+1}\cap A_{j+2}^c}]&= \esp[\SB_{j+1} \ind{A_{j+1}}]\pr(A_{j-1}^c)\pr(A_{j})\pr(A_{j+2}^c)=O(\dhinterseq^3 w_n^2)\;. 
\end{align*}
The term $\esp[\DB_{j,j+1}\ind{A_{j-1}^c\cap A_j\cap A_{j+1}\cap A_{j+2}^c}]$ was bounded in the preceding corollary. 
\end{proof}

\subsection{Moments of clusters - large blocks}\label{sec:moments-of-clusters-large}
We continue with large blocks. Here, we are going to focus on the special case of $H(\bsx)=\ind{\bsx^*>1}$ and the toy example MMA(1); see \Cref{sec:MMA(1)}. Then (cf. \Cref{xmpl:extremal-index-1})
\begin{align*}
\IC_j(H)=\left(\clusterlength(\tepseq^{-1}\bsX_{(j-1)\dhinterseq+1,j\dhinterseq})-1\right)\ind{A_{j-1}^c\cap A_j\cap A_{j+1}^c}\;.
\end{align*} 
\Cref{lem:clusterlength-moments-large} gives
the following counterpart to \Cref{proposition:interior-cluster-means,prop:individual-interior-cluster-variance}:
\begin{corollary}\label{cor:moments-clusters-large-ic}
Assume that $\bsX$ is
stationary and $\ell$-dependent. 
Assume that \ref{eq:conditiondh} holds and $\dhinterseq^2 w_n\to \infty$.
Then
\begin{align*}
\lim_{n\to\infty}\frac{\esp[\IC_1(H)]}{\dhinterseq^3 w_n^2}= \frac{1}{6}\canditheta^2\;.
\end{align*}
\end{corollary}
Next (cf. \Cref{xmpl:extremal-index-2}), 
$\BC_j(H;1)=-\dhinterseq\ind{A_{j-1}^c\cap A_j\cap A_{j+1}\cap A_{j+2}^c}$,
and
\begin{align*}
\BC_j(H;2)=\ind{A_{j-1}^c\cap A_j\cap A_{j+1}\cap A_{j+2}^c}
\left[\left(t_{j+1}{(N_{j+1})}-t_{j}(1)\right)-\left(t_{j+1}(1)-t_{j}{(N_j)}-\dhinterseq\right)_+\right]  
 \;.
 \end{align*}
Thus, for the MMA(1) process we obtain via \Cref{lem:MMA1-P(a1-cap-a2)}
\begin{align*}
\lim_{n\to\infty}\frac{\esp[\BC_1(H;1)]}{\dhinterseq^3 w_n^2}=-\canditheta^2\;
\end{align*}
as long as $\dhinterseq^2 w_n\to \infty$. 
Then, \Cref{lem:boundary-joint-cluster-length-MMA1,lem:mma1-last-and-first-jump-consecutive-blocks-expectation} give
\begin{align*}
\lim_{n\to\infty}\frac{\esp[\BC_1(H;2)]}{\dhinterseq^3 w_n^2}=\frac{7}{6}\canditheta^2-\frac{1}{6}\canditheta^2=\canditheta^2\;,
\end{align*}
whenever $\dhinterseq^3w_n\to \infty$.
In summary:
\begin{corollary}\label{cor:moments-clusters-large-bc}
Assume that $\bsX$ is the MMA(1) process. Assume that \ref{eq:conditiondh} holds and $\dhinterseq^3 w_n\to \infty$. Then 
\begin{align*}
\lim_{n\to\infty}\frac{\esp[\BC_1(H)]}{\dhinterseq^3 w_n^2}=0\;. 
\end{align*}
\end{corollary}
\subsection{Moments of clusters statistics - small blocks scenario}\label{sec:dependence-clusters}
In this section we calculate mean and variance of cluster statistics $\IC(H)$ (cf. 
\eqref{eq:interior-clusters-def} and \eqref{eq:internal-as-functional}) and $\BC(H)$ (cf. 
\eqref{eq:boundary-clusters-def} and \eqref{eq:boundary-clusters-decomposition}). To break dependence between blocks, we need some mixing assumptions. 

The most important conclusion is that while the expectations of $\IC$ and $\BC$ growth at the same rate $nw_n$, the variances growth at different rates: $\var(\IC(H))$ at the rate $nw_n$, while the variance of $\BC(H)$ at the rate $n\dhinterseq w_n$.

\subsubsection{Internal Clusters Statistics}
Recall that the statistics is defined as $\IC(H)=\sum_{j=2}^{m_n-1}\IC_j(H)$; cf. 
\eqref{eq:interior-clusters-def} and \eqref{eq:internal-as-functional}.
Recall also the functional $\widetilde{H}_{\IC}$ defined in \eqref{eq:new-functional-IC}.
The first result is an immediate consequence of \Cref{proposition:interior-cluster-means}.
\begin{corollary}
\label{corollary:interior-cluster-means}
Assume that~\hyperlink{SummabilityAC}{$\mathcal{S}^{(\gamma+1)}(\dhinterseq, \tepseq)$} holds. For any $H\in\mch(\gamma)$ we have
\begin{align*}
\lim_{n\to\infty}\frac{\esp[\IC(H)]}{nw_n}=\tailmeasurestar(\widetilde{H}_{\IC})\;. 
\end{align*}
\end{corollary}
The next result shows that the internal clusters statistics behaves as if the summands were independent. 
\begin{proposition}
    \label{prop:total-interior-cluster-variance}
    Assume that $\bsX$ is stationary and mixing with the rates \eqref{eq:mixing-rates}. 
Assume that~\hyperlink{SummabilityAC}{$\mathcal{S}^{(2\gamma+2+\delta)}(\dhinterseq, \tepseq)$} holds for some $\delta>0$. For any $H\in\mch(\gamma)$ we have
  \begin{align*}
   \lim_{n\to\infty}\frac{\var(\IC(H))}{nw_n} = \tailmeasurestar(\widetilde{H}_{\IC}^2)\;.
  \end{align*}
\end{proposition}
\begin{proof}
    It suffices to show that
    \begin{align*}
        \left|
        \sum_{j=2}^{m_n-1} \cov(\IC_1, \IC_{j})
        \right| = o\left(  \var(\IC_1)  \right) = o (\dhinterseq w_n)\;.
    \end{align*}
    We divide the proof into the following four steps.

    $(\rom1)$: Since $\IC_{j}$ is based on the indicator $\ind{A_{j-1}^c\cap A_j \cap A_{j+1}^c}$, for $j=2$ we have by \Cref{proposition:interior-cluster-means}
    $$
     \left| \cov(\IC_1, \IC_{2}) \right| =  (\esp [\IC_1] ) ^2 = O(\dhinterseq^2 w_n^2)=o(\dhinterseq w_n)\;.
    $$
For this we need \hyperlink{SummabilityAC}{$\mathcal{S}^{(\gamma+1)}(\dhinterseq, \tepseq)$}. 

     $(\rom2)$: Let $j=3,4$. 
      Take $p = 2 + \epsilon $ for any $\epsilon>0$ such that
 $$
  (\gamma + 1 ) (2 + \epsilon ) \leqslant 2\gamma + 2+\delta\;.
 $$
It follows from \Cref{proposition:interior-cluster-means} and the mixing inequality \eqref{eq:mixing-inequality}, 
 \begin{align*}
 | \cov(\ind{A_1} \widetilde{H}_{\IC}(\mathbb{X}_{1}), \ind{A_j} \widetilde{H}_{\IC}(\mathbb{X}_{j}) |\leqslant
\alpha^{1/r}_{(j-2)\dhinterseq} \Vert \ind{A_1} \clusterlength_{1}^{\gamma + 1}  \Vert_p^2  \leqslant & ( \dhinterseq w_n )^{ 2/(2+\epsilon) }
\alpha^{ \epsilon/(2+\epsilon)  }_{ (j-2)\dhinterseq }\;.
     \end{align*}
It is $o(\dhinterseq w_n)$ by applying \eqref{eq:mixing-rates}. For this we need \hyperlink{SummabilityAC}{$\mathcal{S}^{(2\gamma+2+\delta)}(\dhinterseq, \tepseq)$}. 

Next,
\begin{align*}
|\esp[\IC_1(H)-\ind{A_1} \widetilde{H}_{\IC}(\mathbb{X}_{1})]|\leqslant \esp[\ind{A_1}\ind{A_2^c} |\widetilde{H}_{\IC}|(\mathbb{X}_{1})]=O(w_n)=o(\dhinterseq w_n)
\end{align*}
by \eqref{eq;joint-convergence-boundary-cluster;H-front}.
Hence, together, $\cov(\IC_1(H),\IC_j(H))=o(\dhinterseq w_n)$.

     $(\rom3)$: In this step, we will show that for $M>4$,
     $$
    \lim_{M \to \infty} \limsup_{n\to \infty}
    \; \sum_{j=M} ^{m_n-1} \cov(\IC_1, \IC_{j}) = o(\dhinterseq w_n)\;.
     $$
     We apply the covariance inequality under mixing to get
     \begin{align*}
       | \cov(\IC_1, \IC_{j}) | \leqslant \constant\;
        \alpha^{1/r}_{(j-4)\dhinterseq} \Vert \IC_1 \Vert_p ^2
       \leqslant\constant\; \; \alpha^{1/r}_{(j-4)\dhinterseq} \Vert \ind{A_1} \clusterlength_1^{\gamma + 1}  \Vert_p^2 \;.
     \end{align*}
 Take $p = 2 + \epsilon $ for any $\epsilon>0$ such that
 $$
  (\gamma + 1 ) (2 + \epsilon ) \leqslant 2\gamma + 2+\delta\;.
 $$
It follows from \Cref{proposition:interior-cluster-means}
 \begin{align*}
 | \cov(\IC_1, \IC_{j}) |\leqslant
\alpha^{1/r}_{(j-4)\dhinterseq} \Vert \ind{A_1} \clusterlength_{1}^{\gamma + 1}  \Vert_p^2  \leqslant & ( \dhinterseq w_n )^{ 2/(2+\epsilon) }
\alpha^{ \epsilon/(2+\epsilon)  }_{ (j-4)\dhinterseq }\;.
     \end{align*}
Thus,
\begin{align*}
\sum_{j=M} ^{m_n-1}| \cov(\IC_1, \IC_{j})|\leqslant  ( \dhinterseq w_n )^{ 2/(2+\epsilon) }\sum_{j=M} ^{m_n-1} \alpha^{ \epsilon/(2+\epsilon)  }_{ (j-4)\dhinterseq }
\end{align*}
and the result follows by applying \eqref{eq:mixing-rates}. 
\end{proof}
\subsubsection{Boundary Clusters Statistics}
Recall that the boundary clusters statistics is defined as 
\begin{align*}
\BC(H):=\BC(H;1)+\BC(H;2)=
\sum_{j=2}^{m_n-1}\BC_j(H;1)+\sum_{j=2}^{m_n-1}\BC_j(H;2)\;, 
\end{align*}
cf. 
\eqref{eq:boundary-clusters-def} and \eqref{eq:boundary-clusters-decomposition}.  Recall also the decomposition \eqref{eq:BCj-tilde}. 
Set 
\begin{align*}
\widetilde\BC(H;2)=\sum_{j=2}^{m_n-1}\widetilde{\BC}_j(H;2)\;,  \ \ 
\overline\BC(H;2)=\sum_{j=2}^{m_n-1}\overline{\BC}_j(H;2)\;. 
\end{align*}
In view of  \Cref{prop:boundary-cluster-means-minor-var,prop:boundary-cluster-means-minor-overline}, it is enough to consider $\BC(H;1)$ only. 

The first result is an immediate consequence of 
\Cref{prop:boundary-cluster-means-main,prop:boundary-cluster-means-minor,prop:boundary-cluster-means-minor-overline}.
\begin{corollary}\label{corollary:boundary-cluster-means} 
    Assume that~\hyperlink{SummabilityAC}{$\mathcal{S}^{(\gamma+2+\delta)}(\dhinterseq, \tepseq)$} holds for some $\delta>0$. For any $H\in H(\gamma)$ we have
\begin{align*}
    \lim_{n\to\infty}
     \frac{\esp[\BC(H)]}{n w_n} =
   \tailmeasurestar(\widetilde{H}_{\BC})
   \;.
    \end{align*}
\end{corollary}
\begin{proposition}
 \label{prop:boundary-clusters-full-variance}
 Assume that $\bsX$ is stationary and
mixing with the rates \eqref{eq:mixing-rates}. 
Assume that~\hyperlink{SummabilityAC}{$\mathcal{S}^{(2\gamma+2+\delta)}(\dhinterseq, \tepseq)$} holds for some $\delta>0$. For any $H\in\mch(\gamma)$ we have
 \begin{align*}
&\lim_{n\to \infty} \frac{\var(\BC(H;1))}{n\dhinterseq w_n} = \lim_{n\to \infty} \frac{\var(\BC_{1}(H;1))}{\dhinterseq^2 w_n} =\tailmeasurestar(\widetilde{H}_{\BC,2})
 \end{align*}
\end{proposition}
\begin{proof}
We will show that, as $n\to\infty$, 
\begin{align}
\frac{1}{\dhinterseq^2 w_n} \left|
\sum_{j=2}^{m_n-1} \cov(\BC_1(H;1),\BC_j(H;1))
\right| \to 0\;. \label{eq:boundary-clusters-small-full-covariance}
\end{align}

We shall divide the proof of \eqref{eq:boundary-clusters-small-full-covariance} into the following three steps.

$(\rom1)$: For $j=2,3$, we have by \Cref{prop:boundary-cluster-means-main}
\begin{align*}
    \cov(\BC_1(H;1),\BC_j(H;1)) = - (\esp[\BC_1(H;1)]) ^ 2 = (\dhinterseq w_n) ^ 2 = o(\dhinterseq^2 w_n)\;.
\end{align*}

$(\rom2)$: For  $j=4,5,\ldots$, we note first the trivial bound $|\BC_j(H;1)|\leqslant \constant\;\dhinterseq^{\gamma+1}$. Indeed, $|H|\leqslant \constant\;\clusterlength^{\gamma}\leqslant \constant\; \dhinterseq^{\gamma}$. We have 
\begin{align*}
  & \big| \esp [\BC_1(H;1) \BC_j(H;1) ] \big|\leqslant \constant\;  \dhinterseq^{2\gamma+2} \pr( A_1 \cap A_2\cap  A_{j}\cap A_{j+1} ) \leqslant  \constant\;\dhinterseq^{2\gamma + 2 } \pr( A_1\cap  A_{j} ) \\
    &\leqslant   \dhinterseq^{2\gamma + 2 } \cdot \dhinterseq \cdot \sum_{i = (j-2)\dhinterseq }^{j\dhinterseq} \pr( | \bsX_0 | > \tepseq ,  | \bsX_i | > \tepseq) \\
   & =\dhinterseq^{2\gamma + 2 } \cdot \dhinterseq \cdot o( \dhinterseq^{- 2\gamma - 1} w_n ) = o( \dhinterseq^2 w_n )
\end{align*}
whenever \hyperlink{SummabilityAC}{$\mathcal{S}^{(2\gamma+1)}(\dhinterseq, \tepseq)$} is satisfied; cf. \eqref{eq:consequence-condition-s}.

$(\rom3)$: We have 
\begin{align*}
  \sum_{i=\ell}^{m_n-1} \big| \cov[\BC_1(H;1), \BC_j(H;1)] \big|
   \leqslant\constant\;.  \sum_{i=\ell}^{m_n-1}\alpha_{(i-4)\dhinterseq}^{1/r}\|\BC_{1}(H;1)\|_p^2
\end{align*}
and we conclude in the same way as in the case of \Cref{prop:total-interior-cluster-variance}. The assumption  \hyperlink{SummabilityAC}{$\mathcal{S}^{(2\gamma+2+\delta)}(\dhinterseq, \tepseq)$} is needed here. 

\end{proof}

\subsection{Remainder terms}\label{sec:remainder}
The difference between the sliding and the disjoint blocks statistics can be decomposed as follows:
\begin{align}
\label{eq;difference-SB&DB;three-parts}
 \SB(H) - \DB(H) = \IC(H) + \BC(H) + \RC(H)\;,
\end{align}
where $\SB(H)$, $\DB(H)$, $\IC(H)$, $\BC(H)$ are defined in \eqref{eq:def-SB}, \eqref{eq:def-DB}, \eqref{eq:internal-as-functional} and 
\eqref{eq:boundary-clusters-def}, respectively, and the remainder term $\RC(H)$ is given by the sum of three type of terms $\RIC(H)$, $\RBC(H)$, $\RNC(H)$ to be defined below. 
The main result of this section is follows. Note that we are concerned in weakening the assumptions of the result - the assumptions are determined by the considerations on the internal and the boundary clusters.
\begin{lemma}\label{lem:remainder}
\begin{itemize}
\item
Assume that $H\in \mch(\gamma)$ and  \hyperlink{SummabilityAC}{$\mathcal{S}^{(2\gamma+3)}(\dhinterseq, \tepseq)$} holds. Then 
\begin{align*}
\frac{1}{n\dhinterseq w_n}\RC=O_P\left(\frac{\dhinterseq}{n}+\dhinterseq^{-(\gamma+4)}\right)\;.
\end{align*}
\item Assume that $\bsX$ is $\ell$-dependent and $H$ is bounded. 
Assume that $\dhinterseq w_n\to \infty$. Then 
\begin{align*}
\frac{1}{n\dhinterseq w_n}\RC=O_P\left(\frac{\dhinterseq^{2}w_n}{n}+\frac{\dhinterseq^3 w_n^2}{n}\right)=O_P\left(\frac{\dhinterseq^3 w_n^2}{n}\right)\;.
\end{align*}
\end{itemize}
\end{lemma}
In order to prove the above lemma, we need to provide a decomposition of $\RC$. Then, \Cref{lem:remainder} will follow from \Cref{lem:remainder-BC-IC} (small blocks case and large blocks cases applied with $\gamma=0$), \Cref{lem:remainder-NC} (small blocks case)  and \Cref{lem:remainder-NC-large} (large blocks case) below. 
The first summand in the remainder stems from the internal clusters: 
   \begin{align*}
  \RIC(H) =  [ \SB_{1} - \DB_{1}] \ind{ A_1\cap A_2^c} +\SB_{m_n - 1}   \ind{ A_{m_n - 1}^c \cap A_{m_n} }\;.
\end{align*}
The second summand stems from the boundary clusters:
  \begin{subequations}
\begin{align*}
 \RBC(H) &=  [\SB_{1} - \DB_{1}] \ind{ A_1\cap A_2  }  + [\SB_{2} - \DB_{2} ] \ind{ A_1\cap A_2\cap A_3^c }  
 \\
 &\phantom{=}+ [\SB_{m_n - 2}] \ind{ A_{m_n - 2}^c\cap A_{m_n - 1}  \cap A_{m_n}  } + [\SB_{m_n - 1} - \DB_{m_n - 1}] \ind{   A_{m_n - 1} \cap A_{m_n} }  \; . 
\end{align*}
\end{subequations}
Recall the definition \eqref{eq:def-SBj} of $\SB_j$. In $\RIC(H)$, drop $\ind{A_2}$ and $\ind{A_{m_n-1}}$.   
We use then \eqref{eq:SBj-1:nojump-jump} and the assumption $H\in \mch(\gamma)$ to get 
\begin{align*}|\RIC(H)|\leqslant \dhinterseq\left\{ \clusterlength^\gamma(\tepseq^{-1}\bsX_{1,\dhinterseq})\ind{A_1}+
\clusterlength^\gamma(\tepseq^{-1}\bsX_{(m_n-1)\dhinterseq+1,m_n\dhinterseq})\ind{A_{m_n}}\right\}\;. 
\end{align*}
The same bound applies to $\RBC(H)$. 
Hence, \Cref{lem:clusterlength-moments} and \Cref{lem:clusterlength-moments-large} give:
\begin{lemma}\label{lem:remainder-BC-IC}
Assume that $H\in \mch(\gamma)$. 
\begin{itemize}
\item 
Assume that  \hyperlink{SummabilityAC}{$\mathcal{S}^{(\gamma)}(\dhinterseq, \tepseq)$} holds. Then  
\begin{align*}
\esp[|\RIC(H)+\RBC(H)|]=O(\dhinterseq^2 w_n)\;. 
\end{align*}
\item Assume that $\dhinterseq^{\gamma+1}w_n\to \infty$. Then 
\begin{align*}
\esp[|\RIC(H)+\RBC(H)|]=O(\dhinterseq^{\gamma+3} w_n^2)\;. 
\end{align*}
\end{itemize}
\end{lemma}

The negligible part $\RNC$ stems from terms that involve jumps in at least three consecutive blocks: 
\begin{align}
    \RNC(H) & = \sum_{j=2}^{m_n - 2} [ \RNC_{j}(H;1)  + \RNC_{j}(H;2)  + \RNC_{j}(H;3) ] \;,
    \label{eq:expansion;negligible-cluster-terms;1} 
\end{align}
where 
\begin{subequations}
\begin{align*}
      \RNC_{j}(H;1)& = [\SB_{j-1}(H) +\SB_{j}(H) - \DB_{j}(H)]  \ind{ A_{j-1}^c\cap A_{j}\cap A_{j+1} \cap A_{j+2}}   
         \; ,  \\ 
     \RNC_{j} (H;2) &=  
    [\SB_{j}(H) - \DB_{j}(H) + \SB_{j+1}(H) - \DB_{j+1}(H) ]      \cdot  
     \ind{ A_{j-1}\cap A_{j}\cap A_{j+1} \cap A_{j+2}^c  }
      \; ,   \\
     \RNC_{j} (H;3) &= [\SB_{j}(H) - \DB_{j}(H)  ] 
     \cdot  
     \ind{ A_{j-1}\cap A_{j} \cap A_{j+1}\cap A_{j+2}}
      \; .
      \end{align*}
\end{subequations}
\begin{lemma}\label{lem:remainder-NC}
Assume that $H\in \mch(\gamma)$ and 
\hyperlink{SummabilityAC}{$\mathcal{S}^{(2\gamma+3)}(\dhinterseq, \tepseq)$} holds. Then 
\begin{align*}
\RNC =O_P(n\dhinterseq^{-(\gamma+3)}w_n)\;. 
\end{align*} 
\end{lemma}  
\begin{proof}
Each of the terms $\RNC_{j} (H;i)$, $i=1,2,3$, can be bounded as follows:
\begin{align*}
\esp[\RNC_{j} (H;i)]\leq 
\constant\ \dhinterseq^{\gamma+1} \pr(A_1\cap A_3)\;. 
\end{align*}
Then, the first jump (block 1) and the last (block 3) decomposition, give
\begin{align*}
\pr(A_1\cap A_3)&=\sum_{j_1=1}^{\dhinterseq}\sum_{j_3=2\dhinterseq+1}^{3\dhinterseq}\pr(t_1(1)=j_1,t_3(N_3)=j_3)\\
&\leq \sum_{j=\dhinterseq+1}^{3\dhinterseq}[(j-\dhinterseq)\wedge (3\dhinterseq-j)]\pr(\norm{\bsX_0}>\tepseq,\norm{\bsX_j}>\tepseq)\\
&\leq \dhinterseq \sum_{j=\dhinterseq+1}^{3\dhinterseq}\pr(\norm{\bsX_0}>\tepseq,\norm{\bsX_j}>\tepseq)=
o(\dhinterseq \dhinterseq^{-(2\gamma+3)}w_n)
\end{align*}
The last estimate follows from \eqref{eq:consequence-condition-s} on account of the anticlustering assumption. 
Thus
\begin{align*}
\esp[|\RNC(H)|]\leq \constant\ m_n \dhinterseq^{\gamma+1}\dhinterseq^{-(2\gamma+3)}w_n\;. 
\end{align*}
\end{proof}
\begin{lemma}\label{lem:remainder-NC-large}
Assume that $\bsX$ is $\ell$-dependent and $H$ is bounded. 
Then 
\begin{align*}
\esp[|\RNC(H)|]=O(\dhinterseq^4 w_n^3)\;.
\end{align*}
\end{lemma}
\begin{proof}
We have 
\begin{align*}
\esp[|\RNC(H)|]\leq \constant \ \dhinterseq \pr(A_1\cap A_2\cap A_3)=O(\dhinterseq \dhinterseq^3w_n^3)\;,
\end{align*}
where the bound on $\pr(A_1\cap A_2\cap A_3)$ follows easily from $\ell$-dependence. 
\end{proof}

\section{Proofs III - Representations for internal and boundary clusters statistics}\label{sec:detailed-decomposition}
In this section we provide detailed calculations that yield representations for both internal and boundary block statistics, $\IC(H)$, $\BC(H;1)$ and $\BC(H;2)$. 

Recall that $\SB_{j}(H)=\sum_{i=(j-1)\dhinterseq+1}^{j\dhinterseq}H(\tepseq^{-1}\bsX_{i,i+\dhinterseq-1})$. In order to analyse contribution to $\SB_j$, we need to look for large jumps in $j$th and $(j+1)$th block. A special care has to be taken when analysing what happens at the beginning and the end of each block.  
\subsection{General case}
Here, we consider an arbitrary function $H$. 
\paragraph{Explanation for \eqref{eq:internal-as-functional}.}
Note that $\dhinterseq=\sum_{i=0}^{N_j}\Delta t_{j}{(i)}$. Hence. 
\begin{align*}
&\IC=\IC(H)=\sum_{j=2}^{m_n-1}(\SB_{j-1}+\SB_j-\DB_j)   \ind{A_{j-1}^c\cap A_j\cap A_{j+1}^c}\\
&=\sum_{j=2}^{m_n-1}\left[\sum_{i=0}^{N_j}\Delta t_{j}{(i)}\left\{
 H(\mathbb{X}_{j}{(1:i)}+  H(\mathbb{X}_{j}{(i+1:N_j)})-H(\mathbb{X}_{j}) \right\} \right]\ind{A_{j-1}^c\cap A_j\cap A_{j+1}^c}\notag\\
&=\sum_{j=2}^{m_n-1}\left[ \sum_{i=1}^{N_j-1}\Delta t_{j}{(i)}\left\{ H(\mathbb{X}_{j}{(1:i)}+  H(\mathbb{X}_{j}{(i+1:N_j)}) -H(\mathbb{X}_{j})
\right\}  \right]\ind{A_{j-1}^c\cap A_j\cap A_{j+1}^c}\notag\\
&=\sum_{j=2}^{m_n-1}\IC_{j}(H)\;.
\end{align*}

\paragraph{Explanation for \eqref{eq:SBj-1:nojump-jump}.} We analyse $\SB_{j-1}(H)$ on $A_{j-1}^c\cap A_j$. The index $i$ comes from the $(j-1)$th block: $(j-2)\dhinterseq+1\leqslant i\leqslant (j-1)\dhinterseq$. A contribution to $\SB_{j-1}(H)$ comes from the jumps in block $j$ only. 
\begin{itemize}
\item[{\rm (I-a)}] At the beginning of the $(j-1)$th block:
If $(j-1)\dhinterseq\leqslant i+\dhinterseq-1<t_{j}(1)$, then $H\equiv 0$.
\item[{\rm (I-b)}] In the middle of the $(j-1)$th block:
Let $\ell=1,\ldots,N_j-1$. 
If
$t_{j}(\ell)\leqslant i+\dhinterseq-1<t_{j}(\ell+1)$, then the contribution is $H(\tepseq^{-1}\bsX_{i,i+\dhinterseq-1})=H(\tepseq^{-1}\bsX_{t_{j}(1),t_{j}(\ell)})=
H(\mathbb{X}_{j}{(1:\ell)})$. This happens for $\Delta t_{j}(\ell)=t_{j}(\ell+1)-t_{j}(\ell)$ of the indices. 
\item[{\rm (I-c)}]  At the end of the $(j-1)$th block: If $t_{j}{(N_j)}\leqslant i+\dhinterseq-1< j\dhinterseq=t_{j}{(N_j+1)}$, then the contribution is 
$H(\tepseq^{-1}\bsX_{i,i+\dhinterseq-1})=H(\tepseq^{-1}\bsX_{t_{j}(1),t_{j}{(N_j)}})=
H(\mathbb{X}_{j})$. This happens for $j\dhinterseq - t_{j}{(N_j)}=t_{j}{(N_j+1)}- t_{j}{(N_j)}=\Delta t_{j}{(N_j)}$ of the indices. 
\end{itemize}
In summary, on $A_{j-1}^c\cap A_j$, we have 
\begin{align*}
\SB_{j-1}(H)&=\sum_{\ell=1}^{N_j}\Delta t_{j}(\ell)H(\mathbb{X}_{j}{(1:\ell)})\;.
\end{align*}
\paragraph{Explanation for \eqref{eq:SBj:jump-nojump}.}
We analyse $\SB_{j+1}(H)$ on $A_{j+1}\cap A_{j+2}^c$ (then \eqref{eq:SBj:jump-nojump} follows by replacing $j$ with $j-1$.) The index $i$ comes from the $(j+1)$th block: $j\dhinterseq+1\leqslant i\leqslant (j+1)\dhinterseq$. A contribution to $\SB_{j+1}(H)$ comes from the jumps in block $j+1$ only. 
\begin{itemize}
\item[{\rm (II-a)}] At the beginning of the $(j+1)$th block:
If $j\dhinterseq+1\leqslant i\leqslant  t_{j+1}(1)$, then the contribution is  $H(\mathbb{X}_{j+1}{(1:N_{j+1})})$. This happens for $t_{j+1}(1)-j\dhinterseq=t_{j+1}(1)-t_{j+1}(0)=\Delta t_{j+1}(0)$ of the indices. 
\item[{\rm (II-b)}] In the middle of the $(j+1)$th block:
Let $\ell=1,\ldots,N_{j+1}-1$. 
If
$t_{j+1}(\ell)< i\leqslant t_{j+1}(\ell+1)$, then the contribution is $H(\mathbb{X}_{j+1}{(\ell+1:N_{j+1})})$.
This happens for $\Delta t_{j+1}(\ell)=t_{j+1}(\ell+1)-t_{j+1}(\ell)$ of the indices. 
\item[{\rm (II-c)}]  At the end of the $(j+1)$th block: If $t_{j+1}{(N_j+1)}< i$, then 
$H\equiv 0$.
\end{itemize}
In summary, on $A_{j+1}\cap A_{j+2}^c$, we have 
\begin{align*}
\SB_{j+1}(H)&=\sum_{\ell=0}^{N_{j+1}-1}\Delta t_{j+1}(\ell)H(\mathbb{X}_{j+1}{(\ell+1:N_{j+1})})\;.
\end{align*}
\paragraph{Explanation for \eqref{eq:boundary-2-as-functional}.}
\paragraph{Contribution from $\SB_j(H)$ on $A_j\cap A_{j+1}$.} The index $i$ runs from $(j-1)\dhinterseq+1$ to $j\dhinterseq$. Contribution to $\SB_j$ comes from the jumps in both $j$th and $(j+1)$th block. This contribution can be written in a more concise way, but the more detailed decomposition below will be useful when combining $\SB_j(H)\ind{A_j\cap A_{j+1}}$ with other terms. 

Recall the convention $t_{j}(0)=(j-1)\dhinterseq$ and $t_{j}{(N_j+1)}=j\dhinterseq$. Recall also the notation for adjacent blocks from \Cref{sec:adjacent-blocks}.

Let $q=0,\ldots,N_{j}$ and  
$s=0,\ldots,N_{j+1}$. 
Let \begin{align*}
\mathcal{J}_{q,s}=\{i=(j-1)\dhinterseq+1,\ldots,j\dhinterseq:t_{j}(q)< i\leqslant t_{j}(q+1),t_{j+1}(s)\leqslant i+\dhinterseq-1< t_{j+1}(s+1)\}\;.
\end{align*}
If $i\in \mathcal{J}_{q,s}$ then the contribution is $H(\mathbb{X}_{j,j+1}{(q+1:N_j+s)})$. 
    \begin{itemize}
    \item[{\rm (III-a)}] If $t_{j}(q+1)\leqslant t_{j+1}(s)-\dhinterseq$, then $\mathcal{J}_{q,s}$ is empty and there is no contribution.
    \item[{\rm (III-b)}] If $t_{j}(q)<t_{j+1}(s)-\dhinterseq< t_{j}(q+1)\leqslant t_{j+1}(s+1)-\dhinterseq$, then we have the contribution for $t_{j}(q+1)-(t_{j+1}(s)-\dhinterseq)$ of the indices. 
    \item[{\rm (III-c)}] If $t_{j}(q)<t_{j+1}(s)-\dhinterseq<t_{j+1}(s+1)-\dhinterseq< t_{j}(q+1)$, then we have the contribution for $t_{j+1}(s+1)-t_{j+1}(s)$ of the indices. 
    \item[{\rm (III-d)}] If $t_{j+1}(s)-\dhinterseq\leqslant t_{j}(q)< t_{j}(q+1)\leqslant t_{j+1}(s+1)-\dhinterseq$, then we have the contribution for $t_{j}(q+1)-t_{j}(q)$ of the indices. 
    \item[{\rm (III-e)}] If $t_{j+1}(s)-\dhinterseq\leqslant t_{j}(q)<t_{j+1}(s+1)-\dhinterseq< t_{j}(q+1)$, then we have the contribution for $t_{j+1}(s+1)-t_{j}(q)-\dhinterseq$ of the indices. 
    \end{itemize}
\paragraph{Joint contribution from $\SB_{j-1}$, $\SB_j$, $\SB_{j+1}$ on $A_{j-1}^c\cap A_j\cap A_{j+1}\cap A_{j+2}^c$.}

We start by observing that the case $q=1,\ldots,N_j$, $s=0,\ldots,N_{j+1}-1$ in (III) simplifies:
\begin{itemize}
\item The contribution  $H( \mathbb{X}_{j,j+1}{(q+1:N_j+s)})$ can occur only if $\clusterlength_{j,j+1}\geqslant \dhinterseq$. 
\end{itemize}
The other cases of $q$ and $s$ require a precise analysis. 
\begin{itemize}
\item[{\rm (IV-a)}] The contribution (I-a), (I-b)  remains the same. Note further that for $\ell=1,\ldots,N_{j-1}$, $\Delta t_{j}(\ell)=\Delta t_{j,j+1}(\ell)$ and $H(\mathbb{X}_{j}{(1:\ell)})=H(\mathbb{X}_{j,j+1}(1:\ell))$. In total, the contribution from $\SB_{j-1}(H)$ in (I-a) and (I-b) can be written as
    \begin{align}\label{eq:contribution-to-mixed}
    \sum_{\ell=1}^{N_j-1}\Delta t_{j,j+1}(\ell)H(\mathbb{X}_{j,j+1}(1:\ell))\;. 
    \end{align}
\item[{\rm (IV-b)}] We analyse the contribution $H(\mathbb{X}_{j})=H(\mathbb{X}_{j,j+1}(1:N_j))$. This stems from both $\SB_{j-1}$ and $\SB_j$. We combine (I-c) with (III) for $q=s=0$. 
    \begin{itemize}
    \item (III-b) and (III-d) are duplicates in this case. If $t_{j}(1)\leqslant t_{j+1}(1)-\dhinterseq$, we have the contribution for 
        \begin{align*}
    \underbrace{t_{j}(1)-(j-1)\dhinterseq}_{\textrm{from (III)}}+\underbrace{(j\dhinterseq - t_{j}{(N_j)})}_{\textrm{from (I-c)}}=\dhinterseq-(t_{j}{(N_j)}-t_{j}(1))
    \end{align*}
    of the indices. 
    \item (III-c) and (III-e) are duplicates in this case.
    If $t_{j+1}(1)-\dhinterseq < t_{j}(1)$ we have the contribution for 
    \begin{align*}
    \underbrace{t_{j+1}(1)-j\dhinterseq}_{\textrm{from (III)}}+\underbrace{(j\dhinterseq - t_{j}{(N_j)})}_{\textrm{from (I-c)}}=t_{j+1}(1)-t_{j}{(N_j)}=\Delta t_{j,j+1}(N_j)
    \end{align*}
    of the indices.
    \item Thus, the contribution $H(\mathbb{X}_{j,j+1}(1:N_j))$ is for 
    \begin{align}\label{eq:contribution-to-mixed-1}
    \Delta t_{j,j+1}(N_j) \ind{t_{j+1}(1)-t_{j}(1)<\dhinterseq}+
    \left(\dhinterseq-(t_{j}{(N_j)}-t_{j}(1))\right)\ind{t_{j+1}(1)-t_{j}(1)\geqslant \dhinterseq}
    \end{align}
    of the indices. Note that the second part vanishes whenever the joint cluster length fulfills $\clusterlength_{j,j+1}< \dhinterseq$. 
    \end{itemize} 
\item[{\rm (IV-c)}] We take $q=0$ and $s=1,\ldots,N_{j+1}-1$ in (III). Then the contribution is $H(\mathbb{X}_{j,j+1}{(1:N_j+s)})$. It stems from $\SB_j$ only. Set $\ell=N_j+s$. With the re-numeration we have $\Delta t_{j+1}(s)=\Delta t_{j,j+1}(\ell)$. 
    \begin{itemize}
    \item For (III-b): If $t_{j+1}(s)-\dhinterseq< t_{j}(1)\leqslant t_{j+1}(s+1)-\dhinterseq$, then we have the contribution for $t_{j}(1)+\dhinterseq-t_{j+1}(s)\leqslant t_{j+1}(s+1)-t_{j+1}(s)=\Delta t_{j+1}(s)$ of the indices. 
    \item Next, for (III-c), if $t_{j+1}(s+1)-\dhinterseq< t_{j}(1)$, then we have the contribution for $t_{j+1}(s+1)-t_{j+1}(s)=\Delta t_{j+1}(s)$ of the indices. 
    \item Thus, $t_{j+1}(s+1)- t_{j}(1)\leqslant \dhinterseq$, the contribution occurs for at most  $\Delta t_{j+1}(s)$ of the indices. 
    \item The cases (III-d) and (III-e) are empty, since $t_{j+1}(s)-\dhinterseq\leqslant t_{j}(0)=(j-1)\dhinterseq$ cannot happen.  
    \item Thus, the contribution for this part is at most
    \begin{align}\label{eq:contribution-to-mixed-2}
    \sum_{\ell=N_j+1}^{N_{j}+N_{j+1}-1}\Delta t_{j,j+1}(\ell)H(\mathbb{X}_{j,j+1}(1:\ell))\;. 
    \end{align}
    \end{itemize}
\end{itemize}
Thus, whenever $\clusterlength_{j,j+1}< \dhinterseq$, \eqref{eq:contribution-to-mixed}, \eqref{eq:contribution-to-mixed-1} and \eqref{eq:contribution-to-mixed-2} give the joint contribution from $\SB_{j-1}$ and $\SB_j$:  
\begin{align}\label{eq:contribution-to-mixed-3}
    \sum_{i=1}^{N_{j}+N_{j+1}-1} \Delta t_{j,j+1}(i) 
    H( \mathbb{X}_{j,j+1}{(1:i)} )\;.
\end{align}

Similar analysis applies to (II) and (III) combined, and we obtain 
\begin{itemize}
\item[{\rm (IV-d)}] the joint contribution from $\SB_j$ and $\SB_{j+1}$:  
\begin{align}\label{eq:contribution-to-mixed-4}
    \sum_{i=1}^{N_{j}+N_{j+1}-1} \Delta t_{j,j+1}(i) 
    H( \mathbb{X}_{j,j+1}{(i+1:N_j+N_{j+1})} )\;.
\end{align}
\end{itemize}
The case $q=0$ and $s=N_j$ needs a special care. 
\begin{itemize}
\item[{\rm (IV-e)}] We take $q=0$ and $s=N_{j+1}$ in (III). Then the contribution is $H(\mathbb{X}_{j,j+1})$. It stems from $\SB_j$ only. 
    \begin{itemize}
    \item Start with (III-b). If $t_{j+1}{(N_{j+1})}-\dhinterseq< t_{j}(1)$ then we have the contribution for $\dhinterseq-(t_{j+1}{(N_{j+1})}-t_{j}(1))$ of the indices. 
    \item The cases (III-c), (III-d) and (III-e) are empty. 
    \item Hence, the contribution is 
    \begin{align}\label{eq:contribution-to-mixed-5}
    \left(\dhinterseq-\left(t_{j+1}{(N_{j+1})}-t_{j}(1)\right)\right)_+
    H( \mathbb{X}_{j,j+1} )\;.
    \end{align}
    \end{itemize}
\end{itemize}
We combine \eqref{eq:contribution-to-mixed-5} with $\DB_{j,j+1}=\dhinterseq H( \mathbb{X}_{j,j+1} )$ to get 
\begin{align}\label{eq:contribution-to-mixed-6}
&\left(\dhinterseq-\left(t_{j+1}{(N_{j+1})}-t_{j}(1)\right)\right)\ind{t_{j+1}{(N_{j+1})}-t_{j}(1)<\dhinterseq}
    H( \mathbb{X}_{j,j+1} )-\DB_{j,j+1}\nonumber\\
    &= -\left(t_{j+1}{(N_{j+1})}-t_{j}(1)\right)\ind{t_{j+1}{(N_{j+1})}-t_{j}(1)<\dhinterseq}H( \mathbb{X}_{j,j+1} )\\
    &\phantom{=}-\dhinterseq \ind{t_{j+1}{(N_{j+1})}-t_{j}(1)\geqslant \dhinterseq}H( \mathbb{X}_{j,j+1} )\;. \nonumber
\end{align}
Now, with the re-numeration for the adjacent blocks:
\begin{align*}
t_{j+1}{(N_{j+1})}-t_{j}(1)=t_{j,j+1}{(N_j+N_{j+1})}-t_{j,j+1}(1)=\sum_{i=1}^{N_j+N_{j+1}-1}\Delta t_{j,j+1}(i)\;. 
\end{align*}
Combining this with \eqref{eq:contribution-to-mixed-3}-\eqref{eq:contribution-to-mixed-6}
we have 
\begin{align*}
&\big[  \SB_{j-1} + \SB_j + \SB_{j+1} - \DB_{j,j+1}  \big]\ind{A_{j-1}^c\cap A_j\cap A_{j+1}\cap A_{j+2}^c}\\
&= 
\sum_{i=1}^{N_{j}+N_{j+1}-1} \Delta t_{j,j+1}(i) \left\{
    H( \mathbb{X}_{j,j+1}{(1:i)} ) + H( \mathbb{X}_{j,j+1}{(i+1:N_{j}+N_{j+1})} ) 
    - H( \mathbb{X}_{j,j+1} )
    \right\}
\end{align*}
whenever $\clusterlength_{j,j+1}<\dhinterseq$.

\subsection{Extremal index case}
\paragraph{Explanation for \Cref{xmpl:extremal-index-1}.}
\begin{itemize}
\item
On $A_{j-1}^c\cap A_j$:  for $i\in I_{j-1}$, we have
$H(\bsX_{i+1,i+\dhinterseq}/\tepseq)=1$ if and only if $i+\dhinterseq-1\geqslant t_{j}(1)$. Hence, $\SB_{j-1}(H)=(j-1)\dhinterseq-(t_{j}(1)-\dhinterseq)$.
\item
On $A_j\cap A_{j+1}^c$: for $i\in I_{j}$, we have
$H(\bsX_{i+1,i+\dhinterseq}/\tepseq)=1$ if and only if $i\leqslant t_{j}{(N_j)}$. Hence,  $\SB_{j}(H)=t_{j}{(N_j)}-(j-1)\dhinterseq$.
\item On $A_j$, $\DB_j(H)=\dhinterseq$.
\end{itemize}
Hence, on $A_{j-1}^c\cap A_j\cap A_{j+1}^c$, we have
\begin{align*}
\SB_{j-1}(H)+\SB_{j}(H)-\DB_j(H)=t_{j}{(N_j)}-t_{j}(1)\;.
\end{align*}
\paragraph{Explanation for \Cref{xmpl:extremal-index-2}.}
\begin{itemize}
\item
On $A_{j-1}^c\cap A_j$:  for $i\in I_{j-1}$, we have
$H(\tepseq^{-1}\bsX_{i+1,i+\dhinterseq})=1$ if and only if $i+\dhinterseq-1\geqslant t_{j}(1)$. Hence, $\SB_{j-1}(H)=(j-1)\dhinterseq-(t_{j}(1)-\dhinterseq)=j\dhinterseq-t_{j}(1)$.
\item
On $A_{j+1}\cap A_{j+2}^c$: for $i\in I_{j+1}$, we have
$H(\tepseq^{-1}\bsX_{i+1,i+\dhinterseq})=1$ if and only if $i\leqslant t_{j+1}{(N_{j+1})}$. Hence,  $\SB_{j+1}(H)=t_{j+1}{(N_{j+1})}-j\dhinterseq$.
\item On $A_j\cap A_{j+1}$, for $i\in I_j$, $H(\tepseq^{-1}\bsX_{i,i+\dhinterseq-1})\equiv 0$ if and only if $i>t_{j}{(N_j)}$ and $i+\dhinterseq-1<t_{j+1}(1)$. Thus,
    $\SB_j(H)=\dhinterseq- (t_{j+1}(1)-t_{j}{(N_j)}-\dhinterseq)_+$.
\end{itemize}
Thus, on $\{A_{j-1}^c\cap A_j\cap A_{j+1}\cap A_{j+2}^c\}$
\begin{align*}
 &\SB_{j-1} + \SB_j + \SB_{j+1} - \DB_{j,j+1}=\left[\left(t_{j+1}{(N_{j+1})}-t_{j}(1)\right)-\left(t_{j+1}(1)-t_{j}{(N_j)}-\dhinterseq\right)_+\right]  
 \;.
 \end{align*}

\section*{Acknowledgement}
The authors would like to thank Stanislav Volgushev for indicating the expansion problem. The authors are grateful to Philippe Soulier for an extensive discussion.  
Both authors were supported by an NSERC grant. A part of this research was conducted during the second author's stay at the Sydney Mathematical Research Institute (October - November 2022). Rafa{\l} Kulik thanks for the support and hospitality provided by SMRI.

\section{Bibliography}


\end{document}